\documentclass[12pt,fleqn]{article}

\usepackage[normalem]{ulem}

\usepackage{amsmath,amssymb,amsthm,xcolor}
\usepackage{geometry} 
\usepackage{hyperref}
\usepackage[pagewise]{lineno}

\newcommand{\bdry}[1]{\partial #1}
\newcommand{\A}{{\cal A}}
\newcommand{\B}{{\cal B}}
\newcommand{\C}{{\cal C}}
\newcommand{\D}{{\cal D}}
\newcommand{\Da}[1]{\D_{\!a} #1}
\newcommand{\DA}[1]{\D_{\!\A} #1}
\newcommand{\dint}{\ds{\int}}
\newcommand{\dist}[2]{\text{dist}\, (#1,#2)}
\newcommand{\ds}[1]{\displaystyle #1}
\newcommand{\E}{{\cal E}}
\newcommand{\EA}[1]{\if #1'' \E_{\!\A} \else \E_{\!\A}(#1) \fi}
\newcommand{\eps}{\varepsilon}
\newcommand{\norm}[2][]{\left\|#2\right\|_{#1}}
\renewcommand{\o}{\text{o}}
\renewcommand{\O}{\text{O}}
\newcommand{\PS}[1]{$(\text{PS})_{#1}$}
\newcommand{\pnorm}[2][]{\if #1'' \left|#2\right|_p \else \left|#2\right|_{#1} \fi}
\newcommand{\R}{\mathbb R}
\newcommand{\seq}[1]{\left(#1\right)}
\newcommand{\set}[1]{\left\{#1\right\}}
\newcommand{\vol}[1]{\left|#1\right|}
\newcommand{\wto}{\rightharpoonup}

\def\ocirc#1{\ifmmode\setbox0=\hbox{$#1$}\dimen0=\ht0
\advance\dimen0 by1pt\rlap{\hbox to\wd0{\hss\raise\dimen0
\hbox{\hskip.2em$\scriptscriptstyle\circ$}\hss}}#1\else
{\accent"17 #1}\fi}

\DeclareMathOperator{\divg}{div}

\DeclareMathOperator*{\esssup}{ess\, sup}
\DeclareMathOperator{\supp}{supp}

\newenvironment{enumroman}{\begin{enumerate}
\renewcommand{\theenumi}{$(\roman{enumi})$}
\renewcommand{\labelenumi}{$(\roman{enumi})$}}{\end{enumerate}}

\newenvironment{properties}[1]{\begin{enumerate}
\renewcommand{\theenumi}{$(#1_\arabic{enumi})$}
\renewcommand{\labelenumi}{$(#1_\arabic{enumi})$}}{\end{enumerate}}

\newtheorem{corollary}{Corollary}[section]
\newtheorem{lemma}[corollary]{Lemma}
\newtheorem{proposition}[corollary]{Proposition}
\newtheorem{theorem}[corollary]{Theorem}

\theoremstyle{definition}
\newtheorem{definition}[corollary]{Definition}

\theoremstyle{remark}
\newtheorem{remark}[corollary]{Remark}

\numberwithin{equation}{section}

\title{\bf Critical growth double phase problems:\\
the local case and a Kirchhoff type case\thanks{F.C. was partially supported by the INdAM-GNAMPA. 
F.C. also acknowledges the support of the Department of Mathematical Sciences, Florida Institute of Technology during her visits to Melbourne, where parts of this work have been achieved.
\newline \indent\; {\em MSC2010:} Primary 35J92, Secondary 35B33
\newline \indent\; {\em Key Words and Phrases:} Double phase operator; Brezis-Nirenberg type critical problems; Musielak-Orlicz Sobolev spaces; variational methods; Nonlocal Kirchhoff problems.}}
\author{\bf Francesca Colasuonno\\
Dipartimento di Matematica\\
Alma Mater Studiorum Universit\`{a} di Bologna\\
Piazza di Porta San Donato 5, 40126 Bologna, Italy\\
\em francesca.colasuonno@unibo.it\\
[\bigskipamount]
\bf Kanishka Perera\\
Department of Mathematics\\
Florida Institute of Technology\\
150 W University Blvd, Melbourne, FL 32901, USA\\
\em kperera@fit.edu}
\date{}

\begin{document}

\maketitle

\begin{abstract}
We study Brezis-Nirenberg type problems, governed by the double phase operator $- \divg \left(|\nabla u|^{p-2}\, \nabla u + a(x)\, |\nabla u|^{q-2}\, \nabla u\right)$, that involve a critical nonlinearity of the form $|u|^{p^\ast - 2}\, u + b(x)\, |u|^{q^\ast - 2}\, u$. Both for the local case and for related nonlocal Kirchhoff type problems, we prove new compactness and existence results using variational methods in suitable Musielak-Orlicz Sobolev spaces. For these functional spaces, we prove some continuous and compact embeddings that are of independent interest. The study of the local problem is complemented by some nonexistence results of Poho\v{z}aev type. 
\end{abstract}
\tableofcontents

\section{Introduction}
We consider the double phase problem
\begin{equation} \label{1}
\left\{\begin{aligned}
- \divg \left(|\nabla u|^{p-2}\, \nabla u + a(x)\, |\nabla u|^{q-2}\, \nabla u\right) & = f(x,u) && \text{in } \Omega\\[10pt]
u & = 0 && \text{on } \bdry{\Omega},
\end{aligned}\right.
\end{equation}
where $\Omega \subset \R^N,\, N \ge 2$ is a bounded domain, the exponents $p$ and $q$ satisfy
\begin{equation}
\label{eq:pq}
1 < p < q < N\quad \mbox{with}\quad \frac{q}{p} \le \frac{N}{N-1},
\end{equation}
which implies $q < p^*$, $a \in L^\infty(\Omega)\cap C^{0,\frac{N}{p}(q-p)}(\Omega)$ is a nonnegative weight, and $f$ is a Carath\'{e}odory function on $\Omega \times \R$. We introduce the following generalized Young function
\[
\A(x,t) := t^p + a(x)\, t^q, \quad (x,t) \in \Omega \times [0,\infty).
\]
The double phase operator
\[
\DA{u} := \divg \left(|\nabla u|^{p-2}\, \nabla u + a(x)\, |\nabla u|^{q-2}\, \nabla u\right)
\]
and the associated energy
\[
\EA{u} := \int_\Omega \left(\frac{1}{p}\, |\nabla u|^p + \frac{a(x)}{q}\, |\nabla u|^q\right) dx
\]
appear in the study of strongly anisotropic materials, where the weight $a(x)$ drives the geometry of a composite of two different materials with hardening powers $p$ and $q$ (see \cite{MR1329546,MR864171,MR1350506
}). The energy $\mathcal E_\A$ belongs to the class of the integral functionals with {\it nonstandard growth condition}, according to a terminology introduced by Marcellini in \cite{MR969900,MR1094446}. A regularity theory for  minimizers of such energy functionals was recently developed by Mingione et al., see \cite{MR3348922,
MR3298012,MR3294408,MR3360738}.
The functional $\mathcal E_\A$ is naturally defined on the Musielak-Orlicz Sobolev space $W^{1,\,\A}_0(\Omega)$ associated with $\A$. We denote by
\[
\rho_\A(u) := \int_\Omega \A(x,|u|)\, dx
\]
the $\A$-modular and by
\[
\norm{u} := \inf \set{\tau > 0 : \rho_\A\bigg(\frac{\nabla u}{\tau}\bigg) \le 1}
\]
the standard norm in $W^{1,\,\A}_0(\Omega)$. 

When $f$ has subcritical growth, existence and multiplicity of solutions to this problem have been widely studied in the recent literature (see, e.g., \cite{MR4045487,
MR4339038,MR3730751}). Our goal here is to obtain nontrivial solutions of problem \eqref{1} when the nonlinearity $f$ has critical growth.

Critical growth for such spaces was not completely understood until the very recent paper \cite{CianchiDiening}, where Cianchi and Diening 
generalize to Musielak-Orlicz spaces the results obtained in \cite{MR1406683} for Orlicz spaces: they prove the optimal Sobolev embedding for Musielak-Orlicz Sobolev spaces under minimal assumptions. For  $W^{1,\,\A}_0(\Omega)$, the assumptions required on $a$ and $q/p$ are those introduced above. In particular, they show that the sharp Sobolev conjugate of $\A$ is equivalent to $t^{p^*}+a(x)^{q^*/q}t^{q^*}$ when $N>q$, cf. \cite[Examples 2.4 and 3.11, and Theorem 3.7]{CianchiDiening}, see also \cite[Proposition 3.4]{HoWinkert}. 
This leads us to consider critical nonlinearities of the form
\begin{equation}\label{eq:f}
f(x,u) = \mu|u|^{p^\ast - 2}\, u + b(x)\, |u|^{q^\ast - 2}\, u + c(x)\, |u|^{s-2}\, u,
\end{equation}
where $1 < s < q^\ast$, $\mu\ge 0$ is a parameter, $b \in C(\bar\Omega)$ is a nonnegative function satisfying
\begin{equation} \label{2}
a_0 := \inf_{x \in \supp(b)}\, a(x) > 0 \quad\mbox{(possibly $a_0=+\infty$ when $b\equiv 0$),}
\end{equation}
and $c \in L^\infty(\Omega)$ is a nonnegative function that satisfies the additional assumption
\begin{equation}\label{eq:hp-on-c}
c(x)\le C\, a(x)^{s/q} \quad \text{for a.a.\! } x \in \Omega
\end{equation}
for some constant $C>0$ when $s \ge p^\ast$. The critical term $\mu |u|^{p^\ast - 2}\, u + b(x)\, |u|^{q^\ast - 2}\, u$ in the nonlinearity $f$ reflects the structure of the sharp conjugate of $\A$.
The need for assumption \eqref{2} in this paper is essentially technical, we use that $a$ is positive on the support of $b$ to reduce ourselves to the classical Sobolev spaces for which many inequalities are available for modulars and the explicit expression of the extremal functions for the Sobolev inequalities in $\mathbb R^N$ are known. More precisely, \eqref{2} is needed to prove the estimate \eqref{1000}, which in turn is crucial to showing that the weak solution of problem \eqref{1} to which a (PS)$_{\beta}$ sequence converges weakly is nonzero (see the proof of Proposition \ref{Proposition 3}). 
The term of growth $|u|^{s-1}$ is a subcritical perturbation of the critical part of the nonlinearity, and condition \eqref{eq:hp-on-c} is needed to prove its subcriticality when $s\ge p^\ast$ (see Proposition \ref{prop:subcritical}).

Unlike in the classical Brezis-Nirenberg problem, a threshold for compactness for the energy functional associated with problem \eqref{1} cannot be found in a closed form. Moreover, there is no closed form formula for the maximum energy on a ray starting from the origin, which makes it rather delicate to show that the mountain pass level is below the compactness threshold. These new difficulties that problem \eqref{1} presents are due to the inhomogeneity of the double phase operator.
\medskip

In the second part of the paper, we will study also a nonlocal version of \eqref{1} involving a Kirchhoff type term in front of the leading operator, namely
\begin{equation} \label{2K}
\left\{\begin{aligned}
- h(\EA{u})\, \divg \left(|\nabla u|^{p-2}\, \nabla u + a(x)\, |\nabla u|^{q-2}\, \nabla u\right)& = f(x,u)
&& \text{in } \Omega\\[10pt]
u & = 0 && \text{on } \bdry{\Omega},
\end{aligned}\right.
\end{equation}
where $h : [0,\infty) \to [0,\infty)$ is the so-called Kirchhoff function and is assumed to be continuous and nondecreasing, 
and $f$ is as in \eqref{eq:f}-\eqref{eq:hp-on-c}.

In \cite{k}, Kirchhoff introduced a factor of the form $p_0+\frac{\mathcal E h}{2L} \int_0^L u_x^2$ in front of the term $u_{xx}$ in the one dimensional d'Alembert's wave equation, in order to take into account the effects of the changes in the length of the string produced by transversal vibrations. Similarly, nonlocal terms depending on the first derivatives of the unknown function in front of the leading operator in second order elliptic problems are usually referred to as Kirchhoff terms and the prototypical shape of the Kirchhoff function is of polynomial type: $h(t) = \alpha_1 + \alpha_2\, t^{\theta - 1}$, with $\alpha_1,\, \alpha_2 \ge 0$. Kirchhoff problems received much attention after the work of Lions \cite{Lions}, where a functional analysis framework was introduced for them.
These problems are said to be degenerate if the Kirchhoff function vanishes at zero, otherwise they are called nondegenerate. It is worth noting that the compactness issues are more delicate in the degenerate case, see for instance \cite{cp,fp}. Our results for problem \eqref{2K} cover also the degenerate case. We refer to \cite{afmw} for a different double phase Kirchhoff problem, with critical nonlinearity of type $|u|^{p^*-2}u$. 

As for the local case, also the proofs of existence results of \eqref{2K} are in the spirit of the celebrated paper by Brezis and Nirenberg \cite{BN}. They rely on the mountain pass theorem and the major difficulty is the lack of compactness arising from the presence of the critical term in the problem. We prove that the energy functional satisfies the Palais-Smale condition for levels below a certain threshold and we need to show that the mountain pass level is below such a threshold. The nonlocality and the inhomogeneity of the Kirchhoff double phase operator make the energy estimates rather delicate: beyond the difficulties arising from the inhomogeneity of the double phase operator, in this nonlocal setting the Kirchhoff function applies to $\mathcal E_A$, thus combining together the $L^p$ norm and the $L^q_a$ seminorm of the gradient; this makes even more complicated the use of results known for classical Sobolev spaces as described for the local problem.
\smallskip

The paper is organized as follows: we devote the next two sections to the precise statements of the main results for both the local and the nonlocal cases. In Section \ref{sec2}, we recall some notation, definitions and preliminary results for Musielak-Orlicz spaces, and introduce the functional setting for double phase problems. We also prove some useful embedding results and derive explicitly some properties of the Kirchhoff function $h$. Sections \ref{sec3} to \ref{sec6} are devoted to the study of the local problem \eqref{1}: in Section \ref{sec3} we prove the compactness results and the asymptotic estimates of the threshold energy level for compactness, 
in Section \ref{sec4} we prove the existence theorems,  
and in Section \ref{sec6} we prove a Poho\v{z}aev type identity 
and nonexistence results. 
Finally, the last two sections are devoted to the study of the Kirchhoff problem \eqref{2K}: in Section \ref{sec7} we prove the compactness results and in Section~\ref{sec8} we prove the existence results.
\smallskip 

Without further mentioning, {\it throughout the paper} we assume that $\Omega \subset \R^N$ is a smooth bounded domain, that $p$ and $q$ satisfy \eqref{eq:pq}, that $0 \le a \in L^\infty(\Omega)\cap C^{0,\frac{N}{p}(q-p)}(\Omega)$, $b\in C(\bar\Omega)$, and that
$h:[0,\infty)\to[0,\infty)$ is continuous and nondecreasing. 

We observe that the conditions on the regularity of $\Omega$ and $b$ can be weakened in several results, but they are essential in estimating the threshold energy level for compactness (see the proof of Proposition \ref{prop:K-add}, useful e.g. for Proposition \ref{Proposition 3}). So, for the sake of simplicity, unless otherwise stated, we always assume $b$ to be continuous up to the boundary of $\Omega$ and $\Omega$ to be smooth enough to ensure the validity of the classical Sobolev embeddings. 

Throughout the paper, we write
\[
\pnorm{u} = \left(\int_\Omega |u|^p\, dx\right)^{1/p}, \qquad \pnorm[q,a]{u} = \left(\int_\Omega a(x)\, |u|^q\, dx\right)^{1/q}
\]
for the norm in the Lebesgue space $L^p(\Omega)$ and the seminorm in the weighted Lebesgue space $L^q_a(\Omega)$, respectively, and $\pnorm[\infty]{\cdot}$ for the norm in $L^\infty(\Omega)$. We also write $p^\ast := Np/(N - p)$ and $q^\ast := Nq/(N - q)$ for the critical Sobolev exponents.
Moreover, we denote simply by $C_S$ all constants arising from different Sobolev embeddings, when their value is not important for the proofs.
Finally, we use the notation $X\hookrightarrow Y$ and $X\hookrightarrow\hookrightarrow Y$ for continuous and compact embeddings, respectively.

\section{Statements of the results for the local case}
In this section, we state precisely our results both for the local case.

\subsection{Existence results}

First we consider the problem
\begin{equation} \label{6}
\left\{\begin{aligned}
-\Da{u} & = \lambda\, |u|^{r-2}\, u + \mu\, |u|^{p^\ast - 2}\, u + b(x)\, |u|^{q^\ast - 2}\, u && \text{in } \Omega\\[10pt]
u & = 0 && \text{on } \bdry{\Omega},
\end{aligned}\right.
\end{equation}
where $p \le r < p^\ast$ and $\lambda, \mu > 0$ are parameters and $b\in C(\bar\Omega)$ is a nonnegative function satisfying \eqref{2}. A weak solution of this problem is a function $u \in W^{1,\,\A}_0(\Omega)$ satisfying
\[
\int_\Omega \left(|\nabla u|^{p-2} + a(x)\, |\nabla u|^{q-2}\right) \nabla u \cdot \nabla v\, dx = \int_\Omega \left(\lambda\, |u|^{r-2} + \mu\, |u|^{p^\ast - 2} + b(x)\, |u|^{q^\ast - 2}\right) uv\, dx
\]
for all $v \in W^{1,\,\A}_0(\Omega)$. Let
\begin{equation} \label{7}
\lambda_1(p) = \inf_{u \in W^{1,\,p}_0(\Omega) \setminus \set{0}}\, \frac{\dint_\Omega |\nabla u|^p\, dx}{\dint_\Omega |u|^p\, dx}
\end{equation}
be the first Dirichlet eigenvalue of the $p$-Laplacian in $\Omega$. We have the following theorem.

\begin{theorem} \label{Theorem 1}
Assume that \eqref{2} holds and there is a ball $B_\rho(x_0) \subset \Omega$ such that
\begin{equation} \label{33}
a(x) = 0 \quad \text{for all } x \in B_\rho(x_0).
\end{equation}
Then there exists $b^\ast > 0$ such that problem \eqref{6} has a nontrivial weak solution in $W^{1,\,\A}_0(\Omega)$ when $\mu > 0$ and $b_\infty < b^\ast$, where $b_\infty := \pnorm[\infty]{b}$, in each of the following cases:
\begin{enumroman}
\item $N \ge p^2$, $r = p$, and $0 < \lambda < \lambda_1(p)$,
\item $N \ge p^2$, $p < r < p^\ast$, and $\lambda > 0$,
\item $N < p^2$, $(Np - 2N + p)\, p/(N - p)(p - 1) < r < p^\ast$, and $\lambda > 0$.
\end{enumroman}
\end{theorem}

We remark that $(Np - 2N + p)\, p/(N - p)(p - 1) > p$ if and only if $N < p^2$, so the assumption required on $r$ in $(iii)$ is more restrictive than $r > p$. In particular, in the low dimensional case $p < N < p^2$, we do not have the existence of a nontrivial solution when $r$ is close to $p$.

For $b(x) \equiv 0$ and $\mu = 1$, Theorem \ref{Theorem 1} gives the following Brezis-Nirenberg type result for the double phase operator.

\begin{corollary} \label{Corollary 1}
Assume \eqref{2} and \eqref{33}. Then the problem
\begin{equation} \label{100}
\left\{\begin{aligned}
-\Da{u} & = \lambda\, |u|^{r-2}\, u + |u|^{p^\ast - 2}\, u && \text{in } \Omega\\[10pt]
u & = 0 && \text{on } \bdry{\Omega}
\end{aligned}\right.
\end{equation}
has a nontrivial weak solution in $W^{1,\,\A}_0(\Omega)$ in each of the following cases:
\begin{enumroman}
\item $N \ge p^2$, $r = p$, and $0 < \lambda < \lambda_1(p)$,
\item $N \ge p^2$, $p < r < p^\ast$, and $\lambda > 0$,
\item $N < p^2$, $(Np - 2N + p)\, p/(N - p)(p - 1) < r < p^\ast$, and $\lambda > 0$.
\end{enumroman}
\end{corollary}

\begin{remark}
The case $1 < r < p$ of problem \eqref{100} was considered in \cite{MR4441456}, where it was shown that there exist infinitely many negative energy solutions for all sufficiently small $\lambda > 0$.
\end{remark}

\begin{remark}
The existence of a nontrivial weak solution of the critical $p$-Laplacian problem
\begin{equation} \label{101}
\left\{\begin{aligned}
- \Delta_p u & = \lambda\, |u|^{r-2}\, u + |u|^{p^\ast - 2}\, u && \text{in } \Omega\\[10pt]
u & = 0 && \text{on } \bdry{\Omega}
\end{aligned}\right.
\end{equation}
in cases $(i)$-$(iii)$ has been proved in \cite{AzoreroPeral,ArioliGazzola}.
\end{remark}

\begin{remark}
The existence of a positive solution of problem \eqref{101} when $p < N < p^2$ and $r = p$ was recently proved in \cite{AngeloniEsposito}. It would be interesting to investigate possible extensions of this result to the double operator.
\end{remark}

Next we consider the problem
\begin{equation} \label{8}
\left\{\begin{aligned}
-\Da{u} & = c(x)\, |u|^{s-2}\, u + \mu\, |u|^{p^\ast - 2}\, u + b(x)\, |u|^{q^\ast - 2}\, u && \text{in } \Omega\\[10pt]
u & = 0 && \text{on } \bdry{\Omega},
\end{aligned}\right.
\end{equation}
where $p^\ast \le s < q^\ast$,  $\mu \ge 0$ is a parameter, $b\in C(\bar\Omega)\setminus\{0\}$ and $c \in L^\infty(\Omega)$ are nonnegative functions satisfying \eqref{2} and \eqref{eq:hp-on-c}, respectively. 

We note that since $b\not\equiv 0$, in the next theorem and its corollary, $b_\infty>0$. 

\begin{theorem} \label{Theorem 2}
Assume that \eqref{2} and \eqref{eq:hp-on-c} hold and there is a ball $B_\rho(x_0) \subset \Omega$ such that
\begin{equation} \label{9}
a(x) = a_0, \quad b(x) = b_\infty>0, \quad c(x) \ge c_0 \quad \text{for a.a.\! } x \in B_\rho(x_0)
\end{equation}
for some constant $c_0 > 0$. Then there exists $\mu^\ast > 0$ such that problem \eqref{8} has a nontrivial weak solution in $W^{1,\,\A}_0(\Omega)$ when $0 \le \mu < \mu^\ast$ in each of the following cases:
\begin{enumroman}
\item $1 < p < N(q - 1)/(N - 1)$ and $N^2 (q - 1)/(N - 1)(N - q) < s < q^\ast$,
\item $N(q - 1)/(N - 1) \le p < q$ and $Np/(N - q) < s < q^\ast$.
\end{enumroman}
\end{theorem}

In particular, Theorem \ref{Theorem 2} gives the following corollary for $\mu = 0$.

\begin{corollary}
Assume \eqref{2}, \eqref{eq:hp-on-c}, and \eqref{9}. Then the problem
\[
\left\{\begin{aligned}
-\Da{u} & = c(x)\, |u|^{s-2}\, u + b(x)\, |u|^{q^\ast - 2}\, u && \text{in } \Omega\\[10pt]
u & = 0 && \text{on } \bdry{\Omega}
\end{aligned}\right.
\]
has a nontrivial weak solution in $W^{1,\,\A}_0(\Omega)$ in each of the following cases:
\begin{enumroman}
\item $1 < p < N(q - 1)/(N - 1)$ and $N^2 (q - 1)/(N - 1)(N - q) < s < q^\ast$,
\item $N(q - 1)/(N - 1) \le p < q$ and $Np/(N - q) < s < q^\ast$.
\end{enumroman}
\end{corollary}

\subsection{Compactness results}

Proofs of Theorem \ref{Theorem 1} and Theorem \ref{Theorem 2} will be based on a new compactness result that we will prove for the general critical growth double phase problem
\begin{equation} \label{10}
\left\{\begin{aligned}
-\Da{u} & = \mu\, |u|^{p^\ast - 2}\, u + b(x)\, |u|^{q^\ast - 2}\, u + g(x,u) && \text{in } \Omega\\[10pt]
u & = 0 && \text{on } \bdry{\Omega},
\end{aligned}\right.
\end{equation}
where $g$ is a Carath\'{e}odory function on $\Omega \times \R$ satisfying the subcritical growth condition
\begin{equation} \label{11}
|g(x,t)| \le c_1 + c_2\, |t|^{r-1} + c(x)\, |t|^{s-1} \quad \text{for a.a.\! } x \in \Omega \text{ and all } t \in \R
\end{equation}
for some constants $c_1, c_2 > 0$, $1 < r < p^\ast \le s < q^\ast$, nonnegative functions $b \in C(\bar\Omega)$ and $c \in L^\infty(\Omega)$ satisfying \eqref{2} and \eqref{eq:hp-on-c}, respectively. The variational functional associated with this problem is
\[
E(u) := \int_\Omega \left[\frac{1}{p}\, |\nabla u|^p + \frac{a(x)}{q}\, |\nabla u|^q - \frac{\mu}{p^\ast}\, |u|^{p^\ast} - \frac{b(x)}{q^\ast}\, |u|^{q^\ast} - G(x,u)\right] dx, \quad u \in W^{1,\,\A}_0(\Omega),
\]
where $G(x,t) = \int_0^t g(x,\tau)\, d\tau$. We recall that $\seq{u_j} \subset W^{1,\,\A}_0(\Omega)$ is a \PS{\beta} sequence for $E$ if $E(u_j) \to \beta$ and $E'(u_j) \to 0$. To ensure that \PS{\beta} sequences are bounded, we assume that
\begin{equation} \label{19}
G(x,t) - \frac{t}{\sigma}\, g(x,t) \le \mu \left(\frac{1}{\sigma} - \frac{1}{p^\ast}\right) |t|^{p^\ast} + c_3 \quad \text{for a.a.\! } x \in \Omega \text{ and all } t \in \R
\end{equation}
for some constants $c_3 > 0$ and $q < \sigma < p^\ast$. This technical assumption may be replaced with the condition
\[
\frac{c^{q^\ast/(q^\ast - s)}}{b^{s/(q^\ast - s)}} \in L^1(\Omega).
\]
However, we prefer not to impose this condition since the nonlinearities in problems \eqref{6} and \eqref{8} satisfy \eqref{19} without further assumptions on $c$.

We will show that there exists a threshold level for $\beta$ below which every \PS{\beta} sequence has a subsequence that converges weakly to a nontrivial weak solution of problem \eqref{10}. Although we do not have a closed form formula for this threshold, we can characterize it variationally as follows. Let
\begin{equation} \label{26}
I(X,Y,Z,W) := \frac{1}{p}\, X + \frac{1}{q}\, Y - \frac{1}{p^\ast}\, Z - \frac{1}{q^\ast}\, W, \quad X, Y, Z, W \ge 0
\end{equation}
and let
\begin{equation} \label{5}
S_p := \inf_{u \in \D^{1,\,p}(\R^N) \setminus \set{0}}\, \frac{\dint_{\R^N} |\nabla u|^p\, dx}{\left(\dint_{\R^N} |u|^{p^\ast} dx\right)^{p/p^\ast}}, \qquad S_q := \inf_{u \in \D^{1,\,q}(\R^N) \setminus \set{0}}\, \frac{\dint_{\R^N} |\nabla u|^q\, dx}{\left(\dint_{\R^N} |u|^{q^\ast} dx\right)^{q/q^\ast}}
\end{equation}
be the best Sobolev constants for $\D^{1,\,m}(\R^N)=\{u\in L^{m^\ast}(\R^N)\,:\,|\nabla u|\in L^m(\R^N)\}$, with $m=p,\,q$. Denote by $S(\mu,b_\infty)$ the set of points $(X,Y,Z,W) \in \R^4$ with $X, Y, Z, W \ge 0$ satisfying
\begin{gather}
\label{21} I(X,Y,Z,W) > 0,\\[10pt]
\label{22} X + Y = Z + W,\\[10pt]
\label{23} Z \le \frac{\mu}{S_p^{p^\ast/p}}\, X^{p^\ast/p}, \qquad W \le \frac{b_\infty}{(a_0\, S_q)^{q^\ast/q}}\, Y^{q^\ast/q},
\end{gather}
and set
\[
\beta^\ast(\mu,b_\infty) := \inf_{(X,Y,Z,W) \in S(\mu,b_\infty)}\, I(X,Y,Z,W).
\]

\begin{proposition} \label{Proposition 3}
Assume \eqref{2}, \eqref{eq:hp-on-c}, \eqref{11}, and \eqref{19}. If
\[
0 < \beta < \beta^\ast(\mu,b_\infty),
\]
then every {\em \PS{\beta}} sequence for $E$ has a subsequence that converges weakly to a nontrivial weak solution of problem \eqref{10}.
\end{proposition}

We will prove Theorem \ref{Theorem 1} and Theorem \ref{Theorem 2} by combining this compactness result with the following asymptotic estimates for $\beta^\ast(\mu,b_\infty)$.

\begin{proposition} \label{Proposition 4}
For $\mu > 0$,
\begin{equation} \label{15}
\beta^\ast(\mu,b_\infty) \ge \frac{1}{N}\, \frac{S_p^{N/p}}{\mu^{(N-p)/p}} + \o(1) \quad \text{as } b_\infty \to 0.
\end{equation}
For $b_\infty > 0$,
\begin{equation} \label{16}
\beta^\ast(\mu,b_\infty) \ge \frac{1}{N}\, \frac{(a_0\, S_q)^{N/q}}{b_\infty^{(N-q)/q}} + \o(1) \quad \text{as } \mu \to 0.
\end{equation}
\end{proposition}

\begin{remark}
In Theorem \ref{Theorem 1}, $b_\infty$ is small and $\mu>0$ is arbitrary, while in Theorem~\ref{Theorem 2}, $\mu$ is small and $b_\infty$ is arbitrary. Results when both $\mu$ and $b_\infty$ are small can be found in \cite[Theorems 2.5 and 2.6]{HaHo}. We note that such results are easier to obtain as Proposition~\ref{Proposition 4} implies that $\beta^\ast(\mu,b_\infty)\to +\infty$ as $\mu, \,b_\infty\to 0$.
\end{remark}

\subsection{Nonexistence results}

We also prove a Poho\v{z}aev type identity for problem \eqref{1}.
\begin{theorem}\label{thm:poho}
Let $\Omega\subset\mathbb R^N$ be a bounded $C^1$-domain, \eqref{eq:pq} hold, and $0\le a \in C^1(\Omega)$. If $u\in W_0^{1,A}(\Omega)\cap W^{2,\A}(\Omega)$ is a weak solution of \eqref{1}, then we have the identity
\begin{equation}\label{eq:poho}
\begin{aligned}
\left(\frac{1}{p}-\frac{1}{q}\right)\int_\Omega|\nabla u|^p\, dx &+\frac{1}{Nq}\int_\Omega|\nabla u|^q(\nabla a\cdot x)\,dx\\
&+\frac{1}{N}\int_{\partial\Omega}\left[\left(1-\frac{1}{p}\right)|\partial_\nu u|^p+\left(1-\frac{1}{q}\right)a(x)|\partial_\nu u|^q\right](x\cdot\nu)\,d\sigma\\
&\hspace{5cm} = \int_\Omega\left[F(x,u)-\frac{1}{q^\ast}f(x,u)u\right]dx,
\end{aligned}
\end{equation}
where $F(x,t):=\int_0^t f(x,\tau)d\tau$, and $\nu$ is the outward unit normal to $\partial\Omega$.
\end{theorem}

This allows us to prove the following nonexistence result.

\begin{theorem}\label{thm:nonex}
Let $\Omega\subset\mathbb R^N$ be a bounded and starshaped $C^1$-domain, \eqref{eq:pq} hold, $0\le a \in C^1(\Omega)$ be radial and radially nondecreasing, and let $f(x,u)=c(x)|u|^{r-2}u+\mu|u|^{p^\ast-2}u+b(x)|u|^{q^\ast-2}u$, with $p\le r< q^\ast$, $\mu\le 0$, and $b,\, c\in L^\infty(\Omega)$.
Problem~\eqref{1} does not admit a nonzero weak solution $u\in W_0^{1,A}(\Omega)\cap W^{2,\A}(\Omega)$ in each of the following cases:
\begin{enumroman}
\item $c(x)\le 0$ a.e. in $\Omega$;
\item $r=p$ and $0 < c_{\infty} < \lambda_1(p)\frac{N(q-p)}{N(q-p)+pq}$;
\item $\Omega$ is strictly starshaped, $r=p$, and $c_\infty= \lambda_1(p)\frac{N(q-p)}{N(q-p)+pq}$,
\end{enumroman}
where $c_\infty=\pnorm[\infty]c$.
\end{theorem}

Another consequence of the Poho\v{z}aev type identity is the following nonexistence result for solutions with small norm.
\begin{theorem}\label{thm:small}
Let $\Omega\subset\mathbb R^N$ be a bounded and starshaped $C^1$-domain, \eqref{eq:pq} hold, $0\le a \in C^1(\Omega)$ be radial and radially nondecreasing, and let $f(x,u)=c(x)|u|^{r-2}u+\mu|u|^{p^\ast-2}u+b(x)|u|^{q^\ast-2}u$ with $\mu\ge 0$, $0\le b\in L^\infty(\Omega)$ satisfying \eqref{2}, and $0\le c\in C(\bar\Omega)\setminus\{0\}$.
Then there exists a positive constant $\kappa=\kappa(\Omega,p,q,r,a,c,\mu,b_\infty)$ such that
problem \eqref{1} does not admit a weak solution $u\in W_0^{1,A}(\Omega)\cap W^{2,\A}(\Omega)$ belonging to the ball $\{u\in W^{1,\A}(\Omega)\,:\,\|u\|\le \kappa\}$ in each of the following cases:
\begin{enumroman}
\item $p<r\le p^\ast$;
\item $p^\ast<r<q^\ast$ and $c(x)$ satisfies \eqref{eq:hp-on-c};
\item $p^\ast<r<q^\ast$ and $c(x)$ satisfies 
\begin{equation}\label{5q}
a_0':=\inf_{\mathrm{supp} (c)}a(x)>0
\end{equation}.
\end{enumroman}
\end{theorem}
\begin{remark}
Assumption \eqref{5q} implies in particular that $a(x)\neq 0$ for all $x$ having $c(x)\neq 0$. Moreover, we note that \eqref{5q} implies \eqref{eq:hp-on-c}. Indeed, 
\[
\begin{aligned}
\esssup_{\{x\in\Omega\,:\; a(x)\neq 0\}}\frac{c(x)^q}{a(x)^{s}} & = \max\left\{\esssup_{\{x\in\Omega\,:\; a(x) \neq 0, \, c(x) = 0\}}\frac{c(x)^q}{a(x)^{s}}, \esssup_{\{x\in\Omega\,:\; a(x) \neq 0, \,c(x)\neq 0\}}\frac{c(x)^q}{a(x)^{s}}\right\} \\
& = \esssup_{\{x\in\Omega\,:\; c(x)\neq 0\}}\frac{c(x)^q}{a(x)^{s}} \le \frac{c_\infty^q}{(a'_0)^s}<+\infty.
\end{aligned}
\]
Thus, if $a(x) = 0$, $c(x) = 0$ by \eqref{5q}, while for a.a. $x\in\Omega$ such that $a(x) \neq 0$, $c(x) \le \frac{c_\infty}{(a'_0)^{s/q}}a(x)^{s/q}$. Altogether, $c(x) \le C a(x)^{s/q}$ holds for a.a. $x\in\Omega$ with $C = c_\infty/(a'_0)^{s/q} >0$. 
\end{remark}
\section{Statements of the results for the nonlocal problem}
Turning to the question of existence of nontrivial solutions for the Kirchhoff type problem \eqref{2K}, our results, stated in Theorems \ref{Theorem 3} and \ref{Theorem 4} below, apply when $b_\infty \ge 0$ is small and $\mu > 0$ is arbitrary, or when $\mu = 0$ and $b_\infty > 0$  is arbitrary, respectively. 

We first introduce further assumptions on the Kirchhoff function $h$ and some useful notation.
Let
\[
H(t) := \int_0^t h(\tau)\, d\tau, \quad t \ge 0
\]
be the primitive of $h$ that vanishes at zero. For $\ell\in\{p,\,q\}$, the function
\[
K_\ell(t_1,t_2) := H\bigg(\frac{t_1}{p} + \frac{t_2}{q}\bigg) - \frac{1}{\ell^\ast}\, (t_1 + t_2)\, h\bigg(\frac{t_1}{p} + \frac{t_2}{q}\bigg), \quad t_1, t_2 \ge 0
\]
will play an important role in our results. Note that
\begin{equation} \label{8K}
H(\EA{u}) - \frac{1}{\ell^\ast}\, \rho_\A(\nabla u)\, h(\EA{u}) = K_\ell\big(\pnorm{\nabla u}^p,\pnorm[q,a]{\nabla u}^q\big), \quad u \in W^{1,\,\A}_0(\Omega).
\end{equation}
We make the following assumptions on $K_\ell$:
\begin{properties}{K}
\item \label{K1} there exist $\alpha > 0$ and $\gamma < p^*/q$, with $\gamma \ge 1$ if $\ell=p$ and $\gamma > p^*/p$ if $\ell=q$, for which
    \[
    K_\ell(t_1,t_2) \ge \alpha\, (t_1 + t_2)^\gamma \quad \forall \, t_1, t_2 \ge 0;
    \]
\item \label{K2} $K_\ell$ is superadditive in the sense that
    \[
    K_\ell(t_1 + s_1,t_2 + s_2) \ge K_\ell(t_1,t_2) + K_\ell(s_1,s_2) \quad \forall \, t_1, t_2, s_1, s_2 \ge 0.
    \]
\end{properties}

Moreover, for $\ell\in\{p,q\}$, we set
\[
\tilde{h}_\ell(t) := \frac{h(t/\ell)}{t^{\ell^\ast/\ell - 1}}\; \mbox{ for } t > 0\quad\mbox{and}\quad\widetilde{K}_\ell(t) := \begin{cases} K_p(t,0) \quad&\mbox{if }\ell=p,\\
K_q(0,t) &\mbox{if }\ell=q
\end{cases} \;\mbox{ for } t \ge 0
\]

\begin{remark} 
Our model case for the Kirchhoff function is
\begin{equation} \label{43}
h(t) = \alpha_1 + \alpha_2\, t^{\theta - 1},
\end{equation}
where $\alpha_1,\,\alpha_2\ge 0$  are constants with $\alpha_1+\alpha_2>0$ and $\theta \in (1,\ell^\ast/q]$. 
Here
\[
\begin{aligned}
K_\ell(t_1,t_2) = \alpha_1 &\left[\left(\frac{1}{p} - \frac{1}{\ell^\ast}\right)t_1+ \left(\frac{1}{q} - \frac{1}{\ell^\ast}\right) t_2\right] \\ 
& + \alpha_2 \left(\frac{t_1}{p} + \frac{t_2}{q}\right)^{\theta - 1} \left[\left(\frac{1}{\theta p} - \frac{1}{\ell^\ast}\right) t_1 + \left(\frac{1}{\theta q} - \frac{1}{\ell^\ast}\right) t_2\right].
\end{aligned}
\]

Since $p < q$,
\begin{equation}\label{46}
K_\ell(t_1,t_2) \ge \alpha_1 \left(\frac{1}{q} - \frac{1}{\ell^\ast}\right) (t_1 + t_2) + \frac{\alpha_2}{q^{\theta - 1}} \left(\frac{1}{\theta q} - \frac{1}{\ell^\ast}\right) (t_1 + t_2)^\theta,  
\end{equation}
and since $\theta q \le \ell^\ast$, this shows that, if $\ell = p$, \ref{K1} holds with
\begin{equation} \label{41}
\alpha = \alpha_1 \left(\frac{1}{q} - \frac{1}{p^\ast}\right), \qquad \gamma = 1
\end{equation}
when $\alpha_1 > 0$. In both cases $\ell = p,\, q$, \ref{K1} holds with
\begin{equation} \label{42}
\alpha = \frac{\alpha_2}{q^{\theta - 1}} \left(\frac{1}{\theta q} - \frac{1}{\ell^\ast}\right), \qquad \gamma = \theta
\end{equation}
when $\alpha_2 > 0$ and $\theta q<\ell^*$. When $\alpha_1,\,\alpha_2 > 0$ and $\theta q<\ell^*$, \eqref{46} together with the Young's inequality shows that \ref{K1} also holds for any $\gamma \in (1,\theta)$ and
\begin{equation} \label{47}
\alpha = \min \set{\frac{\theta - 1}{\theta - \gamma}\, \alpha_1 \left(\frac{1}{q} - \frac{1}{\ell^\ast}\right),\frac{\theta - 1}{\gamma - 1}\, \frac{\alpha_2}{q^{\theta - 1}} \left(\frac{1}{\theta q} - \frac{1}{\ell^\ast}\right)}.
\end{equation}
Since $\theta q \le \ell^\ast$, \ref{K2} also holds in both cases $\ell=p,\,q$.
\end{remark}

\subsection{Existence results}


We first consider the problem   
\begin{equation} \label{27K}
\left\{\begin{aligned}
- h(\EA{u})\, \DA{u} & = \lambda\, |u|^{\gamma p - 2}\, u + \mu\, |u|^{p^\ast - 2}\, u + c(x)\, |u|^{s - 2}\, u + b(x)\, |u|^{q^\ast - 2}\, u && \text{in } \Omega\\[10pt]
u & = 0 && \text{on } \bdry{\Omega},
\end{aligned}\right.
\end{equation}
where $h : [0,\infty) \to [0,\infty)$ is a continuous and nondecreasing function, $\gamma \in [1,p^\ast/q)$ is as in \ref{K1} for $\ell=p$, $\lambda, \mu > 0$ are parameters, $p^\ast < s < q^\ast$, and $b\in C(\bar\Omega)$ and $c \in L^\infty(\Omega)$ are nonnegative functions satisfying \eqref{2} and \eqref{eq:hp-on-c}, respectively. 

Let 
\begin{equation} \label{10K}
\lambda_1(\gamma,p) := \inf_{u \in W^{1,\,p}_0(\Omega) \setminus \set{0}}\, \frac{\left(\dint_\Omega |\nabla u|^p\, dx\right)^\gamma}{\dint_\Omega |u|^{\gamma p}\, dx}
\end{equation}
be the first eigenvalue of the nonlocal eigenvalue problem
\[
\left\{\begin{aligned}
- \left(\int_\Omega |\nabla u|^p\, dx\right)^{\gamma-1} \Delta_p u & = \lambda\, |u|^{\gamma p - 2}\, u && \text{in } \Omega\\[5pt]
u & = 0 && \text{on } \bdry{\Omega},
\end{aligned}\right.
\]
which is positive since $\gamma p < \gamma q < p^\ast$. Note that $\lambda_1(1,p) = \lambda_1(p)$, the first Dirichlet eigenvalue of the $p$-Laplacian in $\Omega$. 

We have the following theorem.

\begin{theorem} \label{Theorem 3}
Let $\ell = p$ and assume that \ref{K1} and \ref{K2}  
hold for $K_p$. Suppose furthermore that  
there is a ball $B_\rho(x_0) \subset \Omega$ such that $a(x) = 0$ for a.a.\! $x \in B_\rho(x_0)$, and
\[
0 < \lambda < \alpha \gamma p\, \lambda_1(\gamma,p), 
\]
where $\lambda_1(\gamma,p) > 0$ is as in \eqref{10K}. If $\tilde{h}_p$ is strictly decreasing and satisfies 
\begin{equation} \label{21K}
\lim_{t \to 0}\, \tilde{h}_p(t) > \frac{\mu}{S_p^{p^\ast/p}} > \lim_{t \to + \infty}\, \tilde{h}_p(t),
\end{equation}
then there exists $b^\ast > 0$ such that problem \eqref{27K} has a nontrivial weak solution $u \in W^{1,\,\A}_0(\Omega)$ for $0 \le b_\infty < b^\ast$ in each of the following cases:
\begin{enumroman}
\item $N \ge p^2$ and $1 \le \gamma < p^\ast/q$,
\item $N < p^2$ and $(Np - 2N + p)/(N - p)(p - 1) < \gamma < p^\ast/q$.
\end{enumroman}
\end{theorem}

We now give some applications of Theorem \ref{Theorem 3} to the model case \eqref{43}. First we consider the Brezis-Nirenberg type problem
\begin{equation} \label{40}
\left\{\begin{aligned}
- \left(\alpha_1 + \alpha_2\, \EA{u}^{\theta - 1}\right) \DA{u} & = \lambda\, |u|^{p - 2}\, u + |u|^{p^\ast - 2}\, u + b(x)\, |u|^{q^\ast - 2}\, u && \text{in } \Omega\\[10pt]
u & = 0 && \text{on } \bdry{\Omega},
\end{aligned}\right.
\end{equation}
where $\alpha_1 > 0$ and $\alpha_2 \ge 0$ are constants, $\theta \in (1,p^\ast/q]$, $\lambda > 0$ is a parameter, and $b \in C(\bar\Omega)$ is a nonnegative function satisfying \eqref{2}. Applying Theorem \ref{Theorem 3} with $\alpha$ and $\gamma$ as in \eqref{41} gives us the following corollary.

\begin{corollary}\label{cor:1}
Assume that there is a ball $B_\rho(x_0) \subset \Omega$ such that $a(x) = 0$ for a.a.\! $x \in B_\rho(x_0)$ and
\[
0 < \lambda < \alpha_1\, p \left(\frac{1}{q} - \frac{1}{p^\ast}\right) \lambda_1(p),
\]
where $\lambda_1(p) > 0$ is the first Dirichlet eigenvalue of the $p$-Laplacian in $\Omega$. If $N \ge p^2$, then there exists $b^\ast > 0$ such that problem \eqref{40} has a nontrivial weak solution in $W^{1,\,\A}_0(\Omega)$ for $0 \le b_\infty < b^\ast$.
\end{corollary}

Next we consider the problem
\begin{equation} \label{44}
\left\{\begin{aligned}
- \left(\alpha_1 + \alpha_2\, \EA{u}^{\theta - 1}\right) \DA{u} & = \lambda\, |u|^{\theta p - 2}\, u + \mu\, |u|^{p^\ast - 2}\, u + c(x)\, |u|^{s - 2}\, u + b(x)\, |u|^{q^\ast - 2}\, u && \text{in } \Omega\\[10pt]
u & = 0 && \text{on } \bdry{\Omega},
\end{aligned}\right.
\end{equation}
where $\alpha_1 \ge 0$ and $\alpha_2 > 0$ are constants, $\theta \in (1,p^\ast/q)$, $\lambda, \mu > 0$ are parameters, $p^\ast < s < q^\ast$, and $b\in C(\bar\Omega)$ and $c \in L^\infty(\Omega)$ are nonnegative functions satisfying \eqref{2} and \eqref{eq:hp-on-c}, respectively. Applying Theorem \ref{Theorem 3} with $\alpha$ and $\gamma$ as in \eqref{42} gives us the following corollary.

\begin{corollary}\label{cor:2}
Assume that there is a ball $B_\rho(x_0) \subset \Omega$ such that $a(x) = 0$ for a.a.\! $x \in B_\rho(x_0)$ and
\[
0 < \lambda < \frac{\alpha_2\, p}{q^{\theta - 1}} \left(\frac{1}{q} - \frac{\theta}{p^\ast}\right) \lambda_1(\theta,p).
\]
Then there exists $b^\ast > 0$ such that problem \eqref{44} has a nontrivial weak solution in $W^{1,\,\A}_0(\Omega)$ for all $\mu > 0$ and $0 \le b_\infty < b^\ast$ in each of the following cases:
\begin{enumroman}
\item $N \ge p^2$ and $1 < \theta < p^\ast/q$,
\item $N < p^2$ and $(Np - 2N + p)/(N - p)(p - 1) < \theta < p^\ast/q$.
\end{enumroman}
\end{corollary}

Finally we consider the problem
\begin{equation} \label{48}
\left\{\begin{aligned}
- \left(\alpha_1 + \alpha_2\, \EA{u}^{\theta - 1}\right) \DA{u} & = \lambda\, |u|^{r - 2}\, u + \mu\, |u|^{p^\ast - 2}\, u + c(x)\, |u|^{s - 2}\, u + b(x)\, |u|^{q^\ast - 2}\, u && \text{in } \Omega\\[10pt]
u & = 0 && \text{on } \bdry{\Omega},
\end{aligned}\right.
\end{equation}
where $\alpha_1, \alpha_2 > 0$ are constants, $\theta \in (1,p^\ast/q)$, $p < r < \theta p$, $\lambda, \mu > 0$ are parameters, $p^\ast < s < q^\ast$, and $b\in C(\bar\Omega)$ and $c \in L^\infty(\Omega)$ are nonnegative functions satisfying \eqref{2} and \eqref{eq:hp-on-c}, respectively. Applying Theorem \ref{Theorem 3} with $\gamma = r/p$ and $\alpha$ as in \eqref{47} gives us the following corollary.

\begin{corollary}
Assume that there is a ball $B_\rho(x_0) \subset \Omega$ such that $a(x) = 0$ for a.a.\! $x \in B_\rho(x_0)$ and
\[
0 < \lambda < \min \set{\frac{(\theta - 1)\, p}{\theta p - r}\, \alpha_1 \left(\frac{1}{q} - \frac{1}{p^\ast}\right),\frac{(\theta - 1)\, p}{r - p}\, \frac{\alpha_2}{q^{\theta - 1}} \left(\frac{1}{\theta q} - \frac{1}{p^\ast}\right)} r\, \lambda_1(r/p,p).
\]
Then there exists $b^\ast > 0$ such that problem \eqref{48} has a nontrivial weak solution in $W^{1,\,\A}_0(\Omega)$ for all $\mu > 0$ and $0 \le b_\infty < b^\ast$ in each of the following cases:
\begin{enumroman}
\item $N \ge p^2$ and $p < r < \theta p$,
\item $N < p^2$ and $(Np - 2N + p)\, p/(N - p)(p - 1) < r < \theta p$.
\end{enumroman}
\end{corollary}

\begin{remark} We observe that the degenerate case $\alpha_1=0$ is covered in Corollary \ref{cor:2}, while, in Corollary \ref{cor:1}, $\alpha_2=0$ is allowed, this leads to the local case $h\equiv \alpha_1$, for which we recover some known results of Theorem \ref{Theorem 1}-($i$), under a more restrictive assumption on $\lambda$. 
\end{remark}
\bigskip 


\noindent We then consider the problem 
\begin{equation} \label{27q}
\left\{\begin{aligned}
- h(\EA{u})\, \DA{u} & = \eta c(x)\, |u|^{\gamma q - 2}\, u + b(x)\, |u|^{q^\ast - 2}\, u && \text{in } \Omega\\[10pt]
u & = 0 && \text{on } \bdry{\Omega},
\end{aligned}\right.
\end{equation}
where $h : [0,\infty) \to [0,\infty)$ is a continuous and nondecreasing function, $\gamma \in (p^*/p,q^\ast/q)$ is as in \ref{K1} for $\ell = q$, $\eta > 0$ is a parameter, $b \in C(\bar\Omega)\setminus\{0\}$ is a nonnegative function satisfying \eqref{2}, and  $c \in C(\bar\Omega)\setminus\{0\}$ is a nonnegative function satisfying \eqref{5q}, namely $a_0' := \inf_{x \in \supp(c)}\, a(x) > 0$.

Let 
\begin{equation} \label{10q}
\lambda_1(\gamma,q,a) := \inf_{u \in W^{1,\A}_{0}(\Omega) \setminus \{u\,:\, u|_{\supp(c)} \equiv 0\}}\, \frac{\left(\dint_\Omega a(x)|\nabla u|^q\, dx\right)^\gamma}{\dint_\Omega c(x) |u|^{\gamma q}\, dx}
\end{equation}
be the first eigenvalue of the nonlocal eigenvalue problem
\[
\left\{\begin{aligned}
- \left(\int_\Omega a(x)|\nabla u|^q\, dx\right)^{\gamma-1} \mathrm{div}(a(x)|\nabla u|^{q-2}\nabla u) & = \lambda\, c(x) |u|^{\gamma q - 2}\, u && \text{in } \Omega\\[5pt]
u & = 0 && \text{on } \bdry{\Omega}.
\end{aligned}\right.
\]
We observe that for every $u \in W^{1,\A}_{0}(\Omega) \setminus \{u\,:\, u|_{\supp(c)} \equiv 0\}$, since $\gamma q < q^*$, by H\"older's inequality, \eqref{5q}, and the Sobolev embedding $W^{1,q}(\mathrm{supp(c)})\hookrightarrow L^{q^*}(\mathrm{supp(c)})$ we get
\[
\begin{aligned}
\frac{\left(\dint_{\Omega} a(x)|\nabla u|^q\, dx\right)^\gamma}{\dint_{\Omega} c(x) |u|^{\gamma q}\, dx} &\ge \frac{(a'_0)^\gamma\left(\dint_{\supp(c)} |\nabla u|^q\, dx\right)^\gamma}{c_\infty|\supp(c)|^{1-\gamma q/q^*}\bigg(\dint_{\supp(c)} |u|^{q^*}\, dx\bigg)^{\gamma q/q^*}} \ge \frac{C_S(a'_0)^\gamma}{c_\infty|\supp(c)|^{1-\gamma q/q^*}}>0.
\end{aligned}
\]
Thus, $\lambda_1(\gamma,q,a)>0$.

In order to get existence results for this case, we assume further that $h$ satisfies the following assumption:
\begin{equation}\label{hp:h}
\mbox{$h$ is differentiable at $t$, with $h'(t) \ge 0$, for all $t > 0$.}
\end{equation}
We observe that in the model case $h(t)=\alpha_1+\alpha_2t^{\theta-1}$, \eqref{hp:h} is clearly satisfied. 
Moreover, by the mean value theorem \eqref{hp:h} implies 
\begin{equation}\label{eq:conseq-h3}
H(t(1+\xi))-H(t) = \O(\xi)\quad\mbox{as $\xi\to 0^+$, \;for all } t> 0.
\end{equation}
We have the following theorem.

\begin{theorem} \label{Theorem 4}
Let $\ell = q$ and assume that \ref{K1} and \ref{K2} hold for $K_q$. 
Suppose furthermore that \eqref{hp:h} holds, that there is a ball $B_\rho(x_0) \subset \Omega$ such that 
\begin{equation} \label{9q}
a(x) = a_0, \quad b(x) = b_\infty>0, \quad c(x) \ge c_0 \quad \text{for a.a.\! } x \in B_\rho(x_0)
\end{equation}
for some constant $c_0 > 0$, and that 
\[0 < \eta < \alpha\gamma q\lambda_1(\gamma,q,a),\] 
where $\lambda_1(\gamma,q,a) > 0$ is as in \eqref{10q}. If $\tilde{h}_q$ is strictly decreasing and satisfies 
\begin{equation} \label{21q}
\lim_{t \to 0}\, \tilde{h}_q(t) > \frac{b_\infty}{(a_0S_q)^{q^\ast/q}} > \lim_{t \to + \infty}\, \tilde{h}_q(t),
\end{equation}
then 
problem \eqref{27q} has a nontrivial weak solution in $W^{1,\,\A}_0(\Omega)$ when 
$b_\infty > 0$ in each of the following cases:
\begin{enumroman}
\item $1 < p < N(q - 1)/(N - 1)$ and $N^2 (q - 1)/(N - 1)(N - q) < \gamma q < q^\ast$,
\item $N(q - 1)/(N - 1) \le p < q$ and $Np/(N - q) < \gamma q < q^\ast$.
\end{enumroman}
\end{theorem} 

We now give some applications of Theorem \ref{Theorem 4} to the model case \eqref{43}. First we consider the Brezis-Nirenberg type problem
\begin{equation} \label{44q}
\left\{\begin{aligned}
- \left(\alpha_1 + \alpha_2\, \EA{u}^{\theta - 1}\right) \DA{u} & = \eta c(x)\, |u|^{\theta q - 2}\, u + b(x)\, |u|^{q^\ast - 2}\, u && \text{in } \Omega\\[10pt]
u & = 0 && \text{on } \bdry{\Omega},
\end{aligned}\right.
\end{equation}
where $\alpha_1 \ge 0$ and $\alpha_2 > 0$ are constants, $p^*/p < \theta < q^\ast/q$, $\eta > 0$ is a parameter, and $b,\,c\in C(\bar \Omega)\setminus\{0\}$ are nonnegative functions satisfying \eqref{2} and \eqref{5q}, respectively. Applying Theorem \ref{Theorem 4} with $\alpha$ and $\gamma$ as in \eqref{42} gives us the following corollary.

\begin{corollary}\label{cor:4}
Assume that there is a ball $B_\rho(x_0) \subset \Omega$ such that \eqref{9q} holds
for some constant $c_0 > 0$ and that 
\[
0 < \eta < \frac{\alpha_2}{q^{\theta-2}}\left(\frac{1}{q}-\frac{\theta}{q^*}\right)\,\lambda_1(\theta,q,a).
\]
Then problem \eqref{44q} has a nontrivial weak solution in $W^{1,\,\A}_0(\Omega)$ in each of the following cases:
\begin{enumroman}
\item $1 < p < N(q - 1)/(N - 1)$ and $N^2 (q - 1)/(N - 1)(N - q) < \theta q < q^\ast$,
\item $N(q - 1)/(N - 1) \le p < q$ and $Np/(N - q) < \theta q < q^\ast$.
\end{enumroman}
\end{corollary}

Finally we consider the problem
\begin{equation} \label{48q}
\left\{\begin{aligned}
- \left(\alpha_1 + \alpha_2\, \EA{u}^{\theta - 1}\right) \DA{u} & = \eta c(x)\, |u|^{s - 2}\, u + b(x)\, |u|^{q^\ast - 2}\, u && \text{in } \Omega\\[10pt]
u & = 0 && \text{on } \bdry{\Omega},
\end{aligned}\right.
\end{equation}
where $\alpha_1, \alpha_2 > 0$ are constants, $p^*/p < \theta < q^\ast/q$, $p^* q/p < s < \theta q$, $\eta > 0$ is a parameter, and $b,\,c\in C(\bar \Omega)\setminus\{0\}$ are nonnegative functions satisfying \eqref{2} and \eqref{5q}, respectively. Applying Theorem \ref{Theorem 4} with $\gamma=s/q$ and $\alpha$ as in \eqref{47} gives us the following corollary.

\begin{corollary}
Assume that there is a ball $B_\rho(x_0) \subset \Omega$ such that \eqref{9q} holds
\[a(x) = a_0, \quad b(x) = b_\infty, \quad c(x) \ge c_0 \quad \text{for a.a.\! } x \in B_\rho(x_0)\]
for some constant $c_0 > 0$ and that 
\[
0 < \eta < \min\left\{\frac{(\theta -1)q}{\theta q-s}\alpha_1\left(\frac{1}{q}-\frac{1}{q^*}\right),\frac{(\theta -1)q}{s-q}\frac{\alpha_2}{q^{\theta-1}}\left(\frac{1}{\theta q}-\frac{1}{q^*}\right)\right\}\,s\,\lambda_1(s/q,q,a).
\]
Then problem \eqref{48q} has a nontrivial weak solution in $W^{1,\,\A}_0(\Omega)$ in each of the following cases:
\begin{enumroman}
\item $1 < p < N(q - 1)/(N - 1)$ and $N^2 (q - 1)/(N - 1)(N - q) < s < q^\ast$,
\item $N(q - 1)/(N - 1) \le p < q$ and $Np/(N - q) < s < q^\ast$.
\end{enumroman}
\end{corollary}
\noindent We observe that while the degenerate case $\alpha_1=0$ is covered in Corollary \ref{cor:4}, the previous two applications do not include the local case in which $h\equiv \mathrm{Const.}$ (i.e., $\theta=1$ or $\alpha_2=0$).

\subsection{Compactness results}

As for the local case, proofs of Theorems \ref{Theorem 3} and \ref{Theorem 4} will be based on a new compactness result that we will prove for the general critical Kirchhoff double phase problem
\begin{equation} \label{10KCpt}
\left\{\begin{aligned}
-h(\mathcal E_\A(u))\Da{u} & = \mu\, |u|^{p^\ast - 2}\, u + b(x)\, |u|^{q^\ast - 2}\, u + g(x,u) && \text{in } \Omega\\[10pt]
u & = 0 && \text{on } \bdry{\Omega},
\end{aligned}\right.
\end{equation}
where $g$ is a Carath\'{e}odory function on $\Omega \times \R$ satisfying the subcritical growth condition \eqref{11}. The energy functional associated with problem \eqref{10KCpt} is
\begin{equation}
\label{eq:energyJ}
J(u) = H(\EA{u}) - \int_\Omega \left(\frac{\mu}{p^\ast}\, |u|^{p^\ast} + \frac{b(x)}{q^\ast}\, |u|^{q^\ast} + G(x,u)\right) dx, \quad u \in W^{1,\,\A}_0(\Omega),
\end{equation}
where $G(x,t) = \int_0^t g(x,\tau)\, d\tau$.

The critical growth of the nonlinearity in \eqref{10KCpt} prevents the validity of the Palais-Smale condition for all levels $\beta$. We will  characterize variationally the threshold level for compactness as follows: 
\begin{equation} \label{12K}
\beta_{\mu,b_\infty,\ell}^\ast := \inf_{(t_1,t_2) \in S_{\mu,b_\infty} \setminus \set{(0,0)}}\, K_\ell(t_1,t_2),
\end{equation}
where
\[
S_{\mu,b_\infty} := \set{(t_1,t_2) \in \R^2 : t_1, t_2 \ge 0,\, (t_1 + t_2)\, h\bigg(\frac{t_1}{p} + \frac{t_2}{q}\bigg) \le \mu \left(\frac{t_1}{S_p}\right)^{p^\ast/p} + b_\infty \left(\frac{t_2}{a_0\, S_q}\right)^{q^\ast/q}}.
\]
We will show that $J$ satisfies the \PS{\beta} condition for all $\beta < \beta_{\mu,b_\infty,\ell}^\ast$. In particular, if $S_{\mu,b_\infty} = \set{(0,0)}$, then $\beta_{\mu,b_\infty,\ell}^\ast = + \infty$ and hence $J$ satisfies the \PS{\beta} condition for all $\beta \in \R$. More specifically, we have the following theorem.

\begin{theorem} \label{Theorem 1K}
Let $\ell\in\{p,q\}$, assume that \ref{K1} and \ref{K2} hold for $K_\ell$, and suppose that \eqref{2} and \eqref{11} hold. 

If $\ell=p$, assume furthermore that \eqref{eq:hp-on-c} holds and that $g$ satisfies
\begin{enumerate} 
\renewcommand*{\theenumi}{\textnormal{$(g_p)$}}
\renewcommand*{\labelenumi}{\theenumi}
\item\label{11K} $\displaystyle{G(x,t) - \frac{1}{p^\ast}\, t g(x,t) \le \vartheta\, |t|^{\gamma p}} \quad \text{for a.a.\! } x \in \Omega \text{ and all } t \in \R$
\end{enumerate}
for some constant $\vartheta < \alpha\, \lambda_1(\gamma,p)$. 

If $\ell=q$, assume furthermore that  
\eqref{5q} holds and that $g$ satisfies
\begin{enumerate} 
\renewcommand*{\theenumi}{\textnormal{$(g_q)$}}
\renewcommand*{\labelenumi}{\theenumi}
\item\label{11q} $\displaystyle{G(x,t) - \frac{1}{q^\ast}\, t g(x,t) \le \vartheta_1\,|t|^{r} +\vartheta_2\, c(x)|t|^{\gamma q}} \quad \text{for a.a.\! } x \in \Omega \text{ and all } t \in \R$
\end{enumerate}
for some constants $\vartheta_1 \ge 0$, $0 < \vartheta_2 < \alpha\, \lambda_1(\gamma,q,a)$, and $p < r < p^*$.
\smallskip 
 
If $\beta \in \R$ and $\seq{u_j} \subset W^{1,\,\A}_0(\Omega)$ is a {\em \PS{\beta}} sequence, the following statements hold true.
\begin{enumroman}
\item \label{Theorem 1.i} The sequence $\seq{u_j}$ is bounded and hence
    \[
    u_j \wto u, \qquad \pnorm{\nabla (u_j - u)}^p \to t_1, \qquad \pnorm[q,a]{\nabla (u_j - u)}^q \to t_2
    \]
    for a renamed subsequence and some $u \in W^{1,\,\A}_0(\Omega)$ and $t_1, t_2 \ge 0$.
\item \label{Theorem 1.ii} $(t_1,t_2) \in S_{\mu,b_\infty}$, in particular, $u_j \to u$ if $S_{\mu,b_\infty} = \set{(0,0)}$.
\item \label{Theorem 1.iii} If furthermore, $\beta < \beta_{\mu,b_\infty,\ell}^\ast$, then $u_j \to u$ when $\ell = p$ and when $\ell = q$ under the extra assumptions that $\mu = 0$ and \ref{11q} holds with $\vartheta_1 = 0$.
\end{enumroman}
\end{theorem}

\begin{remark} 
\begin{itemize}
\item[($i$)] If $g(x,t) = \lambda\, |t|^{\gamma p - 2}\, t + c(x)\, |t|^{s - 2}\, t$,
with $p^\ast \le s < q^\ast$, $\lambda \ge 0$, and $c \in L^\infty(\Omega)$ a nonnegative function satisfying \eqref{eq:hp-on-c}, \ref{11K} holds when
\[
\vartheta = \left(\frac{1}{\gamma p} - \frac{1}{p^\ast}\right) \lambda < \alpha\, \lambda_1(\gamma,p).
\]
\item[($ii$)] If $g(x,t) = \lambda |t|^{r-2}r + \eta c(x)\, |t|^{\gamma q - 2}\, t$,
with $\lambda \ge 0$ and $\eta > 0$, \ref{11q} holds with $\vartheta_1 = \lambda\left(\frac{1}{r} - \frac{1}{q^\ast}\right) \ge 0$ and 
\[
0 < \vartheta_2 = \eta \left(\frac{1}{\gamma q} - \frac{1}{q^\ast}\right) < \alpha\, \lambda_1(\gamma,q,a).
\]
\end{itemize}
\end{remark}

\begin{corollary} \label{Corollary 1K}
Let $\ell\in\{p,q\}$, assume that \ref{K1} and \ref{K2} hold for $K_\ell$, and suppose that \eqref{2} and \eqref{11} hold. 
Assume furthermore that 
\ref{11K} hold if $\ell = p$, and that $\mu = 0$, 
and \eqref{5q} and  \ref{11q} hold with $\vartheta_1 = 0$ if $\ell = q$. Then the following statements hold true.
\begin{enumroman}
\item \label{Corollary 1.i} $J$ satisfies the {\em \PS{\beta}} condition for all $\beta < \beta_{\mu,b_\infty,\ell}^\ast$.
\item \label{Corollary 1.ii} If the following inequality holds
    \[
    (t_1 + t_2)\, h\bigg(\frac{t_1}{p} + \frac{t_2}{q}\bigg) > \mu \left(\frac{t_1}{S_p}\right)^{p^\ast/p} + b_\infty \left(\frac{t_2}{a_0\, S_q}\right)^{q^\ast/q}
    \]
    for all $t_1, t_2 \ge 0$ with $t_1 + t_2 > 0$, then $J$ satisfies the {\em \PS{\beta}} condition for all $\beta \in \R$.
\end{enumroman}
\end{corollary}

First we show that the infimum in \eqref{12K} is attained and therefore positive by \ref{K1}.

\begin{proposition} \label{Proposition 2}
Let $\ell\in\{p,\,q\}$. Assume that \ref{K1} 
holds for $K_\ell$ and that $S_{\mu,b_\infty} \ne \set{(0,0)}$. Then there exists $(t_1,t_2) \in S_{\mu,b_\infty} \setminus \set{(0,0)}$ such that
\[
\beta_{\mu,b_\infty,\ell}^\ast = K_\ell(t_1,t_2) > 0.
\]
\end{proposition}

We have the following estimates for $\beta_{\mu,b_\infty,p}^\ast$ as $b_\infty \to 0$ and for $\beta_{0,b_\infty,q}^\ast$. 

\begin{theorem} \label{Theorem 2K}
Let $\ell\in\{p,q\}$, assume that \ref{K1} and \ref{K2} hold for $K_\ell$ and that $\tilde{h}_\ell$ is strictly decreasing.
\begin{enumroman}
\item \label{Theorem 2.i} If $\ell=p$ and $\tilde{h}_p$ satisfies \eqref{21K}, then $\beta_{\mu,b_\infty,p}^\ast > 0$ and
\begin{equation} \label{22K}
\beta_{\mu,b_\infty,p}^\ast \ge \widetilde{K}_p\bigg(\tilde{h}_p^{-1}\bigg(\frac{\mu}{S_p^{p^\ast/p}}\bigg)\bigg) + \text{\em o}(1) \quad \text{as } b_\infty \to 0.
\end{equation}
\item \label{Theorem 2.ii} If $\ell=q$, and $\tilde{h}_q$ satisfies \eqref{21q},
then 
\begin{equation} \label{22q}
\beta_{0,b_\infty,q}^\ast \ge \widetilde{K}_q\bigg(\tilde{h}_q^{-1}\bigg(\frac{b_\infty}{(a_0S_q)^{q^\ast/q}}\bigg)\bigg). 
\end{equation}
\end{enumroman}
\end{theorem}

We note that in the model case \eqref{43}, for both $\ell=p$ and $q$, $\tilde{h}_\ell$ is strictly decreasing, \eqref{21K} holds for all $\mu > 0$ and \eqref{21q} holds for all $b_\infty > 0$.

\section{Preliminaries}\label{sec2}
We recall here some notions on Musielak-Orlicz spaces that will be useful in the rest of the paper, see for reference \cite{MR724434}, Section 2 of \cite{MR2790542}, and also Section 2 of \cite{MR3558314}.


\subsection{Generalities on Musielak-Orlicz spaces}

\begin{definition} \label{def:Phi}
\rm A continuous, convex function $\varphi:[0,\infty)\to [0,\infty)$ is called {\it $\Phi$-function} if $\varphi(0)=0$ and $\varphi(t)>0$ for all $t>0$.
\smallskip

\noindent
A function $\varphi:\Omega\times[0,\infty)\to [0,\infty)$ is said to be a {\it generalized $\Phi$-function}, denoted by $\varphi\in\Phi(\Omega)$, if $\varphi(\cdot,t)$ is measurable for all $t\ge0$ and $\varphi(x,\cdot)$ is a $\Phi$-function for a.a. $x\in\Omega$.
\smallskip

\noindent
$\varphi\in\Phi(\Omega)$ is {\it locally integrable} if $\varphi(\cdot,t)\in L^1(\Omega)$ for all $t>0$.
\smallskip

\noindent
$\varphi\in\Phi(\Omega)$ satisfies the {\it $(\Delta_2)$-condition} if there exist a positive constant $C$ and a nonnegative function $h\in L^1(\Omega)$ such that
\[
\varphi(x,2t)\le C\varphi(x,t)+h(x)\quad\mbox{for a.a. }x\in\Omega\mbox{ and all }t\in[0,\infty).
\]
\noindent
Given $\varphi\in \Phi(\Omega)$ that satisfies the {\it $(\Delta_2)$-condition}, the {\it Musielak-Orlicz space} $L^\varphi(\Omega)$ is defined as follows
\[
L^\varphi(\Omega):=\big\{u:\Omega\to\mathbb R\mbox{ measurable }:\,\rho_\varphi(u)<\infty \big\},
\]
where $\rho_\varphi(u):=\int_\Omega\varphi(x,|u|) dx$ is the $\varphi$-{\it modular}. Endowed with the Luxemburg norm
\[
\|u\|_\varphi:=\inf\left\{\gamma>0\,:\,\rho_\varphi(u/\gamma) \le 1\right\},
\]
$L^\varphi$ is a Banach space (see \cite[Theorem 7.7]{MR724434}).
\end{definition}

The following proposition gives a relation between the modular and the norm in $L^\varphi(\Omega)$, the so called {\it unit ball property} (see, e.g., \cite[Lemma 2.1.14]{MR2790542}).
\begin{proposition}
\label{properties}
Let $\varphi\in \Phi(\Omega)$. If $u\in L^\varphi(\Omega)$, then
\begin{equation}\label{eq:unitball}
\rho_\varphi(u)<1\ \mbox{ \em{(resp. $=1;\,>1$)} }\quad \Leftrightarrow \quad \|u\|_\varphi<1\ \mbox{ \em{(resp. $=1;\,>1$).}}
\end{equation}
\end{proposition}

\begin{definition}
{\rm For $\varphi\in \Phi(\Omega)$, the related Sobolev space $W^{1,\varphi}(\Omega)$ is the set of all $L^\varphi(\Omega)$-functions $u$ having $|\nabla u|\in L^\varphi(\Omega)$, and is equipped with the norm $$\|u\|_{1,\varphi}=\|u\|_\varphi+\|\nabla u\|_\varphi,$$
where $\|\nabla u\|_\varphi$ stands for $\|\,|\nabla u|\,\|_\varphi$.\medskip

\noindent Similarly, the $m$-th order Musielak-Orlicz Sobolev space $W^{m,\varphi}(\Omega)$ is the set of all mesaurable functions $u$ on $\Omega$ having $|\nabla^k u|\in L^\varphi(\Omega)$ for $k=0,1,2$. \medskip

\noindent
$\varphi:[0,\infty)\to[0,\infty)$ is called {\it $\mathcal N$-function} ($\mathcal N$ stands for {\it nice}) if it is a $\Phi$-function satisfying
\[
\lim_{t\to0^+}\frac{\varphi(t)}{t}=0\quad\mbox{and}\quad\lim_{t\to\infty}\frac{\varphi(t)}{t}=\infty.
\]
A function $\varphi:\Omega\times\mathbb R\to [0,\infty)$ is said to be a {\it generalized $\mathcal N$-function}, and is denoted by $\varphi\in N(\Omega)$, if $\varphi(\cdot,t)$ is measurable for all $t\in\mathbb R$ and $\varphi(x,\cdot)$ is an $\mathcal N$-function for a.a. $x\in\Omega$.\medskip

\noindent
If $\varphi\in N(\Omega)$ is locally integrable, we denote by $W^{1,\varphi}_0(\Omega)$ the completion of $C^\infty_0(\Omega)$ in $W^{1,\varphi}(\Omega)$.
}
\end{definition}

We introduce below preliminary notation, definitions, and embeddings in general Musielak-Orlicz spaces, we postpone to the next subsection the statements of the embeddings for the specific Musielak-Orlicz Sobolev spaces of our double phase problems.

\begin{definition}
{\rm Let $\varphi,\,\psi\in \Phi(\Omega)$. The function $\varphi$ {\it is weaker than} $\psi$, denoted by $\varphi\preceq \psi$, if there exist two positive constants $C_1,\, C_2$ and a nonnegative function $h\in L^1(\Omega)$ such that
\[
\varphi(x,t)\le C_1\psi(x,C_2 t)+h(x)\quad\mbox{for a.a. }x\in\Omega\mbox{ and all }t\in [0,\infty).
\]
}
\end{definition}

\begin{proposition}\label{embeddingL}{\rm \cite[Theorem 8.5]{MR724434}}
Let $\varphi,\,\psi\in \Phi(\Omega)$, with $\varphi\preceq\psi$. Then $L^\psi(\Omega)\hookrightarrow L^\varphi(\Omega)$.
\end{proposition}

\begin{definition}
{\rm Let $\phi,\,\psi\in N(\Omega)$. We say that {\it $\phi$ increases essentially more slowly than $\psi$ near infinity}, and we write $\phi\ll\psi$, if for any $k>0$
\[
\lim_{t\to\infty}\frac{\phi(x,kt)}{\psi(x,t)}=0\quad\mbox{uniformly for a.a. }x\in\Omega.
\]
}
\end{definition}

\subsection{Functional setting and embeddings for double phase problems}
The function $\mathcal A:\Omega\times[0,\infty)\to[0,\infty)$ defined as
\[
\mathcal A(x,t):=t^p+a(x)t^q\quad\mbox{for all }(x,t)\in\Omega\times[0,\infty),
\]
with $1<p<q < N$, $q/p \le N/(N-1)$, $0\le a \in L^\infty(\Omega)\cap C^{0,\frac{N}{p}(q-p)}(\Omega)$ is a locally integrable generalized $\mathcal N$-function satisfying the $(\Delta_2)$-condition.
Therefore, the Musielak-Orlicz Lebesgue space $L^\A(\Omega)$ consists of all measurable functions $u : \Omega \to \R$ with finite $\A$-modular
\[
\rho_\A(u) := \int_\Omega \A(x,|u|)\, dx < \infty,
\]
endowed with the Luxemburg norm
\[
\norm[\A]{u} := \inf \set{\gamma > 0 : \rho_\A(u/\gamma) \le 1}.
\]
Since by \eqref{eq:unitball}, $\rho_\A(u/\norm[\A]{u}) = 1$ whenever $u \ne 0$, we have
\begin{equation} \label{17}
\min \set{\norm[\A]{u}^p,\norm[\A]{u}^q} \le \int_\Omega \big(|u|^p + a(x)\, |u|^q\big)\, dx \le \max \set{\norm[\A]{u}^p,\norm[\A]{u}^q} \quad \forall u \in L^\A(\Omega).
\end{equation}
For future use, we explicitly observe that, similarly, for a generic $\mathcal N$-function of the form $\varphi(x,t)=t^\alpha+\psi(x)t^\beta$, with $1<\alpha<\beta$, the corresponding chain of inequalities holds
\begin{equation} \label{17generic}
\min \set{\|u\|_\varphi^\alpha,\|u\|_\varphi^\beta} \le \int_\Omega \big(|u|^\alpha + \phi(x)\, |u|^\beta\big)\, dx \le \max \set{\|u\|_\varphi^\alpha,\|u\|_\varphi^q} \quad \forall u \in L^\varphi(\Omega).
\end{equation}
\smallskip 

The corresponding Musielak-Orlicz Sobolev space $W^{1,\,\A}(\Omega)$ consists of all functions $u$ in $L^\A(\Omega)$ with $|\nabla u| \in L^\A(\Omega)$, endowed with the norm
\[
\norm[1,\,\A]{u} := \norm[\A]{u} + \norm[\A]{\nabla u},
\]
where $\norm[\A]{\nabla u} = \norm[\A]{|\nabla u|}$. We work in the completion $W^{1,\,\A}_0(\Omega)$ of $C^\infty_0(\Omega)$ in $W^{1,\,\A}(\Omega)$.

We recall that the spaces $L^\mathcal A(\Omega)$, $W^{1,\mathcal A}(\Omega)$, and $W^{1,\mathcal A}_0(\Omega)$ are reflexive Banach spaces (see \cite[Propositions 2.14]{MR3558314}), and, moreover, the following Poincar\'e type inequality holds
\begin{equation*}
\norm[\A]{u}\le C_P \norm[\A]{\nabla u}\ \forall\, u\in W^{1,\mathcal A}_0(\Omega)
\end{equation*}
for some constant $C_P>0$ (see \cite[Proposition 2.18-$(iv)$]{MR3558314} and \cite[Theorem~1.2]{MR2870889}).

Hence, we can equivalently renorm $W^{1,\,\A}_0(\Omega)$ by setting $\norm{u} := \norm[\A]{\nabla u}$.

We observe that, by \eqref{17}, we have
\begin{equation} \label{18}
\min \set{\norm{u}^p,\norm{u}^q} \le \int_\Omega \big(|\nabla u|^p + a(x)\, |\nabla u|^q\big)\, dx \le \max \set{\norm{u}^p,\norm{u}^q} \quad \forall u \in W^{1,\,\A}_0(\Omega).
\end{equation}

We collect below some embedding results that will be useful for our analysis. 

\begin{proposition}[Proposition 2.15 of \cite{MR3558314}]\label{prop:embeddings1}
	Let $1<p<q<N$, $0\le a\in L^\infty(\Omega)$, 
	 and let $L^q_a(\Omega)$ be the weighted Lebesgue space $L^q_a(\Omega):=\left\{u:\Omega\to\mathbb R\,\mbox{{\rm measurable}}\,:\,\int_\Omega a(x)|u|^q\,dx<\infty\right\}$, endowed with the seminorm $|u|_{q,a}:=\left(\int_\Omega a(x)|u|^qdx\right)^{1/q}$.
Then the following embeddings hold:
\begin{enumroman}
\item $L^\mathcal A(\Omega)\hookrightarrow L^r(\Omega)$ and $W^{1,\mathcal A}_0(\Omega)\hookrightarrow W^{1,r}_0(\Omega)$ for all $r\in [1, p]$;
\item $W_0^{1,\mathcal A}(\Omega)\hookrightarrow L^r(\Omega)$ for all $r\in [1, p^\ast]$ and $W_0^{1,\mathcal A}(\Omega)\hookrightarrow\hookrightarrow L^r(\Omega)$ for all $r\in [1, p^\ast)$;
\item $L^q(\Omega)\hookrightarrow L^\mathcal A(\Omega)\hookrightarrow L^q_a(\Omega)$.
\end{enumroman}
\end{proposition}

\begin{definition}\label{criticalH}
{\rm For all $x\in\Omega$ denote by $\mathcal A^{-1}(x,\cdot):[0,\infty)\to[0,\infty)$ the inverse function of $\mathcal A(x,\cdot)$ and define $\mathcal A^{-1}_\ast:\Omega\times[0,\infty)\to[0,\infty)$ by
$$\mathcal A^{-1}_\ast(x,\tau):=\int_0^\tau\frac{\mathcal A^{-1}(x,\sigma)}{\sigma^{(N+1)/N}}d\sigma\quad\mbox{for all }(x,\tau)\in\Omega\times[0,\infty).$$
The function $\mathcal A_\ast:(x,t)\in \Omega\times[0,\infty)\mapsto \tau\in[0,\infty)$ where $\tau$ is such that $\mathcal A^{-1}_\ast(x,\tau)=t$, is called {\it Sobolev conjugate function of} $\mathcal A$.}
\end{definition}

\begin{proposition}[Theorems 1.1 - 1.2 and Proposition 3.1 of \cite{MR2870889} and Proposition 2.18 of \cite{MR3558314}]\label{prop:embeddings2} 
The following embeddings hold:
\begin{itemize}
\item[$(i)$] $W^{1,\mathcal A}(\Omega)\hookrightarrow L^{\mathcal A_\ast}(\Omega)$;
\item[$(ii)$] if $\mathcal K:\Omega\times[0,\infty)\to[0,\infty)$ is a continuous generalized $\mathcal N$-function such that $\mathcal K\ll\mathcal A_\ast$, then $W^{1,\mathcal A}(\Omega)\hookrightarrow\hookrightarrow L^\mathcal K(\Omega)$;
\item[$(iii)$] $\mathcal A\ll\mathcal A_\ast$, and consequently $W^{1,\mathcal A}(\Omega)\hookrightarrow\hookrightarrow L^\mathcal A(\Omega)$.
\end{itemize}
\end{proposition}

\begin{proposition}\label{prop:subcritical}
Let $\C$ be of the form
\[
\C(x,t) = t^r + c(x)\, t^s, \quad (x,t) \in \Omega \times [0,\infty),
\]
where $1 < r < p^\ast \le s < q^\ast$ and $c \in L^\infty(\Omega)$ is a nonnegative function satisfying \eqref{eq:hp-on-c}.
Then the following compact embedding holds: $W^{1,\,\A}_0(\Omega)\hookrightarrow \hookrightarrow L^\C(\Omega)$.
\end{proposition}

\begin{proof}
By \cite[Theorem 3.7-(i), Remark 3.8, and Examples 2.4 and 3.11]{CianchiDiening}, it is enough to prove that for all $k>0$
\begin{equation}
\label{eq:C<<AN}
\lim_{t\to+\infty}\frac{\C(x,kt)}{t^{p^*}+a(x)^{q^*/q}t^{q^*}}=0\quad\mbox{uniformly for a.a.}x\in\Omega.
\end{equation}
By \eqref{eq:hp-on-c} and Young's inequality, for $\varepsilon>0$ we get
\[
\begin{aligned}
\frac{\C(x,kt)}{t^{p^*}+a(x)^{q^*/q}t^{q^*}}&\le \frac{k^rt^r+Ca(x)^{s/q}k^st^s}{t^{p^*}+a(x)^{q^*/q}t^{q^*}} \le (1+k^s)\frac{\varepsilon t^{p^*}+\varepsilon^{r/(r-p^*)}+C[\varepsilon a(x)^{q^*/q}t^{q^*}+\varepsilon^{s/(s-q^*)}]}{t^{p^*}+a(x)^{q^*/q}t^{q^*}}\\
&\le (1+k^s)\max\{1,C\}\left[\varepsilon+\frac{\varepsilon^{r/(r-p^*)}+\varepsilon^{s/(s-q^*)}}{t^{p^*}+a(x)^{q^*/q}t^{q^*}}\right]\\
&\le (1+k^s)\max\{1,C\}\left[\varepsilon+\frac{\varepsilon^{r/(r-p^*)}+\varepsilon^{s/(s-q^*)}}{t^{p^*}}\right]\; \to (1+k^s)\max\{1,C\}\varepsilon 
\end{aligned}
\]
as $t\to+\infty$. By the arbitrariness of $\varepsilon$, \eqref{eq:C<<AN} follows.
\end{proof}

We refer to \cite[Proposition 3.7]{HoWinkert} for a similar compactness result, under different assumptions on the weight $c(x)$.

\begin{proposition}\label{prop:emb-cont-crit}
Let $b \in L^\infty(\Omega)$ be a nonnegative function such that \eqref{2} holds, namely
\[
a_0 = \inf_{x \in \supp(b)}\, a(x) > 0
\]
and let $L^\B(\Omega)$ be the Musielak-Orlicz space associated with
\[
\B(x,t) = t^{p^\ast} + b(x)\, t^{q^\ast}, \quad (x,t) \in \Omega \times [0,\infty).
\]
Then, $W^{1,\A}_0(\Omega)\hookrightarrow L^\B(\Omega)$.

Furthermore, for every $u\in W^{1,\A}_0(\Omega)$ the following estimates hold:
\begin{equation} \label{3}
\int_\Omega |u|^{p^\ast} dx \le \frac{1}{S_p^{p^\ast/p}} \left(\int_\Omega |\nabla u|^p\, dx\right)^{p^\ast/p}
\end{equation}
and  
\begin{equation} \label{4}
\int_\Omega b(x)\, |u|^{q^\ast} dx \le \kappa \left(\int_\Omega a(x)\, |\nabla u|^q\, dx\right)^{q^\ast/q}, 
\end{equation}
where $S_p$ is the best Sobolev constants introduced in \eqref{5} and $\kappa=C_S b_\infty/a_0^{q^\ast/q}$ involves $b_\infty := \pnorm[\infty]{b}$ and the constant $C_S$ arising from the embedding $W^{1,q}(\supp(b))\hookrightarrow L^{q^\ast}(\supp(b))$.
\end{proposition}

\begin{proof}
Let $u \in W^{1,\,\A}_0(\Omega)$. Since $W^{1,\A}_0(\Omega)\hookrightarrow W^{1,p}_0(\Omega)$ by Proposition \ref{prop:embeddings1}, \eqref{3} follows immediately by the Sobolev embedding $W^{1,p}_0(\Omega)\hookrightarrow L^{p^\ast}(\Omega)$. As for \eqref{4}, if $b\equiv 0$, \eqref{4} is trivially verified, otherwise 
\begin{multline*}
\int_\Omega b(x)\, |u|^{q^\ast} dx \le b_\infty \int_{\supp(b)} |u|^{q^\ast} dx \le C_S b_\infty \left(\int_{\supp(b)} |\nabla u|^q\, dx\right)^{q^\ast/q}\\[7.5pt]
\le \frac{C_S b_\infty}{a_0^{q^\ast/q}} \left(\int_{\supp(b)} a(x)\, |\nabla u|^q\, dx\right)^{q^\ast/q} \le \frac{C_S b_\infty}{a_0^{q^\ast/q}} \left(\int_\Omega a(x)\, |\nabla u|^q\, dx\right)^{q^\ast/q},
\end{multline*}
where we used the embedding $W^{1,q}(\supp(b))\hookrightarrow L^{q^\ast}(\supp(b))$.
Now, let $u_j\to u$ in $W^{1,\A}_0(\Omega)$. By \eqref{18}, $\rho_\A(\nabla u_j-\nabla u)\to 0$, and so, by \eqref{4} and \eqref{5}, $\rho_\B(u_j-u)\to 0$, which by \eqref{17generic} implies that $u_j\to u$ in $L^\B(\Omega)$ and concludes the proof. 
\end{proof}

In order to estimate explicitly the threshold level for which the Palias-Smale condition holds for the energy functionals, we need to derive an inequality, similar to \eqref{4}, involving a constant that is independent of $\mathrm{supp}(b)$. More precisely, under stronger regularity assumptions on $b$, we show that the following estimate holds.

\begin{proposition}\label{prop:K-add}
Let $b \in C(\bar{\Omega})$ be a nonnegative function such that \eqref{2} holds. Then for any $\eps > 0$, there exists a constant $C_\eps > 0$ such that
\begin{equation} \label{1000}
\int_\Omega b(x)\, |u|^{q^\ast} dx \le \frac{b_\infty}{(a_0\, S_q)^{q^\ast/q}} \left[(1 + \eps) \int_\Omega a(x)\, |\nabla u|^q\, dx + C_\eps \int_U |u|^q\, dx\right]^{q^\ast/q}
\end{equation}
for all $u \in W^{1,\A}_0(\Omega)$, where $b_\infty := \pnorm[\infty]{b}$, $U = \set{x \in \Omega : b(x) > 0}$, and $S_q$ is the best Sobolev constant introduced in \eqref{5}.
\end{proposition}

\begin{proof}
We may assume that $u \in C^\infty_0(\Omega)$. Let $\eps_1, \eps_2 > 0$ and set
\[
U_\delta = \set{x \in U : \dist{x}{\partial{U} \setminus \partial{\Omega}} > \delta}
\]
for $\delta > 0$. Since $b = 0$ on $\bdry{U} \setminus \bdry{\Omega}$ and $b$ is uniformly continuous, $b \le \eps_1$ on $U \setminus U_\delta$ if $\delta$ is sufficiently small. Fix such a $\delta$ and let $\eta : U \to [0,1]$ be a smooth functions such that $\eta = 1$ on $U_\delta$ and $\eta = 0$ outside $U_{\delta/2}$. Then
\[
\int_\Omega b(x)\, |u|^{q^\ast} dx = \int_U b(x)\, |\eta u|^{q^\ast} dx + \int_U b(x)\, (1 - \eta^{q^\ast})\, |u|^{q^\ast} dx \le b_\infty \int_U |\eta u|^{q^\ast} dx + \eps_1 \int_U |u|^{q^\ast} dx
\]
and hence
\begin{equation} \label{1001}
\left(\int_\Omega b(x)\, |u|^{q^\ast} dx\right)^{q/q^\ast} \le b_\infty^{q/q^\ast} \left(\int_U |\eta u|^{q^\ast} dx\right)^{q/q^\ast} + \eps_1^{q/q^\ast} \left(\int_U |u|^{q^\ast} dx\right)^{q/q^\ast}.
\end{equation}
Since $\eta u \in C^\infty_0(U)$,
\begin{equation}
\left(\int_U |\eta u|^{q^\ast} dx\right)^{q/q^\ast} \le \frac{1}{S_q} \int_U |\nabla (\eta u)|^q\, dx.
\end{equation}
We have
\begin{multline}
\int_U |\nabla (\eta u)|^q\, dx \le \int_U (\eta\, |\nabla u| + |\nabla \eta|\, |u|)^q\, dx \le \int_U ((1 + \eps_2)\, \eta^q\, |\nabla u|^q + C_{\eps_2}\, |\nabla \eta|^q\, |u|^q)\, dx\\[7.5pt]
\le (1 + \eps_2) \int_U |\nabla u|^q\, dx + C_{\eps_1,\eps_2} \int_U |u|^q\, dx
\end{multline}
for some constants $C_{\eps_2}, C_{\eps_1,\eps_2} > 0$. On the other hand,
\begin{equation}
\left(\int_U |u|^{q^\ast} dx\right)^{q/q^\ast} \le C \int_U (|\nabla u|^q + |u|^q)\, dx
\end{equation}
for some constant $C > 0$ depending on $U$. By \eqref{5},
\begin{equation} \label{1002}
\int_U |\nabla u|^q\, dx \le \frac{1}{a_0} \int_U a(x)\, |\nabla u|^q\, dx \le \frac{1}{a_0} \int_\Omega a(x)\, |\nabla u|^q\, dx.
\end{equation}
Combining \eqref{1001}--\eqref{1002} and choosing $\eps_1$ and $\eps_2$ appropriately gives \eqref{1000}.
\end{proof}

\section{Proofs of compactness results in the local case}\label{sec3}
\subsection{Some preliminary lemmas}
We begin with the following lemmas.
\begin{lemma} \label{Lemma 0}
If $\seq{u_j} \subset W^{1,\,\A}_0(\Omega)$ converges weakly to $u$ and $E'(u_j) \to 0$, then, up to a subsequence, $\nabla u_j\to\nabla u$ a.e. in $\Omega$.
\end{lemma}

\begin{proof}
Let
\[T(\tau):=\begin{cases}\tau\quad&\mbox{if }|\tau|\le 1,\\
\frac{\tau}{|\tau|}&\mbox{if }|\tau|>1\end{cases}
\]
We recall that, if we prove that the following limit holds true
\begin{equation}\label{eq:thesis}
\int_\Omega\left(|\nabla u_j|^{p-2}\nabla u_j-|\nabla u|^{p-2}\nabla u\right)\cdot \nabla T(u_j-u)\,dx\to 0\quad\mbox{as }j\to\infty,
\end{equation}
we can conclude the proof by \cite[Theorem 1.1]{MR2560133}.

Let
\[
\Omega_j:=\{x\in\Omega\,:\,|u_j(x)-u(x)|\le 1\}\quad\mbox{for every }j\in\mathbb N.
\]
We recall that, by Simon's inequality, if $r\ge 2$, there exists a constant $\kappa_r>0$ such that
\[
(|\xi|^{r-2}\xi-|\eta|^{r-2}\eta)\cdot(\xi-\eta)\ge \kappa_r|\xi-\eta|^r
\]
for all $\xi\,\eta\in \mathbb R^N$.
Hence, for every $j\in\mathbb N$
\begin{equation}\label{eq:Ippos}
\begin{aligned}
I_{p,j}&:=\int_\Omega\left(|\nabla u_j|^{p-2}\nabla u_j-|\nabla u|^{p-2}\nabla u\right)\cdot \nabla T(u_j-u)\,dx\\
&=\int_{\Omega_j}\left(|\nabla u_j|^{p-2}\nabla u_j-|\nabla u|^{p-2}\nabla u\right)\cdot \nabla(u_j-u)\,dx\\
&\ge \kappa_p\int_{\Omega_j}|\nabla u_j-\nabla u|^p\,dx\ge 0
\end{aligned}
\end{equation}
and similarly,
\begin{equation}\label{eq:Iqpos}
\begin{aligned}
I_{q,j}&:=\int_\Omega a(x)\left(|\nabla u_j|^{q-2}\nabla u_j-|\nabla u|^{q-2}\nabla u\right)\cdot \nabla T(u_j-u)\,dx\\
&=\int_{\Omega_j}a(x)\left(|\nabla u_j|^{q-2}\nabla u_j-|\nabla u|^{q-2}\nabla u\right)\cdot \nabla(u_j-u)\,dx\\
&\ge \kappa_q\int_{\Omega_j}a(x)|\nabla u_j-\nabla u|^q\,dx\ge 0.
\end{aligned}
\end{equation}
Moreover, by $u_j\rightharpoonup u$ in $W^{1,\,\A}_0(\Omega)\hookrightarrow W^{1,p}_0(\Omega)$, $u_j\to u$ a.e. in $\Omega$ up to a subsequence still denoted by $u_j$.
Therefore, definitely in $j$, along such a subsequence, $T(u_j-u)=u_j-u$ and also $\nabla T(u_j-u)=\nabla (u_j-u)$. Now, since $u_j\rightharpoonup u$ in $W^{1,\mathcal A}_0(\Omega)$, $\nabla (u_j-u)\rightharpoonup 0$ in $(L^{\mathcal A}(\Omega))^N$ and so also $\nabla T(u_j-u)\rightharpoonup 0$ in $(L^{\mathcal A}(\Omega))^N$. Hence,
\begin{equation}\label{eq:conseq-weak}
\int_\Omega\left(|\nabla u|^{p-2}\nabla u+a(x)|\nabla u|^{q-2}\nabla u\right)\cdot \nabla T(u_j-u)\,dx\to 0\quad\mbox{as }j\to\infty
\end{equation}
and so
\begin{equation}\label{eq:limsup}
\begin{aligned}
&\limsup_{j\to\infty}(I_{p,j}+I_{q,j})=\limsup_{j\to\infty}\left\{\int_\Omega\left(|\nabla u_j|^{p-2}\nabla u_j+a(x)|\nabla u_j|^{q-2}\nabla u_j\right)\cdot\nabla T(u_j-u)\right\}\\
&\;=\limsup_{j\to\infty}\left\{E'(u_j)[T(u_j-u)]+\int_\Omega\left(\mu|u_j|^{p^\ast-2}u_j+b(x)|u_j|^{q^\ast-2}u_j+g(x,u_j)\right)T(u_j-u)\,dx\right\}\\&\;=\limsup_{j\to\infty}\int_\Omega\left(\mu|u_j|^{p^\ast-2}u_j+b(x)|u_j|^{q^\ast-2}u_j+g(x,u_j)\right)T(u_j-u)\,dx,
\end{aligned}
\end{equation}
where in the last equality we used the assumption $E'(u_j)\to0$.

Set
\[
I_j:=\int_\Omega\left(\mu|u_j|^{p^\ast-2}u_j+b(x)|u_j|^{q^\ast-2}u_j+g(x,u_j)\right)T(u_j-u)\,dx,
\]
if we prove that $I_j\to0$, we can conclude by \eqref{eq:limsup}, \eqref{eq:Ippos}, and \eqref{eq:Iqpos}, that both $I_{p,j}$ and $I_{q,j}$ go to zero as $j\to\infty$. In particular $I_{p,j}\to 0$ implies \eqref{eq:thesis} by \eqref{eq:conseq-weak}.
To this aim, using H\"older's inequality and \eqref{11}, we can estimate
\begin{equation}\label{eq:ineq}
\begin{gathered}
\int_\Omega \mu|u_j|^{p^\ast-2}u_j T(u_j-u)\,dx\le \mu\left(\int_\Omega|u_j|^{p^\ast}\right)^{\frac{p^\ast-1}{p^\ast}}\left(\int_\Omega|T(u_j-u)|^{p^\ast}\right)^{\frac{1}{p^\ast}},\\
\int_\Omega b(x)^{\frac{q^\ast-1}{q^\ast}+\frac{1}{q^\ast}}|u_j|^{q^\ast-2}u_j T(u_j-u)\,dx\le \left(\int_\Omega b(x)|u_j|^{q^\ast}\right)^{\frac{q^\ast-1}{q^\ast}}\left(\int_\Omega b(x)|T(u_j-u)|^{q^\ast}\right)^{\frac{1}{q^\ast}},
\end{gathered}
\end{equation}
\begin{equation}\label{eq:ineq'}
\begin{aligned}
&\int_\Omega g(x,u_j) T(u_j-u)\,dx \le \int_\Omega \left(c_1+c_2|u_j|^{r-1}+c(x)|u_j|^{s-1}\right)T(u_j-u)\,dx\\
&\hspace{4.2cm}\le c_1\int_\Omega |T(u_j-u)|\,dx\\
&\hspace{4.2cm}\phantom{\le\,} + c_2\left(\int_\Omega|u_j|^r\,dx\right)^{\frac{r-1}{r}}\left(\int_\Omega|T(u_j-u)|^r\,dx\right)^{\frac{1}{r}}\\
&\hspace{4.2cm}\phantom{\le\,} +\left(\int_\Omega c(x)|u_j|^s\,dx\right)^{\frac{s-1}{s}}\left(\int_\Omega c(x)|T(u_j-u)|^s\,dx\right)^{\frac{1}{s}}.
\end{aligned}
\end{equation}
Clearly, since $u_j\rightharpoonup u\in W^{1,\mathcal A}_0(\Omega)\hookrightarrow\hookrightarrow L^1(\Omega)$, up to a subsequence $u_j\to u$ a.e. and so $\Omega_j=\Omega$ for $j$ large. Thus,
\[
\int_\Omega |T(u_j-u)|\,dx=\int_{\Omega_j} |u_j-u|\,dx\le |u_j-u|_{1}\to 0 \quad\mbox{as }j\to\infty.
\]
Furthermore, by Propositions \ref{prop:subcritical} and \ref{prop:emb-cont-crit} and since $\seq{u_j}$ is bounded in $W^{1,\mathcal A}_0(\Omega)$, all the integrals
\[
\int_\Omega|u_j|^{p^\ast}, \quad \int_\Omega b(x)|u_j|^{q^\ast}, \quad \int_\Omega|u_j|^r\,dx,\quad\mbox{and}\quad\int_\Omega c(x)|u_j|^s\,dx
\]
are bounded. Finally, since $u_j\to u$ a.e. in $\Omega$, also $T(u_j-u)\to 0$ a.e. in $\Omega$, and $|T(u_j-u)|\le 1\in L^1(\Omega)$. Thus, by the Dominated Convergence Theorem, all the integrals in the right-hand sides of \eqref{eq:ineq} and \eqref{eq:ineq'} involving $T(u_j-u)$ go to zero as $j\to\infty$.
This proves that $I_j\to0$ as well, and concludes the proof.
\end{proof}

\begin{lemma} \label{Lemma 1}
If $\seq{u_j} \subset W^{1,\,\A}_0(\Omega)$ converges weakly to $u$ and $E'(u_j) \to 0$, then $u$ is a weak solution of problem \eqref{10}.
\end{lemma}

\begin{proof} By assumption we have for every $v\in W^{1,\,\mathcal{A}}_0(\Omega)$
\begin{multline}\label{eq:take-lim}
\int_\Omega \Big[|\nabla u_j|^{p-2}\nabla u_j \cdot \nabla v + a(x)\, |\nabla u_j|^{q-2}\nabla u_j \cdot \nabla v - \mu\, |u_j|^{p^\ast-2}u_j v - b(x)\, |u_j|^{q^\ast-2}v\\[7,5pt]
- g(x,u_j)v\Big]\, dx = o(1)
\end{multline}
as $j\to\infty$. Our goal is to pass in the limit under the integral sign. In what follows the symbols $C,\,C'$ denote different positive constants whose exact values are not important for the argument of the proof.

Let $(u_{j_k})$ be any subsequence of $(u_j)$. By the continuous embedding $W^{1,\,\mathcal A}_0(\Omega)\hookrightarrow W^{1,p}_0(\Omega)$ and the hypothesis that $\seq{u_j}$ is weakly convergent in $W^{1,\,\mathcal A}_0(\Omega)$, $|\nabla u_{j_k}|_{p}\le C\|\nabla u_{j_k}\|_{\mathcal A}\le C'$ for every $k$. Since, by Lemma \ref{Lemma 0}, $\nabla u_{j_k}\to \nabla u$ a.e. in $\Omega$ up to a subsequence still indexed by $j_k$, by \cite[Proposition~A.8-$(i)$, with $w=\chi_\Omega$]{MR3057162}, $|\nabla u_{j_k}|^{p-2}\nabla u_{j_k}\rightharpoonup |\nabla u|^{p-2}\nabla u$ in $(L^{p'}(\Omega))^N$. By the arbitrariness of the subsequence, the entire sequence $(|\nabla u_{j}|^{p-2}\nabla u_{j})$ converges weakly to $|\nabla u|^{p-2}\nabla u$ in $(L^{p'}(\Omega))^N$. In particular, for every $v\in W^{1,\,\mathcal A}_0(\Omega)$,
\begin{equation}\label{eq:w-gradp}
\int_\Omega |\nabla u_j|^{p-2}\nabla u_j \cdot \nabla v\,dx\to \int_\Omega |\nabla u|^{p-2}\nabla u \cdot \nabla v\,dx\quad\mbox{as }j\to\infty,
\end{equation}
since $\nabla v\in (L^\mathcal A(\Omega))^N\subset (L^p(\Omega))^N$.

Similarly, since $\int_\Omega a(x)|\nabla u_{j_k}|^q\,dx\le \rho_\mathcal A(\nabla u_{j_k})\le C$ by \eqref{17}, and using the a.e. convergence of the gradients, by \cite[Proposition~A.8-$(i)$, with $w=a(x)\chi_\Omega$]{MR3057162}, it holds $|\nabla u_{j_k}|^{q-2}\nabla u_{j_k}\rightharpoonup |\nabla u|^{q-2}\nabla u$ in the weighted Lebesgue space
\[
(L^{q'}(\Omega;a(x)))^N := \left\{\varphi:\Omega\to \mathbb R^N\,\mbox{ measurable}\,:\, \int_\Omega a(x)|\varphi|^{q'}\,dx<\infty\right\}.
\]
Hence, the whole sequence $(|\nabla u_{j}|^{q-2}\nabla u_{j})$ converges weakly to $|\nabla u|^{q-2}\nabla u$ in $(L^{q'}(\Omega;a(x)))^N$. In particular, for every $v\in W^{1,\,\mathcal A}_0(\Omega)$,
\begin{equation}\label{eq:w-gradq}
\int_\Omega a(x)|\nabla u_j|^{q-2}\nabla u_j \cdot \nabla v\,dx\to \int_\Omega a(x)|\nabla u|^{q-2}\nabla u \cdot \nabla v\,dx\quad\mbox{as }j\to\infty,
\end{equation}
since $\nabla v\in (L^\mathcal A(\Omega))^N\subset (L^q(\Omega;a(x)))^N$.

Analogously, by the continuous embedding $W^{1,\,\mathcal A}_0(\Omega)\hookrightarrow L^{p^\ast}(\Omega)$ and the hypothesis that $\seq{u_j}$ is weakly convergent -and so bounded- in $W^{1,\,\mathcal A}_0(\Omega)$, the sequence $(u_{j_k})$ is bounded in $L^{p^\ast}(\Omega)$. Moreover, up to a further subsequence, $u_{j_k}\to u$ a.e. in $\Omega$, and so applying again \cite[Proposition~A.8-$(i)$, with $w=\chi_\Omega$]{MR3057162}, we can conclude by the arbitrariness of the subsequence that $|u_j|^{p^\ast-2}u_j\rightharpoonup |u|^{p^\ast-2}u$ in $L^{p^{\ast'}}(\Omega)$. In particular, for every $v\in W^{1,\,\mathcal A}_0(\Omega)$,
\begin{equation}\label{eq:w-p*}
\int_\Omega |u_j|^{p^\ast-2}u_j v\,dx\to \int_\Omega |u|^{p^\ast-2} u v\,dx\quad\mbox{as }j\to\infty.
\end{equation}
We can argue in a similar way to obtain for every $v\in W^{1,\,\mathcal A}_0(\Omega)$,
\begin{equation}\label{eq:w-q*}
\int_\Omega b(x)|u_j|^{q^\ast-2}u_j v\,dx\to \int_\Omega b(x) |u|^{q^\ast-2} u v\,dx\quad\mbox{as }j\to\infty.
\end{equation}
Finally, taking into account that $g(x,\cdot)$ is continuous and assumption \eqref{11}, also the term with $g(x,u_j)$ can be passed to the limit in the same way to get
for every $v\in W^{1,\,\mathcal A}_0(\Omega)$,
\begin{equation}\label{eq:w-g}
\int_\Omega g(x,u_j) v\,dx\to \int_\Omega g(x,u) v\,dx\quad\mbox{as }j\to\infty.
\end{equation}
Combining together \eqref{eq:w-gradp}-\eqref{eq:w-g}, we can pass to the limit in \eqref{eq:take-lim} to have the desired result.
\end{proof}

\subsection{Proofs of Proposition \ref{Proposition 3} and Proposition \ref{Proposition 4}}

\begin{proof}[Proof of Proposition \ref{Proposition 3}]
Let $\seq{u_j}$ be a \PS{\beta} sequence, i.e.,
\begin{equation} \label{12}
E(u_j) = \int_\Omega \left[\frac{1}{p}\, |\nabla u_j|^p + \frac{a(x)}{q}\, |\nabla u_j|^q - \frac{\mu}{p^\ast}\, |u_j|^{p^\ast} - \frac{b(x)}{q^\ast}\, |u_j|^{q^\ast} - G(x,u_j)\right] dx = \beta + \o(1)
\end{equation}
and
\begin{equation} \label{13}
E'(u_j)\, u_j = \int_\Omega \Big[|\nabla u_j|^p + a(x)\, |\nabla u_j|^q - \mu\, |u_j|^{p^\ast} - b(x)\, |u_j|^{q^\ast} - u_j\, g(x,u_j)\Big]\, dx = \o(\norm{u_j}).
\end{equation}
Dividing \eqref{13} by $\sigma$, with $\sigma$ as in \eqref{19}, subtracting from \eqref{12}, and using \eqref{19} gives
\[
\left(\frac{1}{p} - \frac{1}{\sigma}\right) \int_\Omega |\nabla u_j|^p\, dx + \left(\frac{1}{q} - \frac{1}{\sigma}\right) \int_\Omega a(x)\, |\nabla u_j|^q\, dx \le \o(\norm{u_j}) + \beta + c_3 \vol{\Omega} + \o(1),
\]
which together with \eqref{18} implies that $\seq{u_j}$ is bounded. Since $W^{1,\,\A}_0(\Omega)$ is reflexive, a renamed subsequence of $\seq{u_j}$ then converges weakly to some $u \in W^{1,\,\A}_0(\Omega)$. By Lemma \ref{Lemma 1}, $u$ is a weak solution of problem \eqref{10}.

It remains to show that $u$ is nontrivial. Suppose $u = 0$. Then the growth condition \eqref{11} implies by Proposition \ref{prop:subcritical} that
\begin{equation}\label{eq:uses-subcrit}
\int_\Omega G(x,u_j)\, dx \to 0, \qquad \int_\Omega u_j\, g(x,u_j)\, dx \to 0,
\end{equation}
so \eqref{12} and \eqref{13} reduce to
\begin{equation} \label{14}
\int_\Omega \left[\frac{1}{p}\, |\nabla u_j|^p + \frac{a(x)}{q}\, |\nabla u_j|^q - \frac{\mu}{p^\ast}\, |u_j|^{p^\ast} - \frac{b(x)}{q^\ast}\, |u_j|^{q^\ast}\right] dx = \beta + \o(1)
\end{equation}
and
\begin{equation} \label{20}
\int_\Omega \Big[|\nabla u_j|^p + a(x)\, |\nabla u_j|^q - \mu\, |u_j|^{p^\ast} - b(x)\, |u_j|^{q^\ast}\Big]\, dx = \o(1),
\end{equation}
respectively. Since $\seq{u_j}$ is bounded, \eqref{18}, \eqref{3}, and \eqref{4} imply that
\[
\int_\Omega |\nabla u_j|^p\, dx \to X, \quad \int_\Omega a(x)\, |\nabla u_j|^q\, dx \to Y, \quad \mu \int_\Omega |u_j|^{p^\ast} dx \to Z, \quad \int_\Omega b(x)\, |u_j|^{q^\ast} dx \to W
\]
for a renamed subsequence and some $X, Y, Z, W \ge 0$. Then \eqref{14} gives
\begin{equation} \label{24}
\beta = I(X,Y,Z,W),
\end{equation}
in particular, \eqref{21} holds, and \eqref{20} in the limit reduces to \eqref{22}. Now, we recall that 
\[
\mu \int_\Omega |u_j|^{p^\ast} dx \le \frac{\mu}{S_p^{p^\ast/p}} \left(\int_\Omega |\nabla u_j|^p\, dx\right)^{p^\ast/p}
\]
holds by \eqref{3}, then, passing to the limit gives that the first inequality in \eqref{23} holds for $Z$ and $X$. Furthermore, 
    \[
    \int_\Omega b(x)\, |u_j|^{q^\ast} dx \le \frac{b_\infty}{(a_0\, S_q)^{q^\ast/q}} \left[(1 + \eps) \int_\Omega a(x)\, |\nabla u_j|^q\, dx + C_\eps \int_U |u_j|^q\, dx\right]^{q^\ast/q}
    \]
by \eqref{1000}.  
Since $u_j \wto 0$ in $W^{1,\,\A}_0(\Omega)$, $u_j \to 0$ in $L^q(U)$ by Proposition \ref{prop:embeddings1}-($ii$) applied with $r=q<p^*$, so passing to the limit as $j\to\infty$ in the last inequality gives
    \[
    W \le \frac{b_\infty}{(a_0\, S_q)^{q^\ast/q}}\, [(1 + \eps)\, Y]^{q^\ast/q}.
    \]
    Now letting $\eps \to 0$ shows that the second inequality in \eqref{23} holds for $W$ and $Y$.
So $(X,Y,Z,W) \in S(\mu,b_\infty)$ and hence \eqref{24} gives $\beta \ge \beta^\ast(\mu,b_\infty)$, contrary to assumption.
\end{proof}

We stress the importance of the compactness of the embedding $W^{1,\A}(\Omega)\hookrightarrow\hookrightarrow L^\C(\Omega)$ for the proof of $u\neq0$ in the last part of the previous proof (see \eqref{eq:uses-subcrit}).

\begin{proof}[Proof of Proposition \ref{Proposition 4}]
Let $\seq{(X_j,Y_j,Z_j,W_j)} \subset S(\mu,b_\infty)$ be a minimizing sequence for $\beta^\ast(\mu,b_\infty)$. By \eqref{23} and \eqref{22},
\begin{multline} \label{25}
(X_j + Y_j) \left[1 - \frac{\mu}{S_p^{p^\ast/p}}\, X_j^{\frac{p^\ast}{p} - 1} - \frac{b_\infty}{(a_0\, S_q)^{q^\ast/q}}\, Y_j^{\frac{q^\ast}{q} - 1}\right] \le X_j + Y_j\\[7.5pt]
- \frac{\mu}{S_p^{p^\ast/p}}\, X_j^{p^\ast/p} - \frac{b_\infty}{(a_0\, S_q)^{q^\ast/q}}\, Y_j^{q^\ast/q} \le X_j + Y_j - Z_j - W_j = 0.
\end{multline}
If $X_j + Y_j = 0$, then $Z_j + W_j = 0$ also by \eqref{22}, so $X_j = Y_j = Z_j = W_j = 0$ and hence $I(X_j,Y_j,Z_j,W_j) = 0$, contradicting \eqref{21}$_j$ (we denote \eqref{21} applied to $(X_j,Y_j,Z_j,W_j)$ by \eqref{21}$_j$). So $X_j + Y_j > 0$ and hence \eqref{25} implies that
\begin{equation} \label{27}
\frac{\mu}{S_p^{N/(N-p)}}\, X_j^{p/(N-p)} + \frac{b_\infty}{(a_0\, S_q)^{N/(N-q)}}\, Y_j^{q/(N-q)} \ge 1.
\end{equation}

Dividing \eqref{22}$_j$ by $p^\ast$, subtracting from \eqref{26}$_j$, and combining with \eqref{27} gives
\begin{multline} \label{28}
I(X_j,Y_j,Z_j,W_j) = \frac{1}{N}\, X_j + \left(\frac{1}{q} - \frac{1}{p^\ast}\right) Y_j + \left(\frac{1}{p^\ast} - \frac{1}{q^\ast}\right) W_j\\[7.5pt]
\ge \frac{1}{N}\, \frac{S_p^{N/p}}{\mu^{(N-p)/p}} \left[1 - \frac{b_\infty}{(a_0\, S_q)^{N/(N-q)}}\, Y_j^{q/(N-q)}\right]^{(N-p)/p},
\end{multline}
where we used that $p<q<p^\ast<q^\ast$.
Since $S(\mu,0) \subset S(\mu,b_\infty)$, $\beta^\ast(\mu,b_\infty) \le \beta^\ast(\mu,0)$ and hence $I(X_j,Y_j,Z_j,W_j)$ is bounded uniformly in $b_\infty \ge 0$ for fixed $\mu > 0$, so the equality in \eqref{28} implies for $j$ large
\[
\left(\frac{1}{q} - \frac{1}{p^\ast}\right)Y_j\le I(X_j,Y_j,Z_j,W_j)\le \beta^\ast(\mu,b_\infty)+1\le \beta^\ast(\mu,0)+1,
\]
that is, $Y_j$ is bounded uniformly in $b_\infty$. So \eqref{15} follows from the last inequality in \eqref{28}.

Dividing \eqref{22}$_j$ by $q^\ast$, subtracting from \eqref{26}$_j$, and combining with \eqref{27} and \eqref{23}$_j$ gives
\begin{multline} \label{29}
I(X_j,Y_j,Z_j,W_j) = \left(\frac{1}{p} - \frac{1}{q^\ast}\right) X_j + \frac{1}{N}\, Y_j - \left(\frac{1}{p^\ast} - \frac{1}{q^\ast}\right) Z_j\\[7.5pt]
\ge \frac{1}{N}\, \frac{(a_0\, S_q)^{N/q}}{b_\infty^{(N-q)/q}} \left[1 - \frac{\mu}{S_p^{N/(N-p)}}\, X_j^{p/(N-p)}\right]^{(N-q)/q} - \left(\frac{1}{p^\ast} - \frac{1}{q^\ast}\right) \frac{\mu}{S_p^{p^\ast/p}}\, X_j^{p^\ast/p}.
\end{multline}
Since $S(0,b_\infty) \subset S(\mu,b_\infty)$, $\beta^\ast(\mu,b_\infty) \le \beta^\ast(0,b_\infty)$ and hence $I(X_j,Y_j,Z_j,W_j)$ is bounded uniformly in $\mu \ge 0$ for fixed $b_\infty > 0$, so the first equality in \eqref{28} implies that $X_j$ is bounded uniformly in $\mu$. So \eqref{16} follows from \eqref{29}.
\end{proof}

\section{Proofs of existence results in the local case}\label{sec4}

\subsection{Proof of Theorem \ref{Theorem 1}}
The variational functional associated with problem \eqref{6} is
\[
E(u) = \int_\Omega \left[\frac{1}{p}\, |\nabla u|^p + \frac{a(x)}{q}\, |\nabla u|^q - \frac{\lambda}{r}\, |u|^r - \frac{\mu}{p^\ast}\, |u|^{p^\ast} - \frac{b(x)}{q^\ast}\, |u|^{q^\ast}\right] dx, \quad u \in W^{1,\,\A}_0(\Omega).
\]
When $r = p$, \eqref{7}, \eqref{3}, and \eqref{4} give
\begin{multline} \label{30}
E(u) \ge \frac{1}{p} \left(1 - \frac{\lambda}{\lambda_1(p)}\right) \int_\Omega |\nabla u|^p\, dx + \frac{1}{q} \int_\Omega a(x)\, |\nabla u|^q\, dx\\[4.5pt] - \frac{\mu}{p^\ast\, S_p^{p^\ast/p}} \left(\int_\Omega |\nabla u|^p\, dx\right)^{p^\ast/p}
- \frac{\kappa}{q^\ast} \left(\int_\Omega a(x)\, |\nabla u|^q\, dx\right)^{q^\ast/q}.
\end{multline}
For $\norm{u} \le 1$,
\[
\norm{u}^q \le \int_\Omega \big(|\nabla u|^p + a(x)\, |\nabla u|^q\big)\, dx \le \norm{u}^p
\]
by \eqref{18} and hence, using that $0<\lambda<\lambda_1(p)$, \eqref{30} gives
\[
E(u) \ge \min \set{\frac{1}{p} \left(1 - \frac{\lambda}{\lambda_1(p)}\right),\frac{1}{q}} \norm{u}^q - \frac{\mu}{p^\ast\, S_p^{p^\ast/p}} \norm{u}^{p^\ast} - \frac{\kappa}{q^\ast} \norm{u}^{p\, q^\ast/q}.
\]
Since $q < p^\ast < p\, q^\ast/q$, it follows from this that the origin is a strict local minimizer of $E$ when $0 < \lambda < \lambda_1(p)$. A similar argument using $W^{1,\A}_0(\Omega)\hookrightarrow L^r(\Omega)$ (see Proposition \ref{prop:embeddings1}-($ii$)) shows that when $r > p$, the origin is a strict local minimizer of $E$ for all $\lambda > 0$. On the other hand, $E(tu) \to - \infty$ as $t \to + \infty$ for any $u \in W^{1,\,\A}_0(\Omega) \setminus \set{0}$. So $E$ has the mountain pass geometry.

Let
\[
\beta := \inf_{\gamma \in \Gamma}\, \max_{u \in \gamma([0,1])}\, E(u) > 0
\]
be the mountain pass level, where
\[
\Gamma = \big\{\gamma \in C([0,1],W^{1,\,\A}_0(\Omega)) : \gamma(0) = 0,\, E(\gamma(1)) < 0\big\}
\]
is the class of paths in $W^{1,\,\A}_0(\Omega)$ joining the origin to the set $\big\{u \in W^{1,\,\A}_0(\Omega) : E(u) < 0\big\}$. A standard deformation argument shows that $E$ has a \PS{\beta} sequence $\seq{u_j}$. We will show that for all $\mu > 0$ and sufficiently small $b_\infty \ge 0$, in each of the cases in Theorem \ref{Theorem 1},
\begin{equation} \label{31}
\beta < \beta^\ast(\mu,b_\infty)
\end{equation}
and hence $\seq{u_j}$ has a subsequence that converges weakly to a nontrivial weak solution of problem \eqref{6} by Proposition \ref{Proposition 3}.

For any $u \in W^{1,\,\A}_0(\Omega) \setminus \set{0}$, there exists $t_u > 0$ such that $E(t_u u) < 0$ since $E(tu) \to - \infty$ as $t \to + \infty$. Then the line segment $\set{t_u\, tu : 0 \le t \le 1}$ belongs to $\Gamma$ and hence
\[
\beta \le \max_{0 \le t \le 1}\, E(t_u\, tu) \le \max_{t \ge 0}\, E(tu).
\]
So, to show that \eqref{31} holds, it suffices to show that
\begin{equation} \label{32}
\max_{t \ge 0}\, E(t u_0) < \beta^\ast(\mu,b_\infty)
\end{equation}
for some $u_0 \in W^{1,\,\A}_0(\Omega) \setminus \set{0}$. To construct such a function $u_0$, take $x_0 = 0$ for the sake of simplicity, let $\psi \in C^\infty_0(B_\rho(0))$ be a cut-off function such that $0 \le \psi \le 1$ and $\psi = 1$ on $B_{\rho/2}(0)$, and set
\[
u_\eps(x) = \frac{\psi(x)}{\left(\eps^{p/(p-1)} + |x|^{p/(p-1)}\right)^{(N-p)/p}}, \qquad v_\eps(x) = \frac{u_\eps(x)}{\pnorm[p^\ast]{u_\eps}}
\]
for $\eps > 0$. We will show that \eqref{32} holds for $u_0 = v_\eps$ with $\eps > 0$ sufficiently small.

By Dr{\'a}bek and Huang \cite{MR1473856}, we have the estimates
\begin{equation} \label{34}
\int_\Omega |\nabla v_\eps|^p\, dx = S_p + O\big(\eps^{(N-p)/(p-1)}\big)
\end{equation}
and
\begin{equation} \label{35}
\int_\Omega v_\eps^r\, dx = \begin{cases}
O\big(\eps^{[Np-(N-p)r]/p}\big), & r > \frac{N(p - 1)}{N - p}\\[7.5pt]
O\big(\eps^{N/p}\, |\log \eps|\big), & r = \frac{N(p - 1)}{N - p}\\[7.5pt]
O\big(\eps^{(N-p)r/p(p-1)}\big), & r < \frac{N(p - 1)}{N - p}
\end{cases}
\end{equation}
as $\eps \to 0$.

\begin{lemma} \label{Lemma 3}
If the following limit holds
\begin{equation} \label{217}
\lim_{\eps\to 0}\frac{\eps^{(N-p)/(p-1)}}{\dint_\Omega v_\eps^r\, dx} = 0,
\end{equation}
then there exist $\eps_0, b^\ast > 0$ such that
\[
\max_{t \ge 0}\, E(tv_{\eps_0}) < \beta^\ast(\mu,b_\infty)
\]
when $\mu > 0$ and $0 \le b_\infty < b^\ast$.
\end{lemma}

\begin{proof}
Since \eqref{2} implies that $\supp(b) \subset \supp(a)$, $a = 0 = b$ a.e.\! in $B_\rho(0)$ by \eqref{33}. Since $\supp(v_\eps) \subset B_\rho(0)$ and $\pnorm[p^\ast]{v_\eps} = 1$, this gives
\[
E(tv_\eps) = \frac{t^p}{p} \int_\Omega |\nabla v_\eps|^p\, dx - \frac{\lambda t^r}{r} \int_\Omega v_\eps^r\, dx - \frac{\mu t^{p^\ast}}{p^\ast} =: \varphi_\eps(t).
\]
So, in view of \eqref{15}, it suffices to show that
\[
\max_{t \ge 0}\, \varphi_\eps(t) < \frac{1}{N}\, \frac{S_p^{N/p}}{\mu^{(N-p)/p}}
\]
for all sufficiently small $\eps > 0$. Suppose this is false. Then there are sequences $\eps_j \to 0$ and $t_j > 0$ such that
\begin{equation} \label{218}
\varphi_{\eps_j}(t_j) = \frac{t_j^p}{p} \int_\Omega |\nabla v_j|^p\, dx - \frac{\lambda t_j^r}{r} \int_\Omega v_j^r\, dx - \frac{\mu t_j^{p^\ast}}{p^\ast} \ge \frac{1}{N}\, \frac{S_p^{N/p}}{\mu^{(N-p)/p}}
\end{equation}
and
\begin{equation} \label{219}
t_j\, \varphi_{\eps_j}'(t_j) = t_j^p \int_\Omega |\nabla v_j|^p\, dx - \lambda t_j^r \int_\Omega v_j^r\, dx - \mu t_j^{p^\ast} = 0,
\end{equation}
where $v_j = v_{\eps_j}$. By \eqref{34} and \eqref{35},
\[
\int_\Omega |\nabla v_j|^p\, dx \to S_p, \qquad \int_\Omega v_j^r\, dx \to 0.
\]
So \eqref{218} implies that the sequence $\seq{t_j}$ is bounded and hence converges to some $t_0 > 0$ for a renamed subsequence. Passing to the limit in \eqref{219} gives
\begin{equation} \label{220}
S_p\, t_0^p - \mu\, t_0^{p^\ast} = 0,
\end{equation}
so $t_0 = \left(S_p/\mu\right)^{(N-p)/p^2}$.
Subtracting \eqref{220} from \eqref{219} and using \eqref{34} gives
\[
S_p\, \big(t_j^p - t_0^p\big) - \lambda t_j^r \int_\Omega v_j^r\, dx - \mu\, \big(t_j^{p^\ast} - t_0^{p^\ast}\big) = O\big(\eps_j^{(N-p)/(p-1)}\big).
\]
Then
\begin{equation} \label{221}
\left(p\, S_p\, \sigma_j^{p-1} - p^\ast \mu\, \tau_j^{p^\ast - 1}\right) (t_j - t_0) = \lambda t_j^r \int_\Omega v_j^r\, dx + O\big(\eps_j^{(N-p)/(p-1)}\big)
\end{equation}
for some $\sigma_j,\, \tau_j$ between $t_0$ and $t_j$ by the mean value theorem. Since $t_j \to t_0$, $\sigma_j, \tau_j \to t_0$. So
\[
p\, S_p\, \sigma_j^{p-1} - p^\ast \mu\, \tau_j^{p^\ast - 1} \to p\, S_p\, t_0^{p-1} - p^\ast \mu\, t_0^{p^\ast - 1} = - (p^\ast - p)\, \mu\, t_0^{p^\ast - 1} < 0
\]
by \eqref{220}. Thus \eqref{221} and \eqref{217} imply that $t_j \le t_0$ for all sufficiently large $j$. Dividing \eqref{219} by $p^\ast$, subtracting from \eqref{218}, and using \eqref{34} and \eqref{220} gives
\[
\frac{1}{N}\, S_p\, t_j^p - \lambda \left(\frac{1}{r} - \frac{1}{p^\ast}\right) t_j^r \int_\Omega v_j^r\, dx \ge \frac{1}{N}\, S_p\, t_0^p + O\big(\eps_j^{(N-p)/(p-1)}\big).
\]
This together with $t_j \le t_0$ and \eqref{217} gives
\[
\lambda \left(\frac{1}{r} - \frac{1}{p^\ast}\right) t_0^r \le 0,
\]
a contradiction since $\lambda, t_0 > 0$ and $r < p^\ast$.
\end{proof}

In view of Lemma \ref{Lemma 3}, it suffices to show that \eqref{217} holds in each of the cases in Theorem~\ref{Theorem 1} to complete its proof. Equation \eqref{35} gives us the estimate
\begin{equation} \label{36}
\frac{\eps^{(N-p)/(p-1)}}{\dint_\Omega v_\eps^r\, dx} = \begin{cases}
O\big(\eps^{[(N-p)(p-1)r-(Np-2N+p)p]/p(p-1)}\big), & r > \frac{N(p - 1)}{N - p}\\[7.5pt]
O\big(\eps^{(N-p^2)/p(p-1)}/|\log \eps|\big), & r = \frac{N(p - 1)}{N - p}\\[7.5pt]
O\big(\eps^{(N-p)(p-r)/p(p-1)}\big), & r < \frac{N(p - 1)}{N - p}
\end{cases}
\end{equation}
as $\eps \to 0$.

({\em i}) Let $N \ge p^2$ and $r = p$. Then \eqref{36} gives
\[
\frac{\eps^{(N-p)/(p-1)}}{\dint_\Omega v_\eps^r\, dx} = \begin{cases}
O\big(\eps^{(N-p^2)/(p-1)}\big), & N > p^2\\[7.5pt]
O\big(1/|\log \eps|\big), & N = p^2
\end{cases}\qquad\mbox{as }\eps\to 0,
\]
so \eqref{217} follows.

({\em ii}) Let $N \ge p^2$ and $p < r < p^\ast$. Then $p \ge N(p - 1)/(N - p)$ and hence $r > N(p - 1)/(N - p)$, so \eqref{36} gives
\[
\frac{\eps^{(N-p)/(p-1)}}{\dint_\Omega v_\eps^r\, dx} = O\big(\eps^{[(N-p)(p-1)r-(Np-2N+p)p]/p(p-1)}\big)\qquad\mbox{as }\eps\to 0.
\]
Since $(N - p)(p - 1)\, r - (Np - 2N + p)\, p > (N - p^2)\, p \ge 0$, \eqref{217} follows.

({\em iii}) Let $N < p^2$ and $(Np - 2N + p)\, p/(N - p)(p - 1) < r < p^\ast$. Then $(Np - 2N + p)\, p/(N - p)(p - 1) > N(p - 1)/(N - p)$ and hence $r > N(p - 1)/(N - p)$, so \eqref{36} gives
\[
\frac{\eps^{(N-p)/(p-1)}}{\dint_\Omega v_\eps^r\, dx} = O\big(\eps^{[(N-p)(p-1)r-(Np-2N+p)p]/p(p-1)}\big)\qquad\mbox{as }\eps\to 0.
\]
Since $(N - p)(p - 1)\, r - (Np - 2N + p)\, p > 0$ by the hypothesis on $r$, \eqref{217} follows.\hfill \qedsymbol

\subsection{Proof of Theorem \ref{Theorem 2}}
The variational functional associated with problem \eqref{8} is
\[
E(u) = \int_\Omega \left[\frac{1}{p}\, |\nabla u|^p + \frac{a(x)}{q}\, |\nabla u|^q - \frac{c(x)}{s}\, |u|^s - \frac{\mu}{p^\ast}\, |u|^{p^\ast} - \frac{b(x)}{q^\ast}\, |u|^{q^\ast}\right] dx, \quad u \in W^{1,\,\A}_0(\Omega).
\]
Fix $p < r < p^\ast$ and let
\[
\C(x,t) = t^r + c(x)\, t^s, \quad (x,t) \in \Omega \times [0,\infty).
\]
Then
\[
\int_\Omega c(x)\, |u|^s\, dx \le \int_\Omega \big(|u|^r + c(x)\, |u|^s\big)\, dx \le \max \set{\norm[\C]{u}^r,\norm[\C]{u}^s} \le \norm[\C]{u}^r
\]
for all $u \in L^\C(\Omega)$ with $\norm[\C]{u} \le 1$, and $W^{1,\,\A}_0(\Omega)$ is continuously embedded in $L^\C(\Omega)$ by Proposition \ref{prop:subcritical}, so an argument similar to that in the proof of Theorem \ref{Theorem 1} shows that $E$ has the mountain pass geometry. As in that proof, it now suffices to show that for all $b_\infty > 0$ and sufficiently small $\mu \ge 0$,
\begin{equation} \label{38}
\max_{t \ge 0}\, E(t u_0) < \beta^\ast(\mu,b_\infty)
\end{equation}
for some $u_0 \in W^{1,\,\A}_0(\Omega) \setminus \set{0}$. Take $x_0 = 0$, let $\psi \in C^\infty_0(B_\rho(0))$ be a cut-off function such that $0 \le \psi \le 1$ and $\psi = 1$ on $B_{\rho/2}(0)$, and set
\[
u_{\eps,\delta}(x) = \frac{\psi(x/\delta)}{\left(\eps^{q/(q-1)} + |x|^{q/(q-1)}\right)^{(N-q)/q}}, \qquad v_{\eps,\delta}(x) = \frac{u_{\eps,\delta}(x)}{\pnorm[q^\ast]{u_{\eps,\delta}}}
\]
for $\eps > 0$ and $0 < \delta \le 1$. We will show that \eqref{38} holds for $u_0 = v_{\eps,\delta}$ with suitably chosen $\eps, \delta > 0$.

We have the estimates
\begin{equation} \label{214}
\int_\Omega |\nabla v_{\eps,\delta}|^q\, dx = S_q + O\big((\eps/\delta)^{(N-q)/(q-1)}\big),
\end{equation}
\begin{equation} \label{215}
\int_\Omega |\nabla v_{\eps,\delta}|^p\, dx = \begin{cases}
O\big(\eps^{N(q-p)/q}\big), & p > \frac{N(q - 1)}{N - 1}\\[7.5pt]
O\big(\eps^{N(N-q)/(N-1)q}\, |\log\, (\eps/\delta)|\big), & p = \frac{N(q - 1)}{N - 1}\\[7.5pt]
O\big(\eps^{(N-q)p/q(q-1)}\, \delta^{[N(q-1)-(N-1)p]/(q-1)}\big), & p < \frac{N(q - 1)}{N - 1},
\end{cases}
\end{equation}
and
\begin{equation} \label{216}
\int_\Omega v_{\eps,\delta}^s\, dx = \begin{cases}
O\big(\eps^{[Nq-(N-q)s]/q}\big), & s > \frac{N(q - 1)}{N - q}\\[7.5pt]
O\big(\eps^{N/q}\, |\log\, (\eps/\delta)|\big), & s = \frac{N(q - 1)}{N - q}\\[7.5pt]
O\big(\eps^{(N-q)s/q(q-1)}\, \delta^{[N(q-1)-(N-q)s]/(q-1)}\big), & s < \frac{N(q - 1)}{N - q}
\end{cases}
\end{equation}
as $\eps \to 0$ and $\eps/\delta \to 0$ (see Ho et al.\! \cite[Lemma 3.2]{HoPeSi}).

\begin{lemma} \label{Lemma 4}
If $\eps/\delta \to 0$ and the following limits hold
\begin{equation} \label{222}
\lim_{\eps\to 0}\frac{\dint_\Omega |\nabla v_{\eps,\delta}|^p\, dx}{\dint_\Omega v_{\eps,\delta}^s\, dx} = 0, \qquad \lim_{\eps\to 0}\frac{(\eps/\delta)^{(N-q)/(q-1)}}{\dint_\Omega v_{\eps,\delta}^s\, dx} = 0,
\end{equation}
then there exist $\eps_0,\, \delta_0,\, \mu^\ast > 0$ such that
\[
\max_{t \ge 0}\, E(tv_{\eps_0,\delta_0}) < \beta^\ast(\mu,b_\infty)
\]
when $0 \le \mu < \mu^\ast$ and $b_\infty > 0$.
\end{lemma}

\begin{proof}
Since $\supp(v_{\eps,\delta}) \subset B_\rho(0)$ and $\pnorm[q^\ast]{v_{\eps,\delta}} = 1$, \eqref{9} gives
\[
E(tv_{\eps,\delta}) \le \frac{t^p}{p} \int_\Omega |\nabla v_{\eps,\delta}|^p\, dx + \frac{a_0\, t^q}{q} \int_\Omega |\nabla v_{\eps,\delta}|^q\, dx - \frac{c_0\, t^s}{s} \int_\Omega v_{\eps,\delta}^s\, dx - \frac{b_\infty\, t^{q^\ast}}{q^\ast} =: \varphi_\eps(t)
\]
for all $\mu \ge 0$. So, in view of \eqref{16}, it suffices to show that
\[
\max_{t \ge 0}\, \varphi_\eps(t) < \frac{1}{N}\, \frac{(a_0\, S_q)^{N/q}}{b_\infty^{(N-q)/q}}
\]
for all sufficiently small $\eps > 0$. Suppose this is false. Then there are sequences $\eps_j \to 0$ and $t_j > 0$ such that
\begin{equation} \label{2180}
\varphi_{\eps_j}(t_j) = \frac{t_j^p}{p} \int_\Omega |\nabla v_j|^p\, dx + \frac{a_0\, t_j^q}{q} \int_\Omega |\nabla v_j|^q\, dx - \frac{c_0\, t_j^s}{s} \int_\Omega v_j^s\, dx - \frac{b_\infty\, t_j^{q^\ast}}{q^\ast} \ge \frac{1}{N}\, \frac{(a_0\, S_q)^{N/q}}{b_\infty^{(N-q)/q}}
\end{equation}
and
\begin{equation} \label{2190}
t_j\, \varphi_{\eps_j}'(t_j) = t_j^p \int_\Omega |\nabla v_j|^p\, dx + a_0\, t_j^q \int_\Omega |\nabla v_j|^q\, dx - c_0\, t_j^s \int_\Omega v_j^s\, dx - b_\infty\, t_j^{q^\ast} = 0,
\end{equation}
where $v_j = v_{\eps_j,\delta_j}$ and $\delta_j = \delta(\eps_j)$. By \eqref{214}--\eqref{216},
\[
\int_\Omega |\nabla v_j|^q\, dx \to S_q, \qquad \int_\Omega |\nabla v_j|^p\, dx \to 0, \qquad \int_\Omega v_j^s\, dx \to 0.
\]
So \eqref{2180} implies that the sequence $\seq{t_j}$ is bounded and hence converges to some $t_0 > 0$ for a renamed subsequence. Passing to the limit in \eqref{2190} gives
\begin{equation} \label{2200}
a_0\, S_q\, t_0^q - b_\infty\, t_0^{q^\ast} = 0,
\end{equation}
so $t_0 = \left(a_0\, S_q/b_\infty\right)^{(N-q)/q^2}$.
Subtracting \eqref{2200} from \eqref{2190} and using \eqref{214} gives
\[
t_j^p \int_\Omega |\nabla v_j|^p\, dx + a_0\, S_q\, \big(t_j^q - t_0^q\big) - c_0\, t_j^s \int_\Omega v_j^s\, dx - b_\infty\, \big(t_j^{q^\ast} - t_0^{q^\ast}\big) = O\big((\eps_j/\delta_j)^{(N-q)/(q-1)}\big).
\]
Then
\begin{multline} \label{2210}
\left(q a_0\, S_q\, \sigma_j^{q-1} - q^\ast\, b_\infty\, \tau_j^{q^\ast - 1}\right) (t_j - t_0) = c_0\, t_j^s \int_\Omega v_j^s\, dx - t_j^p \int_\Omega |\nabla v_j|^p\, dx\\[7.5pt]
+ O\big((\eps_j/\delta_j)^{(N-q)/(q-1)}\big)
\end{multline}
for some $\sigma_j,\, \tau_j$ between $t_0$ and $t_j$ by the mean value theorem. Since $t_j \to t_0$, $\sigma_j, \tau_j \to t_0$ and hence
\[
q a_0\, S_q\, \sigma_j^{q-1} - q^\ast\, b_\infty\, \tau_j^{q^\ast - 1} \to q a_0\, S_q\, t_0^{q-1} - q^\ast\, b_\infty\, t_0^{q^\ast - 1} = - (q^\ast - q)\, b_\infty\, t_0^{q^\ast - 1} < 0
\]
by \eqref{2200}. So \eqref{2210} and \eqref{222} imply that $t_j \le t_0$ for all sufficiently large $j$. Dividing \eqref{2190} by $q^\ast$, subtracting from \eqref{2180}, and using \eqref{214} and \eqref{2200} gives
\begin{multline*}
\left(\frac{1}{p} - \frac{1}{q^\ast}\right) t_j^p \int_\Omega |\nabla v_j|^p\, dx + \frac{1}{N}\, a_0\, S_q\, t_j^q - c_0 \left(\frac{1}{s} - \frac{1}{q^\ast}\right) t_j^s \int_\Omega v_j^s\, dx \ge \frac{1}{N}\, a_0\, S_q\, t_0^q\\[7.5pt]
+ O\big((\eps_j/\delta_j)^{(N-q)/(q-1)}\big).
\end{multline*}
This together with $t_j \le t_0$ and \eqref{222} gives $c_0 \left(\frac{1}{s} - \frac{1}{q^\ast}\right) t_0^s \le 0$,
a contradiction since $c_0, t_0 > 0$ and $s < q^\ast$.
\end{proof}

In view of Lemma \ref{Lemma 4}, it only remains to find a $\delta = \delta(\eps) \in (0,1]$ such that $\eps/\delta \to 0$ and \eqref{222} holds as $\eps \to 0$ in each of the two cases in Theorem \ref{Theorem 2}.

({\em i}) Let $1 < p < N(q - 1)/(N - 1)$ and $N^2 (q - 1)/(N - 1)(N - q) < s < q^\ast$. We take $\delta = \eps^\kappa$, where $\kappa \in [0,1)$ is to be determined. Since
\[
s > \frac{N^2 (q - 1)}{(N - 1)(N - q)} > \frac{N(q - 1)}{N - q},
\]
\eqref{215} and \eqref{216} give
\[
\frac{\dint_{\R^N} |\nabla v_{\eps,\delta}|^p\, dx}{\dint_{\R^N} v_{\eps,\delta}^s\, dx} 
 = O\left(\eps_j^{\frac{1}{q}[(N-q)(s+\frac{p}{q-1})-Nq]+\frac{\kappa}{q-1}[N(q-1)-(N-1)p]}\right)
= O\left(\eps^{\frac{\kappa - \underline{\kappa}}{q-1}[N(q-1)-(N-1)p]}\right),
\]
where
\[
\underline{\kappa} = \frac{Nq(q-1)-(N-q)(q-1)s-(N-q)p}{[N(q-1)-(N-1)p]q},
\]
and \eqref{216} gives
\[
\frac{(\eps/\delta)^{(N-q)/(q-1)}}{\dint_{\R^N} v_{\eps,\delta}^s\, dx} 
= O\left(\eps^{\frac{1}{q(q-1)}[(N-q)(q-1)s-(Nq-2N+q)q]-\frac{\kappa(N-q)}{q-1}}\right)
= O\left(\eps^{\frac{\overline{\kappa} - \kappa}{q-1}(N-q)}\right),
\]
where
\[
\overline{\kappa} = \frac{(N-q)(q-1)s-(Nq-2N+q)q}{(N-q)q}.
\]
We want to choose $\kappa \in [0,1)$ so that $\kappa > \underline{\kappa}$ and $\kappa < \overline{\kappa}$. This is possible if and only if $\underline{\kappa} < \overline{\kappa}$, $\underline{\kappa} < 1$, and $\overline{\kappa} > 0$. Tedious calculations show that these inequalities are equivalent to
\[
s > \frac{N^2 (q - 1)}{(N - 1)(N - q)},\quad
s > \frac{Np}{N - q},\quad\mbox{and}\quad
s > \frac{N^2 (q - 1)}{(N - 1)(N - q)} - \frac{N - q}{(N - 1)(q - 1)},
\]
respectively, all of which hold under our assumptions on $p$ and $s$.

({\em ii}) Let $N(q - 1)/(N - 1) \le p < q$ and $Np/(N - q) < s < q^\ast$. We take $\delta = 1$. Since
\[
s > \frac{Np}{N - q} \ge \frac{N^2 (q - 1)}{(N - 1)(N - q)} > \frac{N(q - 1)}{N - q},
\]
\eqref{215} and \eqref{216} give
\[
\frac{\dint_{\R^N} |\nabla v_{\eps,\delta}|^p\, dx}{\dint_{\R^N} v_{\eps,\delta}^s\, dx} = \begin{cases}
O\big(\eps^{[(N-q)s-Np]/q}\big), & p > \frac{N(q - 1)}{N - 1}\\[7.5pt]
O\big(\eps^{[(N-q)s-Np]/q}\, |\log \eps|\big), & p = \frac{N(q - 1)}{N - 1}
\end{cases}
\]
and
\[
\frac{(\eps/\delta)^{(N-q)/(q-1)}}{\dint_{\R^N} v_{\eps,\delta}^s\, dx} = O\left(\eps^{\frac{1}{q(q-1)}[(N-q)(q-1)s-(Nq-2N+q)q]}\right).
\]
Since $s > Np/(N - q)$, the first limit in \eqref{222} holds. The second limit also holds since
\[
\frac{Np}{N - q} \ge \frac{N^2 (q - 1)}{(N - 1)(N - q)} > \frac{(Nq - 2N + q) q}{(N - q)(q - 1)}. \hfill \qedsymbol
\]

\section{Proofs of nonexistence results for the local case}\label{sec6}

In this section we prove the Poho\v{z}aev type identity \eqref{eq:poho} and derive subsequent nonexistence results for the double phase problem \eqref{1}, which we report here for the reader's convenience
\begin{equation}\label{eq:P-poho}
\left\{\begin{aligned}
- \divg \left(|\nabla u|^{p-2}\, \nabla u + a(x)\, |\nabla u|^{q-2}\, \nabla u\right) & = f(x,u) && \text{in } \Omega\\[10pt]
u & = 0 && \text{on } \bdry{\Omega},
\end{aligned}\right.
\end{equation}
where $\Omega\subset\mathbb R^N$ is a $C^1$ bounded domain, $1 < p < q < N$, $q/p \le N/(N-1)$, and $0\le a \in C^1(\Omega)$.

\subsection{Proof of Theorem \ref{thm:poho}}
Multiplying the left-hand side of the equation in \eqref{eq:P-poho} by $(x\cdot\nabla u)$, and integrating by parts over $\Omega$, we get by straightforward calculations
\begin{equation}\label{eq:poho-LHS}
\begin{aligned}
& - \int_\Omega\divg \left(|\nabla u|^{p-2}\, \nabla u + a(x)\, |\nabla u|^{q-2}\, \nabla u\right)(x\cdot\nabla u)\,dx\\
& \hspace{1cm}= \int_\Omega\left[\left(1-\frac{N}{p}\right)|\nabla u|^p+\left(1-\frac{N}{q}\right)a(x)|\nabla u|^q\right]\,dx - \int_\Omega \frac{|\nabla u|^q}{q}(\nabla a\cdot x)\,dx\\
& \hspace{5cm}-\int_{\partial\Omega}\left[\left(1-\frac{1}{p}\right)|\nabla u|^p+\left(1-\frac{1}{q}\right)a(x)|\nabla u|^q\right](x\cdot\nu)d\sigma.
\end{aligned}
\end{equation}
Then, performing the same operations on the right-hand side, we have by straightforward calculations
\begin{equation}\label{eq:poho-RHS}
 - \int_\Omega f(x,u)\,(x\cdot\nabla u)\,dx = -N\int_\Omega F(x,u)\, dx.
\end{equation}
Hence, using that $u$ solves the equation in \eqref{eq:P-poho}, by \eqref{eq:poho-LHS} and \eqref{eq:poho-RHS}, we obtain the following identity
\begin{equation}\label{eq:poho-RHS-LHS}
\begin{aligned}
&\int_\Omega\left(\frac{1}{p^\ast}|\nabla u|^p +\frac{1}{q^\ast} a(x)|\nabla u|^q\right) dx+ \frac{1}{Nq}\int_\Omega |\nabla u|^q(\nabla a\cdot x)\,dx\\
&\hspace{3cm}+\frac{1}{N}\int_{\partial\Omega}\left[\left(1-\frac{1}{p}\right)|\nabla u|^p+\left(1-\frac{1}{q}\right)a(x)|\nabla u|^q\right](x\cdot\nu)d\sigma\\
&\hspace{10.5cm} = \int_\Omega F(x,u)\, dx
\end{aligned}
\end{equation}
Now, multiplying the equation in \eqref{eq:P-poho} by $u/q^\ast$ and integrating by parts over $\Omega$, we also have
\begin{equation}\label{eq:poho-aux}
\frac{1}{q^\ast}\int_\Omega\left(|\nabla u|^p+a(x)|\nabla u|^q\right)\,dx=\frac{1}{q^\ast}\int_\Omega f(x,u)u\,dx
\end{equation}
Finally, subtracting \eqref{eq:poho-aux} from \eqref{eq:poho-RHS-LHS} and observing that $1/p^\ast-1/q^\ast=1/p-1/q$ and $|\nabla u|=|\partial_\nu u|$ over $\partial\Omega$, in view of the homogeneous boundary conditions, we obtain the desired identity \eqref{eq:poho} and conclude the proof.

\subsection{Proofs of Theorems \ref{thm:nonex} and \ref{thm:small}}
In the present subsection we assume further that $\Omega$ is starshaped. Without loss of generality we suppose that $0\in \Omega$ and $\Omega$ is starshaped with respect to the origin, and so $(x\cdot\nu)\ge 0$ on $\partial\Omega$. Being $a(x)$ nonnegative, the boundary term in \eqref{eq:poho} is nonnegative, and the Poho\v{z}aev type identity \eqref{eq:poho} gives the inequality
\begin{equation}\label{eq:poho-ineq}
\left(\frac{1}{p}-\frac{1}{q}\right)\int_\Omega|\nabla u|^p dx +\frac{1}{Nq}\int_\Omega|\nabla u|^q(\nabla a\cdot x)dx \le \int_\Omega\left(F(x,u)-\frac{1}{q^\ast}f(x,u)u\right)dx.
\end{equation}
We are now ready to prove the nonexistence theorem.
\begin{proof}[Proof of Theorem \ref{thm:nonex}]
Being $a(x)$ radial and radially nondecreasing, $(\nabla a\cdot x)\ge 0$. Thus, by \eqref{eq:poho-ineq}, we have
\begin{equation}\label{eq:ineq-poho-special}
\left(\frac{1}{p}-\frac{1}{q}\right)\int_\Omega|\nabla u|^p dx\le \int_\Omega\left[c(x)\left(\frac{1}{r}-\frac{1}{q^\ast}\right)|u|^r+\mu\left(\frac{1}{p^\ast}-\frac{1}{q^\ast}\right)|u|^{p^\ast}\right]dx.
\end{equation}
In case $(i)$ this immediately gives $\left(\frac{1}{p}-\frac{1}{q}\right)\int_\Omega|\nabla u|^p dx\le 0$, that is $u \equiv 0$.
In case $(ii)$, by inequality \eqref{eq:ineq-poho-special} we can infer that
\[
\left(\frac{1}{p}-\frac{1}{q}\right)\int_\Omega|\nabla u|^p dx \le c_{\infty}\left(\frac{1}{p}-\frac{1}{q^\ast}\right) \int_\Omega |u|^p\,dx.
\]
Hence, using the variational characterization of the first eigenvalue $\lambda_1(p)$ of the $p$-Laplacian, we have
\[
\left[\left(\frac{1}{p}-\frac{1}{q}\right)-\frac{c_{\infty}}{\lambda_1(p)}\left(\frac{1}{p}-\frac{1}{q^\ast}\right)\right]\int_\Omega|\nabla u|^p dx\le 0,
\]
which again implies $u \equiv 0$, by the assumption on $c_{\infty}$.
Finally, for $(iii)$, we go back to the Poho\v{z}aev type identity \eqref{eq:poho} and observe that in this case, using again the variational characterization of $\lambda_1(p)$ and the assumptions on $\mu$ and $c_{\infty}$, it implies
\[
\begin{aligned}
&\left(\frac{1}{p}-\frac{1}{q}\right)\int_\Omega|\nabla u|^p\, dx +\frac{1}{N}\int_{\partial\Omega}\left[\left(1-\frac{1}{p}\right)|\partial_\nu u|^p+\left(1-\frac{1}{q}\right)a(x)|\partial_\nu u|^q\right](x\cdot\nu)\,d\sigma\\
&\hspace{5cm} \le c_\infty \left(\frac{1}{p}-\frac{1}{q^\ast}\right) \int_\Omega |u|^p\,dx \le \left(\frac{1}{p}-\frac{1}{q}\right)\int_\Omega|\nabla u|^p\, dx.
\end{aligned}
\]
Therefore,
\[
\int_{\partial\Omega}\left[\left(1-\frac{1}{p}\right)|\partial_\nu u|^p+\left(1-\frac{1}{q}\right)a(x)|\partial_\nu u|^q\right](x\cdot\nu)\,d\sigma\le 0,
\]
which in view of the strict starshapedness of $\Omega$ (i.e., $(x\cdot\nu)>0$ on $\partial\Omega$) forces
\begin{equation}\label{eq:partialnu=0}
\partial_\nu u=0\quad\mbox{ on }\partial\Omega.
\end{equation}
Since, by the explicit expression for $c_\infty$ given in the hypothesis, it holds
\[
\left(\frac{1}{p}-\frac{1}{q}\right)\int_\Omega|\nabla u|^p dx\le c_\infty\left(\frac{1}{p}-\frac{1}{q^\ast}\right)\int_\Omega|u|^p dx = \left(\frac{1}{p}-\frac{1}{q}\right)\lambda_1(p)\int_\Omega|u|^p dx.
\]
By the variational characterization of $\lambda_1(p)$, if $u$ is nonzero, it must be the first eigenvalue of the $p$-Laplacian, and so it satisfies the equation $-\Delta_p u=\lambda_1(p)u^{p-1}$. Hence, by the Hopf Lemma, $\partial_\nu u\neq 0$ on $\partial\Omega$. Thus, in view of \eqref{eq:partialnu=0}, $u \equiv 0$.
\end{proof}

\begin{proof}[Proof of Theorem \ref{thm:small}] 
By contradiction, suppose there exists a sequence $(u_j) \subset W^{1,\A}_0(\Omega)\cap W^{2,\A}(\Omega)$ of solutions of \eqref{eq:P-poho} such that $\|u_j\|\to 0$ as $j\to \infty$. Let $\gamma>q^\ast$. Multiplying the equation in \eqref{eq:P-poho} by $u/\gamma$, we get for $f$ as in the hypothesis
\[
\frac{1}{\gamma}\int_\Omega(|\nabla u_j|^p+a(x)|\nabla u_j|^q)dx=\frac{1}{\gamma}\int_\Omega(c(x)|u_j|^r+\mu|u_j|^{p^\ast}+b(x)|u_j|^{q^\ast})dx.
\]
Subtracting this last identity from \eqref{eq:poho-RHS-LHS}, and using that $\Omega$ is starshaped and $a$ is radial and radially nondecreasing, this implies
\[
\begin{aligned}
&\left(\frac{1}{p^\ast}-\frac{1}{\gamma}\right)\int_\Omega|\nabla u_j|^pdx+\left(\frac{1}{q^\ast}-\frac{1}{\gamma}\right)\int_\Omega a(x)|\nabla u_j|^qdx\\
&\le \left(\frac{1}{r} -\frac{1}{\gamma}\right)\int_\Omega c(x)|u_j|^r dx+\mu \left(\frac{1}{p^\ast}-\frac{1}{\gamma}\right)\int_\Omega |u_j|^{p^\ast}dx+\left(\frac{1}{q^\ast}-\frac{1}{\gamma}\right)\int_\Omega b(x)|u_j|^{q^\ast}dx\\
&\le \left(\frac{1}{r} -\frac{1}{\gamma}\right)\int_\Omega c(x)|u_j|^r dx+\mu \left(\frac{1}{p^\ast}-\frac{1}{\gamma}\right)\frac{1}{S_p^{p^\ast/p}}(\rho_\A(\nabla u_j))^{p^\ast/p}
+\kappa\left(\frac{1}{q^\ast}-\frac{1}{\gamma}\right)(\rho_\A(\nabla u_j))^{q^\ast/q},
\end{aligned}
\]
where in the last steps we used \eqref{3} and \eqref{4}. Thus, the last chain of inequalities gives
\begin{equation}\label{eq:poho.ineq-small}
\begin{aligned}
\left(\frac{1}{q^\ast}-\frac{1}{\gamma}\right)\rho_\A(\nabla u_j)\le&\left(\frac{1}{r}-\frac{1}{\gamma}\right)\int_\Omega c(x)|u_j|^r dx+\mu \left(\frac{1}{p^\ast}-\frac{1}{\gamma}\right)\frac{1}{S_p^{p^\ast/p}}(\rho_\A(\nabla u_j))^{p^\ast/p}\\
&+\kappa\left(\frac{1}{q^\ast}-\frac{1}{\gamma}\right)(\rho_\A(\nabla u_j))^{q^\ast/q}.
\end{aligned}
\end{equation}  
Now, if $r\in(p,p^\ast]$, by the Sobolev embedding $W^{1,p}_0(\Omega)\hookrightarrow L^r(\Omega)$, we obtain
\[
\begin{aligned}
\left(\frac{1}{q^\ast}-\frac{1}{\gamma}\right)&\rho_\A(\nabla u_j)\le
\left(\frac{1}{r}-\frac{1}{\gamma}\right)c_\infty C_S\left(\int_\Omega |\nabla u_j|^p dx\right)^{r/p}\\
&+\mu \left(\frac{1}{p^\ast}-\frac{1}{\gamma}\right)\frac{1}{S_p^{p^\ast/p}}(\rho_\A(\nabla u_j))^{p^\ast/p}+\kappa\left(\frac{1}{q^\ast}-\frac{1}{\gamma}\right)(\rho_\A(\nabla u_j))^{q^\ast/q}\\
&\hspace{-.5cm}\le \left(\frac{1}{r}-\frac{1}{\gamma}\right)c_\infty C_S(\rho_\A(\nabla u_j))^{r/p}+\mu \left(\frac{1}{p^\ast}-\frac{1}{\gamma}\right)\frac{1}{S_p^{p^\ast/p}}(\rho_\A(\nabla u_j))^{p^\ast/p}\\
&+\kappa\left(\frac{1}{q^\ast}-\frac{1}{\gamma}\right)(\rho_\A(\nabla u_j))^{q^\ast/q}.
\end{aligned}
\]
Since by \eqref{18}, $\rho_\A(\nabla u_j)\to 0$ as well, this gives a contradiction, because all the exponents of $\rho_\A(\nabla u_j)$ on the right-hand side are larger than 1. Thus, $(i)$ is proved.
As for $(ii)$, setting $\C(x,t):=t^\sigma+c(x)t^r$, with $\sigma\in(q,p^\ast)$, by Proposition \ref{prop:subcritical}, we know that $W^{1,\A}(\Omega)\hookrightarrow L^\C(\Omega)$. Thus,
\[
\int_\Omega c(x)|u_j|^rdx\le \rho_\C(u_j)\le \max\{\|u_j\|_\C^{\sigma},\|u_j\|_\C^{r}\}\le C_S\max\{\|u_j\|^{\sigma},\|u_j\|^{r}\}=C_S \|u_j\|^{\sigma},
\]
where used that $\|u_j\|<1$ for $j$ large. Therefore, combining with \eqref{eq:poho.ineq-small} and \eqref{18}, we have
\[
\begin{aligned}
\left(\frac{1}{q^\ast}-\frac{1}{\gamma}\right)\|u_j\|^q\le & \left(\frac{1}{q^\ast}-\frac{1}{\gamma}\right)\rho_\A(\nabla u_j)
\le \left(\frac{1}{r}-\frac{1}{\gamma}\right)C_S \|u_j\|^{\sigma}\\
& + \mu \left(\frac{1}{p^\ast}-\frac{1}{\gamma}\right)\frac{1}{S_p^{p^\ast/p}}\|u_j\|^{p^\ast}+\kappa\left(\frac{1}{q^\ast}-\frac{1}{\gamma}\right)\|u_j\|^{pq^\ast/q},
\end{aligned}
\]
which contradicts $\|u_j\|\to 0$, because all the exponents of $\|u_j\|$ on the right-hand side are larger than $q$, in particular $pq^\ast/q>q$ is a consequence of $q/p \le N/(N-1)$. This concludes the proof of $(ii)$.
Finally, to prove $(iii)$, we first estimate
\begin{multline*}
\int_\Omega c(x)\, |u|^{r} dx \le c_\infty \int_{\supp(c)} |u|^{r} dx \le c_\infty C_S\left(\int_{\supp(c)} |\nabla u|^q\, dx\right)^{r/q}\\[7.5pt]
\le \frac{c_\infty C_S}{(a'_0)^{r/q}} \left(\int_{\supp(c)} a(x)\, |\nabla u|^q\, dx\right)^{r/q} \le
\frac{c_\infty C_S}{(a'_0)^{r/q}} (\rho_A(\nabla u))^{r/q},
\end{multline*}
in view of the embedding $W^{1,q}(\mathrm{supp}(c))\hookrightarrow L^r(\mathrm{supp}(c))$. The conclusion follows as for $(i)$.
\end{proof}

\section{Proofs of compactness results in the nonlocal case}\label{sec7}
In this section $J$ is the energy functional introduced in \eqref{eq:energyJ} for problem \eqref{10KCpt}.
\subsection{Some preliminary results}
We explicitly derive here some consequences on $h$ of the assumptions \ref{K1} and \ref{K2} made on $K$. Note that \ref{K1} implies that $K_\ell$ is nonnegative and hence nondecreasing in both $t_1$ and $t_2$ by \ref{K2}. We also note that 
\begin{equation} \label{9K}
K_\ell\big(\pnorm{\nabla u}^p,\pnorm[q,a]{\nabla u}^q\big) \ge \alpha\, \rho_\A(\nabla u)^\gamma \quad \forall \, u \in W^{1,\,\A}_0(\Omega)
\end{equation}
by \ref{K1}.
Furthermore, by the monotonicity of $h$, $H(t)\le th(t)$ for all $t\ge0$, thus \ref{K1}, together with the properties of $h$, implies that for every $t_1,\,t_2\ge 0$, 
\[
(t_1+t_2)h(t_1+t_2)\ge\left(\frac{t_1}{p}+\frac{t_2}{q}\right)h\left(\frac{t_1}{p}+\frac{t_2}{q}\right)\ge H\left(\frac{t_1}{p}+\frac{t_2}{q}\right)\ge\alpha(t_1+t_2)^\gamma, 
\]
that is 
\begin{equation}\label{19K}
h(t) \ge \alpha t^{\gamma - 1}\quad \forall \, t \ge 0.
\end{equation}
We observe that in turn this implies that $h$ can vanish only at zero, namely
\begin{equation}\label{h0}
h(t)>0 \quad\forall\,t>0.
\end{equation}
Moreover, by \eqref{19K} we also infer that $h$ is not constant, unless $\gamma = 1$. 
\\
Finally \ref{K1} gives also for every $t_1,\,t_2\ge 0$, 
\[
H\left(\frac{t_1}{p}+\frac{t_2}{q}\right) \ge \frac{1}{\ell^*}(t_1+t_2)h\left(\frac{t_1}{p}+\frac{t_2}{q}\right),
\]
so, if $t_1=0$, we get for every $t_2\ge 0$, $H(\frac{t_2}{q}) \ge \frac{q}{\ell^*}\,\frac{t_2}{q}\,h(\frac{t_2}{q})$, 
which means that 
\[
H(t)\ge \frac{q}{\ell^*}th(t)=\frac{q}{\ell^*}tH'(t)\quad \mbox{for all }t\ge 0.
\] 
Consequently, $H(t)\le H(1) t^{\ell^*/q}$ for all $t\ge 1$. This, together with \eqref{19K} and the monotonicity of $h$, implies
\begin{equation}\label{eq:growth}
\frac{\alpha}{\gamma}t^\gamma\le H(t)\le th(t)\le \frac{\ell^*}{q} H(t)\le \frac{\ell^*}{q} H(1)t^{\ell^*/q}\quad\forall\,t\ge 1.
\end{equation}
\\
We stress that, apart from the monotonicity in both $t_1$ and $t_2$, all the other properties listed above are consequences of condition \ref{K1} alone.

\begin{lemma} \label{Lemma 0K}
Let $h$ satisfy \eqref{h0}. 
If $\seq{u_j} \subset W^{1,\,\A}_0(\Omega)$ converges weakly to $u$ and $J'(u_j) \to 0$, then, up to a subsequence, $\nabla u_j\to\nabla u$ a.e. in $\Omega$.
\end{lemma}

\begin{proof} Arguing as in the proof of Lemma \ref{Lemma 0} up to the limit in \eqref{eq:conseq-weak}, we get
\begin{equation}\label{eq:limsupK}
\begin{aligned}
\limsup_{j\to\infty}&(I_{p,j}+I_{q,j})h(\mathcal E_\A(u_j))\\
&=\limsup_{j\to\infty}\Big\{J'(u_j)[T(u_j-u)]\\
&\hspace{2cm}+\int_\Omega\left(\mu|u_j|^{p^\ast-2}u_j+b(x)|u_j|^{q^\ast-2}u_j+f(x,u_j)\right)T(u_j-u)\,dx\Big\}\\
&=\limsup_{j\to\infty}\int_\Omega\left(\mu|u_j|^{p^\ast-2}u_j+b(x)|u_j|^{q^\ast-2}u_j+g(x,u_j)\right)T(u_j-u)\,dx.
\end{aligned}
\end{equation} 
Now, reasoning as in the last part of the proof of Lemma \ref{Lemma 0}, we get
\[
I_j:=\int_\Omega\left(\mu|u_j|^{p^\ast-2}u_j+b(x)|u_j|^{q^\ast-2}u_j+g(x,u_j)\right)T(u_j-u)\,dx\to 0 \quad\mbox{as }j\to\infty,
\]
which by \eqref{eq:limsupK}, \eqref{eq:Ippos}, and \eqref{eq:Iqpos} implies that $(I_{p,j}+I_{q,j})h(\mathcal E_\A(u_j))\to 0$ as $j\to\infty$.

We observe that since the sequence $(u_j)$ is bounded in $W^{1,\A}_0(\Omega)$, also $(\E_\A(u_j))\subset \mathbb R^+$ is bounded by \eqref{18}. Thus $(\E_\A(u_j))$ converges to a nonnegative number up to a subsequence.
Now, if $h(\mathcal E_\A(u_j))\to 0$ (case arising only for degenerate Kirchhoff functions), by the continuity of $h$ and the hypothesis \eqref{h0}, we get that $\mathcal E_\A(u_j)\to 0$, which in particular implies that $|\nabla u_j|_p\to 0$. In this case, up to a subsequence, $\nabla u_j\to 0$ a.e. in $\Omega$ and the proof is concluded. If $I_{p,j}+I_{q,j}\to 0$, in particular $I_{p,j}\to 0$, from which we infer that \eqref{eq:thesis} holds by \eqref{eq:conseq-weak}. 
\end{proof}

\subsection{Proofs of Theorem \ref{Theorem 1K}, Proposition \ref{Proposition 2}, and Theorem \ref{Theorem 2K}}
\begin{proof}[Proof of Theorem \ref{Theorem 1K}]
\ref{Theorem 1.i} We have
\begin{equation} \label{13K}
J(u_j) = H(\EA{u_j}) - \int_\Omega \left(\frac{\mu}{p^\ast}\, |u_j|^{p^\ast} + \frac{b(x)}{q^\ast}\, |u_j|^{q^\ast} + F(x,u_j)\right) dx = \beta + \o(1)
\end{equation}
and
\begin{equation} \label{14K}
J'(u_j)\, u_j = \rho_\A(\nabla u_j)\, h(\EA{u_j}) - \int_\Omega \left(\mu\, |u_j|^{p^\ast} + b(x)\, |u_j|^{q^\ast} + u_j\, f(x,u_j)\right) dx = \o(\norm{u_j}).
\end{equation}
$\bullet$ \textit{Case $\ell=p$}: 
Dividing \eqref{14K} by $p^\ast$, subtracting from \eqref{13K}, and combining with \eqref{8K} and \ref{11K} gives
\begin{equation} \label{15K}
K_p\big(\pnorm{\nabla u_j}^p,\pnorm[q,a]{\nabla u_j}^q\big) \le \vartheta \int_\Omega |u_j|^{\gamma p}\, dx + \beta + \o(1 + \norm{u_j}).
\end{equation}
Since
\[
K_p\big(\pnorm{\nabla u_j}^p,\pnorm[q,a]{\nabla u_j}^q\big) \ge \alpha\, \rho_\A(\nabla u_j)^\gamma
\]
by \eqref{9K} and
\begin{equation}
\label{eq:conseq-eigen-p}
\lambda_1(\gamma,p) \int_\Omega |u_j|^{\gamma p}\, dx \le \left(\int_\Omega |\nabla u_j|^p\, dx\right)^\gamma \le \rho_\A(\nabla u_j)^\gamma
\end{equation}
by \eqref{10K}, this gives
\[
\left(\alpha - \frac{\vartheta}{\lambda_1(\gamma,p)}\right) \rho_\A(\nabla u_j)^\gamma \le \beta + \o(1 + \norm{u_j}).
\]
Since $\vartheta < \alpha\, \lambda_1(\gamma,p)$, this together with the first inequality in \eqref{18} implies that $(\norm{u_j})$ is bounded.

$\bullet$ \textit{Case $\ell=q$}: 
Dividing \eqref{14K} by $q^\ast$, subtracting from \eqref{13K}, and combining with \eqref{8K} and \ref{11q} gives
\begin{equation} \label{15q}
\begin{aligned}
K_q\big(\pnorm{\nabla u_j}^p,\pnorm[q,a]{\nabla u_j}^q\big) \le  \mu\left(\frac{1}{p^*}-\frac{1}{q^*}\right)&\int_\Omega |u_j|^{p^*}\, dx + \vartheta_1 \int_\Omega |u_j|^{r}\, dx\\ 
&+ \vartheta_2 \int_\Omega c(x)|u_j|^{\gamma q}\, dx + \beta + \o(1 + \norm{u_j}).
\end{aligned}
\end{equation}
Now, by \eqref{9K} we get
\[
K_q\big(\pnorm{\nabla u_j}^p,\pnorm[q,a]{\nabla u_j}^q\big) \ge \alpha\, \rho_\A(\nabla u_j)^\gamma \ge \alpha (|\nabla u_j|_p^{\gamma p}+|\nabla u_j|_{q,a}^{\gamma q}),
\]
by \eqref{10q}
\[
\lambda_1(\gamma,q,a) \int_\Omega c(x)|u_j|^{\gamma q}\, dx \le |\nabla u_j|^{\gamma q}_{q,a}.
\]
Moreover, by H\"older's and Sobolev's inequalities, $|u_j|_r\le |\Omega|^{1/r-1/p^*} |u_j|_{p^*}\le |\Omega|^{1/r-1/p^*}|\nabla u_j|_p/S_p^{1/p}$. Altogether, from \eqref{15q} we obtain
\begin{equation}
\label{eq:tohaveabsurd}
\begin{aligned}
\alpha (|\nabla u_j|_p^{\gamma p}+|\nabla u_j|_{q,a}^{\gamma q}) \le \frac{\mu}{S_p^{p^*/p}}\left(\frac{1}{p^*}-\frac{1}{q^*}\right)&|\nabla u_j|_p^{p^*} + 
\frac{\vartheta_1}{S_p^{r/p}}|\Omega|^{1-r/p^*}|\nabla u_j|^r_p \\
&+ \frac{\vartheta_2}{\lambda_1(\gamma,q,a)}|\nabla u_j|^{\gamma q}_{q,a}+ \beta + \o(1 + \norm{u_j}).
\end{aligned}
\end{equation}
Suppose by contradiction that $(\norm{u_j})$ is unbounded, then for a renamed subsequence $\|u_j\|\to\infty$, and so $|\nabla u_j|_p\to\infty$ or $|\nabla u_j|_{q,a}\to\infty$ by the first inequality of \eqref{18}. 
Therefore, eventually for $j$ we have
\[
\begin{aligned}
\alpha (|\nabla u_j|_p^{\gamma p}+|\nabla u_j|_{q,a}^{\gamma q}) \le &\left[\frac{\mu}{S_p^{p^*/p}}\left(\frac{1}{p^*}-\frac{1}{q^*}\right)+ \frac{\vartheta_1}{S_p^{r/p}}|\Omega|^{1-r/p^*}\right] \max\{|\nabla u_j|_p^{p^*}, |\nabla u_j|^r_p\}\\
&\hspace{4cm}+ \frac{\vartheta_2}{\lambda_1(\gamma,q,a)}|\nabla u_j|^{\gamma q}_{q,a}+ \beta + \o(1 + \norm{u_j}).
\end{aligned}
\]
Since $\gamma p > p^* > r$ and $\vartheta_2 < \alpha\, \lambda_1(\gamma,q,a)$, this gives a contradiction and allows to conclude that $(\norm{u_j})$ is bounded.

\ref{Theorem 1.ii} By the Brezis-Lieb lemma (see \cite{MR699419}),
\[
\rho_\A(\nabla u_j) = \rho_\A(\nabla (u_j - u)) + \rho_\A(\nabla u) + \o(1), \qquad \EA{u_j} = \EA{u_j - u} + \EA{u} + \o(1).
\]
Since
\begin{multline*}
\rho_\A(\nabla (u_j - u)) = \pnorm{\nabla (u_j - u)}^p + \pnorm[q,a]{\nabla (u_j - u)}^q \to t_1 + t_2,\\[7.5pt]
\EA{u_j - u} = \frac{1}{p}\, \pnorm{\nabla (u_j - u)}^p + \frac{1}{q}\, \pnorm[q,a]{\nabla (u_j - u)}^q \to \frac{t_1}{p} + \frac{t_2}{q},
\end{multline*}
this together with the continuity of $h$ gives
\[
\rho_\A(\nabla u_j) \to t_1 + t_2 + \rho_\A(\nabla u), \quad h(\EA{u_j}) \to h\bigg(\frac{t_1}{p} + \frac{t_2}{q} + \EA{u}\bigg).
\]
So passing to the limit in \eqref{14K} using \eqref{11} and the compact embedding $W^{1,\,\A}_0(\Omega) \hookrightarrow \hookrightarrow L^\C(\Omega)$ in Proposition \ref{prop:subcritical} gives
\begin{equation} \label{16K}
(t_1 + t_2 + \rho_\A(\nabla u))\, h\bigg(\frac{t_1}{p} + \frac{t_2}{q} + \EA{u}\bigg) = \int_\Omega \left(\mu\, |u_j|^{p^\ast} + b(x)\, |u_j|^{q^\ast} + u\, f(x,u)\right) dx + \o(1).
\end{equation}
On the other hand, passing to the limit in
\begin{multline*}
J'(u_j)\, u = \left(\int_\Omega |\nabla u_j|^{p-2}\, \nabla u_j \cdot \nabla u\, dx + \int_\Omega a(x)\, |\nabla u_j|^{q-2}\, \nabla u_j \cdot \nabla u\, dx\right) h(\EA{u_j})\\[7.5pt]
- \int_\Omega \left(\mu\, |u_j|^{p^\ast - 2}\, u_j + b(x)\, |u_j|^{q^\ast - 2}\, u_j + f(x,u_j)\right) u\, dx = \o(1)
\end{multline*}
gives 
\begin{equation} \label{17K}
\rho_\A(\nabla u)\, h\bigg(\frac{t_1}{p} + \frac{t_2}{q} + \EA{u}\bigg) = \int_\Omega \left(\mu\, |u|^{p^\ast} + b(x)\, |u|^{q^\ast} + u\, f(x,u)\right) dx.
\end{equation}
Indeed, by (i), $(\nabla u_j)$ is bounded in $[L^p(\Omega)]^N$, and so $(|\nabla u_j|^{p-2}\nabla u_j)$ is bounded in the reflexive Banach space $[L^{p'}(\Omega)]^N$. Thus, up to a subsequence $|\nabla u_j|^{p-2}\nabla u_j\rightharpoonup V$ in $[L^{p'}(\Omega)]^N$. On the other hand, by Lemma \ref{Lemma 0K}, $\nabla u_j\to  \nabla u$ a.e. in $\Omega$, hence by uniqueness of the a.e. limit, $|\nabla u_j|^{p-2}\nabla u_j\rightharpoonup |\nabla u|^{p-2}\nabla u$ in $[L^{p'}(\Omega)]^N$. This proves the convergence 
\[
\int_\Omega |\nabla u_j|^{p-2}\, \nabla u_j \cdot \nabla u\, dx\to \int_\Omega |\nabla u|^{p}\, dx.
\]
Arguing similarly for the other terms, we obtain \eqref{17K}.
Subtracting \eqref{17K} from \eqref{16K} and again using the Brezis-Lieb lemma gives
\[
(t_1 + t_2)\, h\bigg(\frac{t_1}{p} + \frac{t_2}{q} + \EA{u}\bigg) = \int_\Omega \left(\mu\, |u_j - u|^{p^\ast} + b(x)\, |u_j - u|^{q^\ast}\right) dx + \o(1).
\]
Combining this with the assumption that $h$ is nondecreasing and the inequalities \eqref{3} and \eqref{1000} gives for every $\varepsilon>0$
\begin{multline*}
(t_1 + t_2)\, h\bigg(\frac{t_1}{p} + \frac{t_2}{q}\bigg) \le \frac{\mu}{S_p^{p^\ast/p}} \left(\int_\Omega |\nabla (u_j - u)|^p\, dx\right)^{p^\ast/p}\\[7.5pt]
+ \frac{b_\infty}{(a_0\, S_q)^{q^\ast/q}} \left((1+\varepsilon)\int_\Omega a(x)\, |\nabla (u_j - u)|^q\, dx + C_\varepsilon\int_\Omega |u_j - u|^q\, dx\right)^{q^\ast/q} + \o(1).
\end{multline*}
By the compact embedding of $W^{1,\A}_0(\Omega)$ into $L^q(\Omega)$ (cf. Proposition \ref{prop:embeddings1}-($ii$)), passing to the limit as $j\to\infty$ gives
\[
(t_1 + t_2)\, h\bigg(\frac{t_1}{p} + \frac{t_2}{q}\bigg) \le \mu \left(\frac{t_1}{S_p}\right)^{p^\ast/p} + b_\infty \left((1+\varepsilon)\frac{t_2}{a_0\, S_q}\right)^{q^\ast/q}\quad\mbox{for every } \varepsilon>0
\]
so, letting $\varepsilon\to 0$, we conclude that $(t_1,t_2) \in S_{\mu,b_\infty}$.

\ref{Theorem 1.iii} Let $\beta < \beta_{\mu,b_\infty,\ell}^\ast$. We will show that $(t_1,t_2) = (0,0)$. Suppose $(t_1,t_2) \ne (0,0)$. Then $(t_1,t_2) \in S_{\mu,b_\infty} \setminus \set{(0,0)}$ and hence
\begin{equation} \label{18K}
K_\ell(t_1,t_2) \ge \beta_{\mu,b_\infty,\ell}^\ast
\end{equation}
by \eqref{12K}.

$\bullet$ \textit{Case $\ell=p$}: Since
\[
\pnorm{\nabla u_j}^p \to t_1 + \pnorm{\nabla u}^p, \qquad \pnorm[q,a]{\nabla u_j}^q \to t_2 + \pnorm[q,a]{\nabla u}^q
\]
by the Brezis-Lieb lemma and $\gamma p < p^\ast$, passing to the limit in \eqref{15K} gives
\[
K_p\big(t_1 + \pnorm{\nabla u}^p,t_2 + \pnorm[q,a]{\nabla u}^q\big) \le \vartheta \int_\Omega |u|^{\gamma p}\, dx + \beta.
\]
Since $\lambda_1(\gamma,p)\int_\Omega |u|^{\gamma p} \le \left(\int_\Omega |\nabla u|^p\right)^{\gamma} \le \rho_\A(\nabla u)^\gamma$,
\[
K_p\big(t_1 + \pnorm{\nabla u}^p,t_2 + \pnorm[q,a]{\nabla u}^q\big) \ge K_p(t_1,t_2) + K_p\big(\pnorm{\nabla u}^p,\pnorm[q,a]{\nabla u}^q\big) \ge \beta_{\mu,b_\infty,p}^\ast + \alpha\, \rho_\A(\nabla u)^\gamma
\]
by \ref{K2}, \eqref{18K}, \eqref{9K}, and $\vartheta < \alpha\lambda_1(\gamma,p)$, 
this gives $\beta \ge \beta_{\mu,b_\infty,p}^\ast$, contrary to assumption. So $(t_1,t_2) = (0,0)$.

$\bullet$ \textit{Case $\ell=q$}: Similarly to the previous case, when $\mu = \vartheta_2 = 0$, passing to the limit in \eqref{15q} gives
\[
K_q\big(t_1 + \pnorm{\nabla u}^p,t_2 + \pnorm[q,a]{\nabla u}^q\big) \le \vartheta_2 \int_\Omega c(x)|u|^{\gamma q}\, dx + \beta.
\]
Since
\[
K_q\big(t_1 + \pnorm{\nabla u}^p,t_2 + \pnorm[q,a]{\nabla u}^q\big) \ge K_q(t_1,t_2) + K_q\big(\pnorm{\nabla u}^p,\pnorm[q,a]{\nabla u}^q\big) \ge \beta_{0,b_\infty,q}^\ast + \alpha\, \rho_\A(\nabla u)^\gamma
\]
by \ref{K2}, \eqref{18K}, and \eqref{9K}, and
\[
\vartheta_2 \int_\Omega c(x)|u|^{\gamma q}\, dx \le \frac{\vartheta_2}{\lambda_1(\gamma,q,a)} \left(\int_\Omega a(x)|\nabla u|^q\, dx\right)^\gamma \le \alpha\, \rho_\A(\nabla u)^\gamma
\]
by \eqref{10q}, this gives $\beta \ge \beta_{0,b_\infty,q}^\ast$, contrary to assumption. So $(t_1,t_2) = (0,0)$.
\end{proof}
\smallskip

For all $(t_1,t_2) \in S_{\mu,b_\infty} \setminus \set{(0,0)}$,
\[
\mu \left(\frac{t_1}{S_p}\right)^{p^\ast/p} + b_\infty \left(\frac{t_2}{a_0\, S_q}\right)^{q^\ast/q} \ge \alpha\, (t_1 + t_2) \bigg(\frac{t_1}{p} + \frac{t_2}{q}\bigg)^{\gamma - 1} \ge \frac{\alpha}{q^{\gamma - 1}}\, (t_1 + t_2)^\gamma
\]
by \eqref{19K} and hence
\begin{equation} \label{20K}
\frac{\mu}{S_p^{p^\ast/p}}\, t_1^{p^\ast/p - \gamma} + \frac{b_\infty}{(a_0\, S_q)^{q^\ast/q}}\, t_2^{q^\ast/q - \gamma} \ge \frac{\alpha}{q^{\gamma - 1}}.
\end{equation}
We are now ready to prove Proposition \ref{Proposition 2} and Theorem \ref{Theorem 2K}.

\begin{proof}[Proof of Proposition \ref{Proposition 2}]
Since $S_{\mu,b_\infty} \ne \set{(0,0)}$, $\beta_{\mu,b_\infty,\ell}^\ast < + \infty$. Let $\seq{(t_{1j},t_{2j})} \subset S_{\mu,b_\infty} \setminus \set{(0,0)}$ be a minimizing sequence for $\beta_{\mu,b_\infty,\ell}^\ast$. Then $\seq{(t_{1j},t_{2j})}$ is bounded by \ref{K1} and hence converges to some $(t_1,t_2) \in S_{\mu,b_\infty}$ for a renamed subsequence. We have $\beta_{\mu,b_\infty,\ell}^\ast = K_\ell(t_1,t_2)$ by the continuity of $K_\ell$ and $(t_1,t_2) \ne (0,0)$ by \eqref{20K}.
\end{proof}

\begin{proof}[Proof of Theorem \ref{Theorem 2K}]
\ref{Theorem 2.i} We have $(t_1,0) \in S_{\mu,b_\infty}$ for all
\[
t_1 \ge \tilde{h}_p^{-1}\bigg(\frac{\mu}{S_p^{p^\ast/p}}\bigg)
\]
and hence $S_{\mu,b_\infty} \ne \set{(0,0)}$ by \eqref{21K}. So there exists $(t_1,t_2) \in S_{\mu,b_\infty} \setminus \set{(0,0)}$ such that
\begin{equation} \label{25K}
\beta_{\mu,b_\infty,p}^\ast = K_p(t_1,t_2) > 0
\end{equation}
by Proposition \ref{Proposition 2}. Since $h$ is nondecreasing and $(t_1,t_2) \in S_{\mu,b_\infty}$,
\[
t_1\, h\bigg(\frac{t_1}{p}\bigg) \le (t_1 + t_2)\, h\bigg(\frac{t_1}{p} + \frac{t_2}{q}\bigg) \le \mu \left(\frac{t_1}{S_p}\right)^{p^\ast/p} + b_\infty \left(\frac{t_2}{a_0\, S_q}\right)^{q^\ast/q}
\]
and hence
\begin{equation} \label{26K}
\tilde{h}_p(t_1) \le \frac{\mu}{S_p^{p^\ast/p}} + \frac{b_\infty}{(a_0\, S_q)^{q^\ast/q}}\, \frac{t_2^{q^\ast/q}}{t_1^{p^\ast/p}}.
\end{equation}
Since $S_{\mu,0} \subset S_{\mu,b_\infty}$, $\beta_{\mu,b_\infty,p}^\ast \le \beta^\ast_{\mu,0,p}$, so \eqref{25K} together with \ref{K1} implies that $t_2$ is bounded uniformly in $b_\infty$. Then \eqref{20K} implies that $t_1$ is bounded away from zero for all sufficiently small $b_\infty$. So \eqref{26K} gives
\[
\tilde{h}_p(t_1) \le \frac{\mu}{S_p^{p^\ast/p}} + C b_\infty
\]
for some constant $C > 0$ independent of $b_\infty$. Since $\tilde{h}_p$ is decreasing, this gives
\[
t_1 \ge \tilde{h}_p^{-1}\bigg(\frac{\mu}{S_p^{p^\ast/p}} + C b_\infty\bigg).
\]
Since $K_p$ is nondecreasing, this in turn gives
\[
K_p(t_1,t_2) \ge K_p\bigg(\tilde{h}_p^{-1}\bigg(\frac{\mu}{S_p^{p^\ast/p}} + C b_\infty\bigg),0\bigg) \to K_p\bigg(\tilde{h}_p^{-1}\bigg(\frac{\mu}{S_p^{p^\ast/p}}\bigg),0\bigg) \quad \text{as } b_\infty \to 0
\]
by the continuity of $\tilde{h}_p^{-1}$ and $K_p$. The desired conclusion follows from this and \eqref{25K}. 

\ref{Theorem 2.ii} We have $(0,t_2) \in S_{0,b_\infty}$ for all
\[
t_2\ge \tilde{h}_q^{-1}\bigg(\frac{b_\infty}{(a_0S_q)^{q^\ast/q}}\bigg)
\]
and hence $S_{0,b_\infty} \ne \set{(0,0)}$ by \eqref{21q}. So there exists $(t_1,t_2) \in S_{0,b_\infty} \setminus \set{(0,0)}$ such that
\begin{equation} \label{25q}
\beta_{0,b_\infty,q}^\ast = K_q(t_1,t_2) > 0
\end{equation}
by Proposition \ref{Proposition 2}. Now, if $t_1 = 0$, $t_2 \neq 0$. If $t_1 \neq 0$, then $t_1 h(t_1/p+t_2/q) > 0$ by \eqref{h0}. Thus, since $h$ is nondecreasing and $(t_1,t_2) \in S_{0,b_\infty}$,
\[
t_2\, h\bigg(\frac{t_2}{q}\bigg) < (t_1 + t_2)\, h\bigg(\frac{t_1}{p} + \frac{t_2}{q}\bigg) \le  b_\infty \left(\frac{t_2}{a_0\, S_q}\right)^{q^\ast/q},
\]
which implies that $t_2 \neq 0$ also in this case. Now, in both cases the following inequality holds 
\begin{equation}
\label{eq:ineqq}
t_2\, h\bigg(\frac{t_2}{q}\bigg) \le  b_\infty \left(\frac{t_2}{a_0\, S_q}\right)^{q^\ast/q}
\end{equation}
and $t_2 \neq 0$. Hence, dividing \eqref{eq:ineqq} by $t_2^{q^*/q}$, we obtain
\[
\tilde{h}_q(t_2) \le \frac{b_\infty}{(a_0\, S_q)^{q^\ast/q}}.
\]
Since $\tilde{h}_q$ is decreasing, this gives $t_2 \ge \tilde{h}_q^{-1}(b_\infty/(a_0S_q)^{q^\ast/q})$.
Since $K_q$ is nondecreasing, this in turn gives
\[
K_q(t_1,t_2) \ge K_q\bigg(0,\tilde{h}_q^{-1}\bigg(\frac{b_\infty}{(a_0S_q)^{q^\ast/q}}\bigg)\bigg).
\]
The desired conclusion follows from this and \eqref{25q}.
\end{proof}

\section{Proofs of existence results in the nonlocal case}\label{sec8}

In this section we prove Theorems \ref{Theorem 3} and \ref{Theorem 4} concerning the existence for solutions to \eqref{27K} and \eqref{27q}, respectively.

\subsection{Proof of Theorem \ref{Theorem 3}}

Throughout this subsection, $\ell=p$. 
The energy functional associated with problem \eqref{27K} is
\[
J(u) = H(\EA{u}) - \int_\Omega \left(\frac{\lambda}{\gamma p}\, |u|^{\gamma p} + \frac{\mu}{p^\ast}\, |u|^{p^\ast} + \frac{c(x)}{s}\, |u|^s + \frac{b(x)}{q^\ast}\, |u|^{q^\ast}\right) dx, \quad u \in W^{1,\,\A}_0(\Omega).
\]

\begin{lemma} \label{Lemma 1K}
Let \ref{K1} hold for $K_p$. If $\lambda < \alpha \gamma p\, \lambda_1(\gamma,p)$, there exists $\varrho > 0$ such that
\[
\inf_{\norm{u} = \varrho}\, J(u) > 0.
\]
\end{lemma}

\begin{proof}
By \eqref{8K} and \eqref{9K},
\[
H(\EA{u}) = K_p\big(\pnorm{\nabla u}^p,\pnorm[q,a]{\nabla u}^q\big) + \frac{1}{p^\ast}\, \rho_\A(\nabla u)\, h(\EA{u}) \ge \alpha\, \rho_\A(\nabla u)^\gamma,
\]
while
\[
\int_\Omega |u|^{\gamma p}\, dx \le \frac{1}{\lambda_1(\gamma,p)} \left(\int_\Omega |\nabla u|^p\, dx\right)^\gamma \le \frac{\rho_\A(\nabla u)^\gamma}{\lambda_1(\gamma,p)}
\]
by \eqref{10K}. Fix $r \in (\gamma q,p^\ast)$, let
\[
\C(x,t) = t^r + c(x)\, t^s, \quad (x,t) \in \Omega \times [0,\infty).
\]
By Proposition \ref{prop:subcritical},
\[
\norm[\C]{u} \le k \norm{u} \quad \forall\, u \in W^{1,\,\A}_0(\Omega)
\]
for some constant $k \ge 1$. Then for all $u \in W^{1,\,\A}_0(\Omega)$ with $\norm{u} \le 1/k$, we have $\norm[\C]{u} \le 1$ and $\norm{u} \le 1$, so
\[
\int_\Omega c(x)\, |u|^s\, dx \le \rho_\C(u) \le \norm[\C]{u}^r \le k^r \norm{u}^r \le k^r \rho_\A(\nabla u)^{r/q}
\]
by \eqref{18}. By \eqref{3} and \eqref{4},
\begin{gather*}
\int_\Omega |u|^{p^\ast} dx \le \frac{\rho_\A(\nabla u)^{p^\ast/p}}{S_p^{p^\ast/p}}, \qquad \int_\Omega b(x)\, |u|^{q^\ast} dx \le \kappa\, \rho_\A(\nabla u)^{q^\ast/q},
\end{gather*}
respectively. Thus, for all $u \in W^{1,\,\A}_0(\Omega)$ with $\norm{u} \le 1/k$,
\[
J(u) \ge \left(\alpha - \frac{\lambda}{\gamma p\, \lambda_1(\gamma,p)}\right) \rho_\A(\nabla u)^\gamma - \frac{k^r \rho_\A(\nabla u)^{r/q}}{s} - \frac{\mu\, \rho_\A(\nabla u)^{p^\ast/p}}{p^\ast S_p^{p^\ast/p}} - \frac{\kappa}{q^\ast} \rho_\A(\nabla u)^{q^\ast/q}.
\]
Since $\lambda < \alpha \gamma p\, \lambda_1(\gamma,p)$, $\gamma < r/q < p^\ast/p < q^\ast/q$, and $\norm{u}^q \le \rho_\A(\nabla u) \le \norm{u}^p$ by \eqref{18}, the desired conclusion follows.
\end{proof}

We take $x_0 = 0$ for the sake of simplicity. Let $\psi \in C^\infty_0(B_\rho(0))$ be a cut-off function such that $0 \le \psi \le 1$ and $\psi = 1$ on $B_{\rho/2}(0)$, and set
\[
u_\eps(x) := \frac{\psi(x)}{\left(\eps^{p/(p-1)} + |x|^{p/(p-1)}\right)^{(N-p)/p}}, \quad v_\eps(x) := \frac{u_\eps(x)}{\pnorm[p^\ast]{u_\eps}}, \quad \eps > 0.
\]
Then, by \eqref{35}, 
for some constant $\kappa > 0$,
\begin{equation} \label{35q}
\int_\Omega v_\eps^{\gamma p}\, dx \ge \begin{cases}
\kappa\, \eps^{N-\gamma(N-p)} & \text{if } \gamma > \frac{N(p - 1)}{(N - p)p}\\[7.5pt]
\kappa\, \eps^{N/p}\, |\log \eps| & \text{if } \gamma = \frac{N(p - 1)}{(N - p)p}\\[7.5pt]
\kappa\, \eps^{\gamma(N-p)/(p-1)} & \text{if } \gamma < \frac{N(p - 1)}{(N - p)p}
\end{cases}
\end{equation}
for all sufficiently small $\eps > 0$. Since $\supp(v_\eps) \subset B_\rho(0)$ and the assumption that $a = 0$ a.e.\! in $B_\rho(0)$ together with \eqref{2} and \eqref{eq:hp-on-c} imply that $b = c = 0$ a.e.\! in $B_\rho(0)$,
\begin{equation} \label{36K}
J(\tau v_\eps) = H\bigg(\frac{\tau^p}{p} \int_\Omega |\nabla v_\eps|^p\, dx\bigg) - \frac{\lambda \tau^{\gamma p}}{\gamma p} \int_\Omega v_\eps^{\gamma p}\, dx - \frac{\mu \tau^{p^\ast}}{p^\ast}, \quad \tau \ge 0.
\end{equation}

\begin{lemma} \label{Lemma 2}
Let \ref{K1} hold for $K_p$. For all $\eps > 0$
\[
J(\tau v_\eps) \to - \infty \quad \text{as } \tau \to + \infty.
\]
\end{lemma}

\begin{proof}
We have
\begin{equation} 
J(\tau v_\eps) \le H\bigg(\frac{\tau^p}{p} \int_\Omega |\nabla v_\eps|^p\, dx\bigg) - \frac{\mu \tau^{p^\ast}}{p^\ast}, \quad \tau \ge 0
\end{equation}
by \eqref{36K}. 
Thus, by the last inequality of \eqref{eq:growth} with $\ell=p$,
\[
J(\tau v_\eps) \le \frac{p^*}{q}	H(1)\frac{|\nabla v_\eps|_p^{pp^*/q}}{p^{p^*/q}}\tau^{pp^*/q} - \frac{\mu \tau^{p^\ast}}{p^\ast}\to -\infty \quad\mbox{as } \tau \to\infty,
\]
the desired conclusion follows.
\end{proof}

By Lemma \ref{Lemma 2} and \eqref{34}, there exists $\tau_0 > 0$ such that $J(\tau_0 v_\eps) < 0$ and $\norm{\tau_0 v_\eps} > \varrho$ for all $\eps > 0$. This together with Lemma \ref{Lemma 1K} shows that $J$ has the mountain pass geometry. Denote by
\[
\Gamma := \big\{\varphi \in C([0,\tau_0],W^{1,\,\A}_0(\Omega)) : \varphi(0) = 0,\, \varphi(\tau_0) = \tau_0 v_\eps\big\}
\]
the class of paths in $W^{1,\,\A}_0(\Omega)$ joining the origin to $\tau_0 v_\eps$ and let
\[
\beta := \inf_{\varphi \in \Gamma}\, \max_{\tau \in [0,\tau_0]}\, J(\varphi(\tau)) > 0
\]
be the mountain pass level. Since the segment $\varphi_0(\tau) = \tau v_\eps,\, \tau \in [0,\tau_0]$ belongs to $\Gamma$,
\begin{equation} \label{220K}
\beta \le \max_{\tau \in [0,\tau_0]}\, J(\varphi_0(\tau)) \le \max_{\tau \ge 0}\, J(\tau v_\eps).
\end{equation}

\begin{lemma} \label{Lemma 3K}
Suppose that the assumptions of Theorem \ref{Theorem 3} hold. In each of the two cases in Theorem \ref{Theorem 3},
\[
\max_{\tau \ge 0}\, J(\tau v_\eps) < \widetilde{K}_p\bigg(\tilde{h}_p^{-1}\bigg(\frac{\mu}{S_p^{p^\ast/p}}\bigg)\bigg)
\]
for all sufficiently small $\eps > 0$.
\end{lemma}

\begin{proof}
Suppose that the conclusion of the lemma is false. Then there are sequences $\seq{\eps_j}$ and $\seq{\tau_j}$, with $\eps_j, \tau_j > 0$ and $\eps_j \to 0$, such that
\begin{equation} \label{213}
J(\tau_j v_{\eps_j}) = H\bigg(\frac{\tau_j^p}{p} \int_\Omega |\nabla v_{\eps_j}|^p\, dx\bigg) - \frac{\lambda \tau_j^{\gamma p}}{\gamma p} \int_\Omega v_{\eps_j}^{\gamma p}\, dx - \frac{\mu \tau_j^{p^\ast}}{p^\ast} \ge \widetilde{K}_p\bigg(\tilde{h}_p^{-1}\bigg(\frac{\mu}{S_p^{p^\ast/p}}\bigg)\bigg)
\end{equation}
and
\begin{equation} \label{214K}
\tau_j\, J'(\tau_j v_{\eps_j})\, v_{\eps_j} = h\bigg(\frac{\tau_j^p}{p} \int_\Omega |\nabla v_{\eps_j}|^p\, dx\bigg)\, \tau_j^p \int_\Omega |\nabla v_{\eps_j}|^p\, dx - \lambda \tau_j^{\gamma p} \int_\Omega v_{\eps_j}^{\gamma p}\, dx - \mu \tau_j^{p^\ast} = 0.
\end{equation}
Let $t_j = \tau_j^p \int_\Omega |\nabla v_{\eps_j}|^p\, dx$. Then \eqref{213} and \eqref{214K} can be written as
\begin{equation} \label{216K}
H\bigg(\frac{t_j}{p}\bigg) \ge \widetilde{K}_p\bigg(\tilde{h}_p^{-1}\bigg(\frac{\mu}{S_p^{p^\ast/p}}\bigg)\bigg) + \frac{\mu t_j^{p^\ast/p}}{p^\ast \left(\dint_\Omega |\nabla v_{\eps_j}|^p\, dx\right)^{p^\ast/p}} + \frac{\lambda t_j^\gamma \dint_\Omega v_{\eps_j}^{\gamma p}\, dx}{\gamma p \left(\dint_\Omega |\nabla v_{\eps_j}|^p\, dx\right)^\gamma}
\end{equation}
and
\begin{equation} \label{217K}
t_j\, h\bigg(\frac{t_j}{p}\bigg) = \frac{\mu t_j^{p^\ast/p}}{\left(\dint_\Omega |\nabla v_{\eps_j}|^p\, dx\right)^{p^\ast/p}} + \frac{\lambda t_j^\gamma \dint_\Omega v_{\eps_j}^{\gamma p}\, dx}{\left(\dint_\Omega |\nabla v_{\eps_j}|^p\, dx\right)^\gamma},
\end{equation}
respectively.

By \eqref{216K} and \eqref{eq:growth}, 
\[
\frac{H(1)}{p^{p^*/q}}t_j^{p^*/q}\ge H\bigg(\frac{t_j}{p}\bigg) \ge \widetilde{K}_p\bigg(\tilde{h}_p^{-1}\bigg(\frac{\mu}{S_p^{p^\ast/p}}\bigg)\bigg) + \frac{\mu}{p^\ast \left(\dint_\Omega |\nabla v_{\eps_j}|^p\, dx\right)^{p^\ast/p}}t_j^{p^\ast/p}.
\]
Since $p^*/q<p^*/p$, this implies that $\seq{t_j}$ is bounded and hence converges to some $t_0 \ge 0$ for a renamed subsequence.
Passing to the limit in \eqref{216K} gives
\[
H\bigg(\frac{t_0}{p}\bigg) \ge \widetilde{K}_p\bigg(\tilde{h}_p^{-1}\bigg(\frac{\mu}{S_p^{p^\ast/p}}\bigg)\bigg) + \frac{\mu}{p^\ast} \left(\frac{t_0}{S_p}\right)^{p^\ast/p} > 0,
\]
so $t_0 > 0$. Dividing \eqref{217K} by $t_j^{p^\ast/p}$ gives
\begin{equation} \label{215K}
\tilde{h}_p(t_j) = \frac{\mu}{\left(\dint_\Omega |\nabla v_{\eps_j}|^p\, dx\right)^{p^\ast/p}} + \frac{\lambda \dint_\Omega v_{\eps_j}^{\gamma p}\, dx}{t_j^{p^\ast/p - \gamma} \left(\dint_\Omega |\nabla v_{\eps_j}|^p\, dx\right)^\gamma}.
\end{equation}
Now passing to the limit in \eqref{215K} gives
\begin{equation} \label{219K}
\tilde{h}_p(t_0) = \frac{\mu}{S_p^{p^\ast/p}},
\end{equation}
so
\begin{equation} \label{218K}
t_0 = \tilde{h}_p^{-1}\bigg(\frac{\mu}{S_p^{p^\ast/p}}\bigg).
\end{equation}

Combining \eqref{215K} with \eqref{34}, \eqref{35q}, and \eqref{219K} gives as $j\to\infty$
\[
\tilde{h}_p(t_j) \ge \tilde{h}_p(t_0) + \begin{cases}
\sigma_j\, \eps_j^{N-\gamma(N-p)} + \O(\eps_j^{(N-p)/(p-1)}) & \text{if } \gamma > \frac{N(p - 1)}{(N - p)p}\\[7.5pt]
\sigma_j\, \eps_j^{N/p}\, |\log \eps_j| + \O(\eps_j^{(N-p)/(p-1)}) & \text{if } \gamma = \frac{N(p - 1)}{(N - p)p},
\end{cases}
\]
where
\[
\sigma_j = \frac{\kappa \lambda}{t_j^{p^\ast/p - \gamma} \left(\dint_\Omega |\nabla v_{\eps_j}|^p\, dx\right)^\gamma} \to \frac{\kappa \lambda}{t_0^{p^\ast/p - \gamma} S_p^\gamma} > 0.
\]
It follows from this that in each of the two cases in the lemma and for all sufficiently large $j$, $\tilde{h}_p(t_j) \ge \tilde{h}_p(t_0)$. Since $\tilde{h}_p$ is decreasing, this gives $t_j \le t_0$. Since $\widetilde{K}_p$ is nondecreasing, this in turn gives $\widetilde{K}_p(t_j) \le \widetilde{K}_p(t_0)$. However, dividing \eqref{217K} by $p^\ast$, subtracting from \eqref{216K}, and using \eqref{218K} gives
\[
\widetilde{K}_p(t_j) \ge \widetilde{K}_p(t_0) + \frac{\left(\dfrac{1}{\gamma p} - \dfrac{1}{p^\ast}\right) \lambda t_j^\gamma \dint_\Omega v_{\eps_j}^{\gamma p}\, dx}{\left(\dint_\Omega |\nabla v_{\eps_j}|^p\, dx\right)^\gamma} > \widetilde{K}_p(t_0).
\]
This contradiction completes the proof.
\end{proof}

We can now conclude the proof of Theorem \ref{Theorem 3}. By \eqref{220K} and Lemma \ref{Lemma 3K},
\[
\beta < \widetilde{K}_p\bigg(\tilde{h}_p^{-1}\bigg(\frac{\mu}{S_p^{p^\ast/p}}\bigg)\bigg).
\]
Then for all sufficiently small $b_\infty \ge 0$,
\[
\beta < \beta_{\mu,b_\infty,p}^\ast
\]
by Theorem \ref{Theorem 2K} ($i$) and hence $J$ satisfies the \PS{\beta} condition by Corollary \ref{Corollary 1K} \ref{Corollary 1.i}. So $\beta > 0$ is a critical level of $J$ by the mountain pass theorem.

\begin{remark} The previous proof provides in particular the following estimate of the energy of the mountain pass solution $u$:
\[
0 < J(u) < \widetilde{K}_p\bigg(\tilde{h}_p^{-1}\bigg(\frac{\mu}{S_p^{p^\ast/p}}\bigg)\bigg).
\]
\end{remark}

\subsection{Proof of Theorem \ref{Theorem 4}}
Throughout this subsection, $\ell=q$ and $h$ satisfies also condition \eqref{hp:h}. The energy functional associated with problem \eqref{27q} is
\[
J(u) = H(\EA{u}) - \int_\Omega \left(\eta\frac{c(x)}{\gamma q}\, |u|^{\gamma q} + \frac{b(x)}{q^\ast}\, |u|^{q^\ast}\right) dx, \quad u \in W^{1,\,\A}_0(\Omega).
\]

\begin{lemma} \label{Lemma 1q}
Let \ref{K1} hold for $K_q$. If $\eta < \alpha \gamma q\, \lambda_1(\gamma,q,a)$, there exists $\varrho > 0$ such that
\[
\inf_{\norm{u} = \varrho}\, J(u) > 0.
\]
\end{lemma}

\begin{proof}
By \eqref{8K} and \eqref{9K},
\[
H(\EA{u}) = K_q\big(\pnorm{\nabla u}^p,\pnorm[q,a]{\nabla u}^q\big) + \frac{1}{q^\ast}\, \rho_\A(\nabla u)\, h(\EA{u}) \ge \alpha\, \rho_\A(\nabla u)^\gamma,
\]
while
\[
\int_\Omega c(x)|u|^{\gamma q}\, dx \le \frac{1}{\lambda_1(\gamma,q,a)} \left(\int_\Omega a(x)|\nabla u|^q\, dx\right)^\gamma \le \frac{\rho_\A(\nabla u)^\gamma}{\lambda_1(\gamma,q,a)}
\]
by \eqref{10q}. 
Thus, by \eqref{4}, for all $u \in W^{1,\,\A}_0(\Omega)$,
\[
J(u) \ge \left(\alpha - \frac{\eta}{\gamma q\, \lambda_1(\gamma,q,a)}\right) \rho_\A(\nabla u)^\gamma - \frac{\kappa}{q^\ast}\rho_\A(\nabla u)^{q^\ast/q}.
\]
Since $\eta < \alpha \gamma q\, \lambda_1(\gamma,q,a)$, $\gamma < q^\ast/q$, and by \eqref{18}, the desired conclusion follows.
\end{proof}

We take $x_0 = 0$ without loss of generality. Let $\psi \in C^\infty_0(B_\rho(0))$ be a cut-off function such that $0 \le \psi \le 1$ and $\psi = 1$ on $B_{\rho/2}(0)$, and set
\[
u_{\eps,\delta}(x) := \frac{\psi(x/\delta)}{\left(\eps^{q/(q-1)} + |x|^{q/(q-1)}\right)^{(N-q)/q}}, \qquad v_{\eps,\delta}(x) := \frac{u_{\eps,\delta}(x)}{|u_{\eps,\delta}|_{q^\ast}}
\]
for $\eps > 0$ and $0 < \delta \le 1$.

\begin{lemma} \label{Lemma 2q}
Let the assumptions of Theorem \ref{Theorem 4} hold. For all sufficiently small $\eps > 0$ and $\eps/\delta>0$,
\[
J(\tau v_{\eps,\delta}) \to - \infty \quad \text{as } \tau \to + \infty.
\]
\end{lemma}

\begin{proof}
Since $\supp(v_{\eps,\delta}) \subset B_\rho(0)$, by the assumptions in \eqref{9q} we have 
\begin{equation} \label{28q}
J(\tau v_{\eps,\delta}) \le H\bigg(\frac{\tau^p}{p} |\nabla v_{\eps,\delta}|_p^p + \frac{a_0\tau^q}{q} |\nabla v_{\eps,\delta}|_q^q\bigg) - \frac{b_\infty \tau^{q^\ast}}{q^\ast}, \quad \tau \ge 0.
\end{equation}
Let $t = a_0\tau^q |\nabla v_{\eps,\delta}|_q^q$. Then $t \to +\infty$ as $\tau \to +\infty$, and \eqref{28q}, \eqref{215}, and \eqref{eq:conseq-h3} give
\[
J(\tau v_{\eps,\delta}) \le - \frac{t^{q^\ast/q}}{q^\ast} \left[b_\infty\left(a_0|\nabla v_{\eps,\delta}|_q^q\right)^{-q^\ast/q} - \frac{q^\ast H(t/q)}{t^{q^\ast/q}}\right]+\O(t^{p/q-1})\quad\mbox{as } \tau\to\infty.
\]
Since $|\nabla v_{\eps,\delta}|_q^q \to S_q$ as $\eps \to 0$ and $\eps/\delta\to 0$ by \eqref{214} and
\[
\lim_{t \to + \infty}\, \frac{q^\ast H(t/q)}{t^{q^\ast/q}} = \lim_{t \to + \infty}\, \frac{h(t/q)}{t^{q^\ast/q - 1}} = \lim_{t \to + \infty}\, \tilde{h}_q(t) < \frac{b_\infty}{(a_0 S_q)^{q^\ast/q}}
\]
by L’H\^{o}pital’s rule and \eqref{21q}, the desired conclusion follows.
\end{proof}

By Lemma \ref{Lemma 2q} and by \eqref{214}, there exists $\tau_0 > 0$ such that $J(\tau_0 v_{\eps,\delta}) < 0$ and $\norm{\tau_0 v_{\eps,\delta}} > \varrho$ for all sufficiently small $\eps > 0$ and $\eps/\delta>0$. This, together with Lemma \ref{Lemma 1q}, shows that, under \eqref{21q}, $J$ has the mountain pass geometry. Denote by
\[
\Gamma := \big\{\varphi \in C([0,\tau_0],W^{1,\,\A}_0(\Omega)) : \varphi(0) = 0,\, \varphi(\tau_0) = \tau_0 v_{\eps,\delta}\big\}
\]
the class of paths in $W^{1,\,\A}_0(\Omega)$ joining the origin to $\tau_0 v_{\eps,\delta}$ and let
\[
\beta := \inf_{\varphi \in \Gamma}\, \max_{\tau \in [0,\tau_0]}\, J(\varphi(\tau)) > 0
\]
be the mountain pass level. Since the segment $\varphi_0(\tau) = \tau v_{\eps,\delta},\, \tau \in [0,\tau_0]$ belongs to $\Gamma$,
\begin{equation} \label{220q}
\beta \le \max_{\tau \in [0,\tau_0]}\, J(\varphi_0(\tau)) \le \max_{\tau \ge 0}\, J(\tau v_{\eps,\delta}).
\end{equation}

For $\delta=\delta(\varepsilon)$, 
\begin{equation} \label{36q}
\begin{aligned}
J(\tau v_{\eps,\delta}) \le & H\bigg(\frac{\tau^p}{p} \int_\Omega |\nabla v_{\eps,\delta}|^p\, dx+a_0\frac{\tau^q}{q} \int_\Omega |\nabla v_{\eps,\delta}|^q\, dx\bigg) - b_\infty\frac{\tau^{q^\ast}}{q^\ast} - \eta c_0\frac{\tau^{\gamma q}}{\gamma q}\int_\Omega v_{\eps,\delta}^{\gamma q}\, dx\\
&=:\varphi_\eps(\tau).
\end{aligned}
\end{equation}
Consequently, in view of \eqref{220q} and \eqref{22q}, in order to prove Theorem \ref{Theorem 4} it is enough to show that 
\begin{equation}\label{eq:thesisK}
\max_{\tau\ge 0}\varphi_\eps(\tau)<\widetilde{K}_q\left(\tilde{h}_q^{-1}\left(\frac{b_\infty}{(a_0S_q)^{q^\ast/q}}\right)\right)
\end{equation}
for all sufficiently small $\eps>0$.

\begin{lemma} 
Under the hypotheses of Theorem \ref{Theorem 4}, if $\eps/\delta \to 0$ and the following limits hold
\begin{equation} \label{222K}
\lim_{\eps\to 0}\frac{|\nabla v_{\eps,\delta}|_p^p}{|v_{\eps,\delta}|_{\gamma q}^{\gamma q}} = 0, \qquad \lim_{\eps\to 0}\frac{(\eps/\delta)^{(N-q)/(q-1)}}{|v_{\eps,\delta}|_{\gamma q}^{\gamma q}} = 0,
\end{equation}
then there exist $\eps_0,\, \delta_0 > 0$ such that \eqref{eq:thesisK} holds true when $b_\infty > 0$.
\end{lemma}

\begin{proof}
Suppose \eqref{eq:thesisK} is false. Then there are sequences $\eps_j \to 0$ and $\tau_j > 0$ such that
\begin{equation} \label{2180K}
\varphi_{\eps_j}(\tau_j) = H\left(\frac{\tau_j^p}{p} |\nabla v_j|_p^p + \frac{a_0\, \tau_j^q}{q} |\nabla v_j|_q^q\right) - \eta \frac{c_0\, \tau_j^{\gamma q}}{{\gamma q}} |v_j|_{\gamma q}^{\gamma q} - \frac{b_\infty\, \tau_j^{q^\ast}}{q^\ast} \ge \widetilde{K}_q\left(\tilde{h}_q^{-1}\left(\frac{b_\infty}{(a_0S_q)^{q^\ast/q}}\right)\right)
\end{equation}
and
\begin{equation} \label{2190K}
\tau_j\, \varphi_{\eps_j}'(\tau_j) = (\tau_j^p |\nabla v_j|_p^p + a_0\, \tau_j^q |\nabla v_j|_q^q)h\left(\frac{\tau_j^p}{p} |\nabla v_j|_p^p + \frac{a_0\, \tau_j^q}{q} |\nabla v_j|_q^q\right) - \eta c_0\, \tau_j^{\gamma q} |v_j|_{\gamma q}^{\gamma q} - b_\infty\, \tau_j^{q^\ast} = 0,
\end{equation}
where $v_j = v_{\eps_j,\delta_j}$ and $\delta_j = \delta(\eps_j)$. By \eqref{214}--\eqref{216},
\begin{equation}\label{eq:convergences}
|\nabla v_j|_q^q \to S_q, \qquad |\nabla v_j|_p^p \to 0, \qquad |v_j|_{\gamma q}^{\gamma q} \to 0.
\end{equation}
So, dividing \eqref{2190K} by $\left[q\left(\frac{\tau_j^p}{p} |\nabla v_j|_p^p + \frac{a_0\, \tau_j^q}{q} |\nabla v_j|_q^q\right)\right]^{q*/q-1}$ we get 
\begin{equation}\label{eq:htildeinfty}
\tilde h_q\left(\frac{\tau_j^p}{p} |\nabla v_j|_p^p + \frac{a_0\, \tau_j^q}{q} |\nabla v_j|_q^q\right)\ge \frac{b_\infty \tau_j^{q^*}}{(\tau_j^p |\nabla v_j|_p^p + a_0\, \tau_j^q |\nabla v_j|_q^q)\left[q\left(\frac{\tau_j^p}{p} |\nabla v_j|_p^p + \frac{a_0\, \tau_j^q}{q} |\nabla v_j|_q^q\right)\right]^{q*/q-1}}
\end{equation}
Suppose by contradiction that $\seq{\tau_j}$ is not bounded, then $\tau_j\to+\infty$ up to a  subsequence. Along this subsequence, in view of \eqref{eq:convergences}, 
\[
\tau_j^p |\nabla v_j|_p^p + a_0\, \tau_j^q |\nabla v_j|_q^q\sim a_0\, \tau_j^q S_q,\quad \left[q\left(\frac{\tau_j^p}{p} |\nabla v_j|_p^p + \frac{a_0\, \tau_j^q}{q} |\nabla v_j|_q^q\right)\right]^{q*/q-1}\sim (a_0\, \tau_j^q S_q)^{q*/q-1},
\]
and so the right hand side of \eqref{eq:htildeinfty} goes to $b_\infty/(a_0S_q)^{q^*/q}$, contradicting the second inequality in \eqref{21q}. Therefore, $(\tau_j)$ is bounded and so it converges to some $\tau_0 > 0$ for a renamed subsequence ($\tau_0\neq 0$ by \eqref{2180K}). Passing to the limit in \eqref{2190K} gives
\begin{equation} 
a_0\, S_q\, \tau_0^q h\left(\frac{a_0\tau_0^q S_q}{q} \right) - b_\infty\, \tau_0^{q^\ast} = 0,
\end{equation}
that is $a_0\tau_0^q S_q = \tilde h_q^{-1}\left(\frac{b_\infty}{(a_0\, S_q)^{q^\ast/q}}\right)$.
Now, setting $t_{1j}:=\tau_j^p|\nabla v_j|_p^p$ and $t_{2j}:=a_0\tau_j^q|\nabla v_j|_q^q$, we get $t_{1j}=\O(|\nabla v_j|_p^p)\to 0$ and $t_{2j}\to a_0\tau_0^q S_q=:t_{2,0}>0$ as $j\to \infty$. 
Dividing \eqref{2190K} by $t_{2j}^{q^\ast/q}$ and taking into account \eqref{19K} and \eqref{hp:h} we obtain
\[
\begin{aligned}
\tilde h_q(t_{2j})+\O(|\nabla v_j|_p^p)&=\frac{t_{1j}+t_{2j}}{t_{2j}^{q^*/q}}\left[h\left(\frac{t_{2j}}{q}\right)+\O(|\nabla v_j|_p^p)\right]\\
&= \frac{t_{1j}+t_{2j}}{t_{2j}^{q^*/q}}h\left(\frac{t_{1j}}{p}+\frac{t_{2j}}{q}\right)=\frac{b_\infty}{(a_0|\nabla v_j|_q^q)^{q^*/q}}+\eta\frac{c_0\tau_j^{\gamma q}}{t_{2j}^{q^*/q}}|v_j|_{\gamma q}^{\gamma q}\\
&=\frac{b_\infty}{(a_0S_q)^{q^*/q}}+\O\big((\eps_j/\delta_j)^{\frac{N-q}{q-1}}\big)+\eta\frac{c_0\tau_j^{\gamma q}}{t_{2j}^{q^*/q}}|v_j|_{\gamma q}^{\gamma q}
\end{aligned}
\]
where in the last equality we used \eqref{214}. By the hypothesis \eqref{222K} and the fact that $\eta c_0\tau_j^{\gamma q}/t_{2j}^{q^*/q}\to \eta c_0\tau_0^{\gamma q}/t_{20}^{q^*/q}>0$, we infer that $\tilde h_q(t_{2j}) \ge b_\infty/(a_0S_q)^{q^*/q} = \tilde h_q(t_{20})$. Since $\tilde h_q$ is decreasing and $\widetilde{K}_q$ is nondecreasing, this implies that 
\begin{equation}\label{eq:to_contradict}
\widetilde K_q(t_{2j})\le \widetilde K_q(t_{20})
\end{equation}
is eventually true for $j$. On the other hand, dividing \eqref{2190K} by $q^\ast$, subtracting from \eqref{2180K}, and using \eqref{214} and \eqref{hp:h} gives
\[
\begin{aligned}
H\left(\frac{t_{2j}}{q}\right)-\frac{1}{q^*}t_{2j}h\left(\frac{t_{2j}}{q}\right)+\O(|\nabla v_j&|_p^p)\ge H\left(\frac{t_{1j}}{p}+\frac{t_{2j}}{q}\right)-\frac{1}{q^*}(t_{1j}+t_{2j})h\left(\frac{t_{1j}}{p}+\frac{t_{2j}}{q}\right)\\
&\ge \widetilde{K}_q\left(\tilde{h}_q^{-1}\left(\frac{b_\infty}{(a_0S_q)^{q^\ast/q}}\right)\right)+\eta c_0\left(\frac{1}{{\gamma q}}-\frac{1}{q^*}\right)\tau_j^{\gamma q}|v_j|_{\gamma q}^{\gamma q}.
\end{aligned}
\]
This in turn gives by \eqref{214} and \eqref{222K}
\[
\widetilde K_q(t_{2j}) = K(0,t_{2j}) > \widetilde K_q(\tilde{h}_q^{-1}(\tilde{h}_q(t_{20})))= \widetilde K_q(t_{20})\quad \text{eventually for }j,
\]
contradicting \eqref{eq:to_contradict} and concluding the proof. 
\end{proof}

The proof of Theorem \ref{Theorem 4} can now be concluded as in Theorem \ref{Theorem 2}.

\begin{remark} As for Theorem \ref{Theorem 3}, also the proof of Theorem \ref{Theorem 4} give the following estimates for the energy of the mountain pass solution $u$ found:
\[
0 < J(u) < \widetilde{K}_q\bigg(\tilde{h}_q^{-1}\bigg(\frac{b_\infty}{(a_0S_q)^{q^\ast/q}}\bigg)\bigg).
\]
\end{remark}

\def\cprime{$''$}

\end{document}